\documentclass[11pt]{amsart}
\usepackage[T1]{fontenc}
\usepackage{amsaddr}
\usepackage{amsxtra}%
\usepackage{amsfonts}%
\usepackage{amssymb}%
\usepackage{mathtools}
\usepackage{bm} 
\usepackage[margin=1.2in]{geometry}
\usepackage{color}
\usepackage{graphicx}
\usepackage{subfig}
\usepackage{cite}
\usepackage{color}
\usepackage{array}
\usepackage{booktabs}
\usepackage{caption}
\usepackage{microtype}
\usepackage{mathabx}
\usepackage{tikz, color}
\usetikzlibrary{matrix,arrows}

\usepackage{mleftright}
\mleftright
\usepackage{centernot}

\usepackage[hypertexnames=false]{hyperref}
\hypersetup{colorlinks,breaklinks,
             linkcolor=blue,urlcolor=blue,
             anchorcolor=blue,citecolor=blue}

\renewcommand{\L}{{\mathcal{GL}}}
\newcommand{\boldeta}{\bm{\eta}}
\newcommand{\Id}{\mathrm{Id}}

\DeclareMathOperator{\supp}{supp}

\DeclareMathOperator{\dist}{dist}

\DeclareMathOperator{\divv}{div}

\DeclareMathOperator{\Tr}{Tr}
\DeclareMathOperator{\esssup}{esssup}
\DeclareMathOperator*{\Glim}{\Gamma\text{-}\lim}

\renewcommand{\P}{{\mathbb P}}

\renewcommand{\bar}[1]{{\overline{#1}}}
\newcommand{\F}{{\mathcal F}}
\newcommand{\E}{{\mathcal E}}
\newcommand{\GE}{{\mathcal{GE}}}
\newcommand{\X}{{\mathcal X}}
\newcommand{\R}{\mathbb{R}}
\newcommand{\N}{\mathbb{N}}

\newcommand{\eps}{\varepsilon}
\renewcommand{\hat}[1]{\widehat{#1}}
\renewcommand{\tilde}[1]{\widetilde{#1}}

\def\Xint#1{\mathchoice
{\XXint\displaystyle\textstyle{#1}}%
{\XXint\textstyle\scriptstyle{#1}}%
{\XXint\scriptstyle\scriptscriptstyle{#1}}%
{\XXint\scriptscriptstyle\scriptscriptstyle{#1}}%
\!\int}
\def\XXint#1#2#3{{\setbox0=\hbox{$#1{#2#3}{\int}$ }
\vcenter{\hbox{$#2#3$ }}\kern-.6\wd0}}
\def\dashint{\Xint-}

\newcommand{\red}{\normalcolor}

\definecolor{mygreen}{rgb}{0.1,0.75,0.2}

\newcommand{\nc}{\normalcolor}
\newcommand{\te}{\textrm}
\newcommand{\tacka}{\,\cdot\,}

\newtheorem{theorem}{Theorem}[section]
\newtheorem{lemma}[theorem]{Lemma}
\newtheorem{corollary}[theorem]{Corollary}

\newtheorem{proposition}[theorem]{Proposition}
\theoremstyle{definition}
\newtheorem{remark}[theorem]{Remark}
\newtheorem{definition}[theorem]{Definition}

\graphicspath{{./figures/png8bit1024res/}}

\begin{document} 
\title[Properly-weighted graph Laplacian]{properly-weighted graph Laplacian for semi-supervised learning}

\author{Jeff Calder}
\address{Department of Mathematics, University of Minnesota}
\email{jcalder@umn.edu}

\author{Dejan Slep\v cev}
\address{Department of Mathematical Sciences, Carnegie Mellon University}
\email{slepcev@math.cmu.edu}

\begin{abstract}
The performance of traditional graph Laplacian methods for semi-supervised learning degrades substantially as the ratio of labeled to unlabeled data decreases, due to a degeneracy in the graph Laplacian. Several approaches have been proposed recently to address this, however we show that some of them remain ill-posed in the large-data limit.

 In this paper, we show a way to  correctly set the weights in Laplacian regularization so that the estimator remains well posed and stable in the large-sample limit. We prove that our semi-supervised learning algorithm converges, in the infinite sample size limit, to the smooth solution of a continuum variational problem that attains the labeled values continuously. Our method is fast and easy to implement.
\end{abstract}


\keywords{ 
semi-supervised learning, label propagation,  asymptotic consistency, PDEs on graphs, Gamma-convergence}

\subjclass{
49J55,   35J20, 35B65, 62G20,  65N12}

\maketitle

\section{Introduction}

For many applications of machine learning, such as medical image classification and speech recognition, labeling data requires human input and is expensive \cite{ssl}, while unlabeled data is relatively cheap. \emph{Semi-supervised} learning aims to exploit this dichotomy by utilizing the geometric or topological properties of the unlabeled data, in conjunction with the labeled data, to obtain better learning algorithms. A significant portion of the semi-supervised literature is on \emph{transductive} learning, whereby a function is learned only at the unlabeled points, and not as a parameterized function on an ambient space. In the transductive setting, graph based algorithms, such the graph Laplacian-based learning pioneered by \cite{zhu2003semi}, are widely used and have achieved great success\cite{zhou2005learning,zhou2004learning,zhou2004ranking,ando2007learning,he2004manifold,he2006generalized,wang2013multi,yang2013saliency,zhou2011iterated,xu2011efficient}.

Using graph Laplacians to propagate information from labeled to unlabeled points is one of the earliest and most popular approaches \cite{zhu2003semi}. The 
 constrained version of the graph Laplacian learning problem is to minimize over all $u : \X \to \R$
\begin{align}\label{eq:L2}
\begin{split}
GL(u) =&
 \sum_{x,y\in \X} w_{xy}(u(x) - u(y))^2 \\
\hspace*{-30pt} \te{subject to constraint } \quad  & u(x)=g(x) \te{ for  all }  x\in \Gamma
\end{split}
\end{align}
 where the data points $\X$ form the vertices of a graph with edge weights $w_{xy}$ and $\Gamma \subset \X$ are the labeled nodes with label function $g:\Gamma\to \R$. The minimizer $u$ of \eqref{eq:L2} is the learned function, which extends the given labels $g$ on $\Gamma$ to the remainder of the graph. In classification contexts, the values of $u$ are often rounded to the nearest label. The method amounts to minimizing a Dirichlet energy on the graph, subject to a Dirichlet condition $u=g$ on $\Gamma$. Minimizers $u$ are harmonic functions on the graph, and thus the problem can be view as harmonic extension.

\begin{figure}
\centering
\includegraphics[width=0.5\textwidth,clip = true, trim = 30 30 30 30]{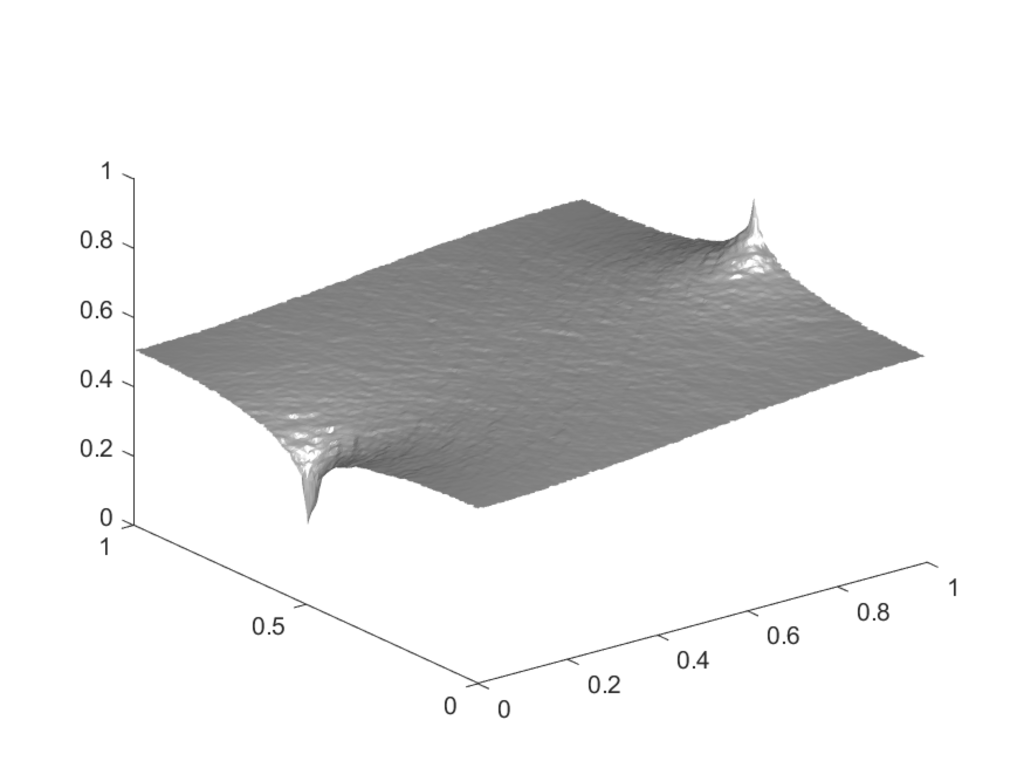}
\caption{Example of the degeneracy of graph Laplacian learning with few labels. The graph is a sequence of $n=10^5$ \emph{i.i.d.}~random variables drawn from the unit box $[0,1]^2$ in $\R^2$, and two labels are given $g(0,0.5)=0$ and $g(1,0.5)=1$.}
\label{fig:demo}
\end{figure}
It has been observed \cite{nadler2009semi,el2016asymptotic} that when the size of $\Gamma$ (the labeled points) is small, the performance of graph Laplacian learning algorithms degrades substantially. In practice, the learned function $u$ fails to attain the conditions $u=g$ on $\Gamma$ continuously, and degenerates into a constant label function that provides little information about the machine learning problem. Figure \ref{fig:demo} gives an example of this issue. There are several ways to explain this degeneracy. First, in the limit of infinite data, the variational problem \eqref{eq:L2} is consistent with the continuum Dirichlet problem
\begin{equation}\label{eq:DP}
\min_u \int_\Omega |\nabla u|^2 \, dx,
\end{equation}
subject to a boundary condition $u=g$ on $\Gamma\subset \Omega\subset \R^d$. If $\Gamma$ is finite this problem is ill-posed since the trace of an $H^1(\Omega)$ function at a point is not well-defined.
In particular, there are minimizing sequences for the constrained problem converging to a constant function outside of $\Gamma$ for which the Dirichlet energy converges to zero. In particular the minimum is not attained. 
 From another perspective, minimizers of the continuum Dirichlet problem \eqref{eq:DP} satisfy Laplace's equation $\Delta u=0$ with Dirichlet condition $u=g$ on $\Gamma$, and Laplace's equation is not well-posed without some boundary regularity (an exterior sphere condition), which does not hold for isolated points. In both cases, we are simply observing that the capacity of a point is zero in dimensions $d\geq 2$. 

Several methods have been proposed recently to address the degeneracy of Laplacian learning with few labels. In \cite{el2016asymptotic}, a class of $p$-Laplacian learning algorithms was proposed, which replace the exponent $2$ in \eqref{eq:L2} with $p>2$. The $p$-Laplacian models were considered previously for other applications~\cite{bridle2013p,alamgir2011phase,el2016asymptotic, ZhoSch05}, and the $p\to\infty$ case, which is called Lipschitz learning, was considered in~\cite{kyng2015algorithms,luxburg2004distance}. The idea behind the $p$-Laplacian models is that the continuum variational problem is now the $p$-Dirichlet problem
\begin{equation}\label{eq:Lp}
\min_u \int_\Omega |\nabla u|^p \, dx,
\end{equation}
and for $p>d$ the Sobolev embedding $W^{1,p}(\Omega)\hookrightarrow C^{0,\alpha}(\bar{\Omega})$ allows the assignment of boundary values at isolated points. The $p$-Laplacian models, including the $p=\infty$ version, were proven to be well-posed in the limit of infinite unlabeled data and finite labeled data precisely when $p>d$ in \cite{calderLip2017,calder2017game,SleTho17plap}. The disadvantage of $p$-Laplacian models is that  the nonlinearity renders them more computationally challenging to solve, compared with standard Laplacian regularization. Other approaches include higher order Laplacian regularization \cite{BLSZ17, DSST18, zhou11} and using a spectral cut-off \cite{BelNiy03SSL}.

The approach most closely related to our work is the \emph{weighted nonlocal Laplacian} of Shi, Osher, and Zhu~\cite{shi2017weighted}, which replaces the learning problem \eqref{eq:L2} with
\begin{equation}\label{eq:L2W}
\min_{u:\X\to \R} \sum_{x\in \X\setminus \Gamma}\sum_{y\in \X}w_{xy}(u(x)- u(y))^2 + \mu\sum_{x\in \Gamma}\sum_{y\in \X}w_{xy}( g(x)   - u(y))^2,
\end{equation}
where $\mu>0$ is selected as the ratio of unlabeled to labeled data. The method increases the weights of edges adjacent to labels, which encourages the label function to be flat near labels. The authors show in \cite{shi2017weighted} that the method produces superior results, compared to the standard graph Laplacian, for classification with very few labels. Furthermore, since the method is a standard graph Laplacian with a modified weight matrix, it has similar computational complexity to Laplacian learning, and is fast compared to the non-linear $p$-Laplace methods, for example.  However, as we prove in this paper, the weighted nonlocal Laplacian of  \cite{shi2017weighted} becomes ill-posed (degenerate) in the limit of infinite unlabeled and finite labeled data. This is a direct consequence of Corollary \ref{cor:gen_neg}.  Numerical simulations in Section \ref{sec:numerics} illustrate the way in which the method becomes degenerate. The issue is the same as for Laplacian learning, since the weights are modified only locally near label points and the size of this neighborhood shrinks to zero in the large sample size limit.

\subsubsection{Properly weighted Laplacian}
In this paper, we show how to properly weight  the graph Laplacian so that it remains well-posed in the limit of infinite unlabeled and finite labeled data.  Our method, roughly speaking, modifies the problem to one of the form:
\begin{align}\label{eq:L2Wour}
\begin{split}
\te{Minimize } & \sum_{x,y\in \X} \gamma(x) w_{xy}(u(x) - u(y))^2 \quad \te{ over } {u:\X\to \R}, \\ 
\hspace*{-30pt} \te{subject to constraint } \quad  & \;u(x)=g(x) \te{ for  all }  x\in \Gamma 
\end{split}
\end{align}
where $\gamma(x) = \text{dist}(x,\Gamma)^{-\alpha}$ and $\alpha> d-2$ (see Section \ref{sec:main} for precise definitions). Here, we are modifying the weights not just of edges connecting to points of $\Gamma$, but also in a  neighborhood of $\Gamma$. We show that this model is stable as the number unlabeled data points increases to infinity, under appropriate scaling of the graph construction. In particular we show that the minimizers of the graph problem above converge as the number of unlabeled data points increases to the minimizer of 
a  ``continuum learning problem''. We give the precise assumptions on the discrete model below and describe the continuum problem in Section \ref{sec:elliptic}. Here we give a brief explanation as to why $\alpha> d-2$ is the natural scaling for the weight.

To illustrate what is happening near a labeled point, consider $\Gamma=\{0\}$ and take the domain from which the points are sampled to be  the unit ball $\Omega=B(0,1)$ in $\R^d$.  The  continuum variational problem corresponding to \eqref{eq:L2Wour} involves minimizing 
\begin{equation}\label{eq:cont}
I[u] = \int_{B(0,1)} |x|^{-\alpha}|\nabla u|^2\, dx.
\end{equation}
The Euler-Lagrange equation satisfied by minimizers of $I$ is
\begin{equation}\label{eq:EL}
\divv\left( |x|^{-\alpha}\nabla u \right)=0.
\end{equation}
This equation has a radial solution $u(x) = |x|^{\alpha+2-d}$, which is continuous at $x=0$ when $\alpha > d-2$. This suggests the solutions will assume this radial profile near labels, and the model will be well-posed for $\alpha>d-2$. Furthermore  when $\alpha \geq d-1$ one can expect  the solution to be Lipschitz near labels, and for $\alpha \geq d$ it is should be differentiable at the labels. It is important to point out that the proper weighting changes the degenerate limiting continuum problem to one that is well-posed with ``boundary'' data at isolated points. 
\medskip

We now provide a precise description of the properly-weighted graph Laplacian.
 
\subsection{Model and definitions} \label{sec:main}

 Let $\Omega\subset\R^d$ be open and bounded with a Lipschitz boundary. Let $\Gamma\subset \Omega$ be a finite collection of points along with a given label function $g:\Gamma\to\R$. 
 \red Let $\mu$  be a probability measure on $\Omega$ with continuous  density $\rho$ which is bounded from above and below by positive constants. \nc 
 Let $x_1,x_2,\cdots,x_n$ be independent and identically distributed random variables with distribution $\mu$, and let
\[X_n := \left\{x_1,x_2,\cdots,x_n\right\},\]
and $\X_n:= X_n\cup \Gamma$. To define the edge weights we use a radial kernel $\eta$ with profile
$\boldeta:[0,\infty)\to [0,\infty)$ which is nonincreasing, continuous at $0$ and satisfies 
\begin{equation}\label{eq:eta}
\begin{cases}
\boldeta(t) \geq 1,&\text{if }0 \leq  t \leq 1\\
\boldeta(t) =0,&\text{if } t > 2.
\end{cases}
\end{equation}
All of the results we state can be extended to kernels which decay sufficiently fast, in particular the Gaussian. For  $\eps>0$ we define the rescaled kernel
\begin{equation}\label{eq:weight}
\eta_\eps(x-y)= \frac{1}{\eps^d} \boldeta\left( \frac{|x-y|}{\eps} \right).
\end{equation}

We now introduce the penalization of the gradient, which is heavier near labeled points. Let $R>0$ be the minimum distance between pairs of points in $\Gamma$:
\begin{equation}\label{eq:Rdef}
R = \min \{ |x-y| \;:\; x,y \in \Gamma, \;\; x  \neq y \}. 
\end{equation}
 For $r_0>0$ and $\alpha\geq 0$ let $\gamma\in C^\infty(\bar{\Omega}\setminus \Gamma)$ be any function satisfying $\gamma\geq 1$ on $\Omega$ and 
\begin{equation}\label{eq:g}
\gamma(x) = 1 + \left( \frac{r_0}{\dist(x,\Gamma)} \right)^{\alpha} \text{ whenever } \dist(x,\Gamma)\leq \frac{R}{4},
\end{equation}
where $\dist(x,\Gamma)$ denotes the Euclidean distance from $x$ to the closest point in $\Gamma$. For $\zeta>1$ we set
\begin{equation}\label{eq:gn}
\gamma_\zeta(x) = \min\{\gamma(x),\zeta\}.
\end{equation}
For $u\in L^2(\X_n)$ we define the energy
\begin{equation} \label{eq:energy}
\GE_{n,\eps, \zeta}(u) = \frac{1}{n^2\eps^{2}}\sum_{x,y\in \X_n}\gamma_\zeta(x)\eta_\eps(x-y)|u(x)-u(y)|^2.
\end{equation}
The Laplacian learning problem is to
\begin{equation}\label{eq:LapLearn}
\te{minimize }\;  \GE_{n,\eps, \zeta}(u) \; \te{ over } \left\{ u\in L^2(\X_n) \text{ and } u=g \text{ on }  \Gamma \right\}.
\end{equation}
We note that the unique minimizer $u\in L^2(\X_n)$ of \eqref{eq:LapLearn} satisfies the optimality condition
\begin{equation}\label{eq:optimality}
\left\{\begin{aligned}
\L_{n,\eps, \zeta} u(x) &= 0,&&\text{if }x\in X_n \; (=\X_n \setminus \Gamma) \\ 
u(x) &=g(x),&&\text{if }x\in \Gamma,
\end{aligned}\right.
\end{equation}
where $\L_{n,\eps, \zeta}:L^2(\X_n)\to L^2(\X_n)$ is the graph Laplacian, given by
\begin{equation}\label{eq:graphlap}
\L_{n,\eps, \zeta}u(x) = \frac{1}{2n\eps^2}\sum_{y\in \X_n}(\gamma_\zeta(x) + \gamma_\zeta(y))\eta_{\eps}(x-y)(u(y) - u(x)).
\end{equation}

Some remarks about the model are in order.
\begin{remark}
 When considering the discrete functional,  $\zeta$  depends on $n$ and diverges to infinity (sufficiently fast) as $n \to \infty$. The constant $r_0$ represents
the length scale of the crossover from the strong local penalization near $\Gamma$ to uniform far-field penalization. The introduction of $\zeta$ is needed since $\gamma(x)=\infty$ on $\Gamma$ and so using $\gamma$ directly would impose a hard constraint on neighbors of labeled points. While we can allow $\zeta=\infty$ in our model by interpreting products $\infty \cdot 0$ as $0$, we wanted to allow for a model with far less stringent constraints on agreement with the labeled points in the immediate vicinity of $\Gamma$. We note that the critical distance to $\Gamma$, when $\gamma_\zeta$ crosses over from $\gamma$ to $\zeta$ equals
\begin{equation} \label{eq_rzeta}
 r_\zeta = r_0 (\zeta-1)^{-1/\alpha} \quad \te{ provided that }  r_0 (\zeta-1)^{-1/\alpha} < \frac{R}{4}.
 \end{equation} 
\end{remark}
\begin{remark}
In practice, one can take \eqref{eq:g} to be the definition of the weights on the whole domain $\Omega$. We only need $\gamma$ to be smooth for a part of our analysis in Section \ref{sec:ep}. The issue is  that since the distance function $d(x, \Gamma)$ is not differentiable  (it is only Lipschitz on $\Omega \setminus \Gamma$ if $\Gamma$ has more than one point), $\gamma$ cannot be both smooth and globally given by \eqref{eq:g}. To elaborate,
 $\gamma$ appears as part of the diffusion coefficient in the limiting elliptic problem (see Eq.~\eqref{eq:pde}). The solutions have nicer regularity properties when we take $\gamma$ to be smooth, away from the labels. For the other results we only need that $\gamma$ is bounded from below by a positive number and has singularities, with a particular growth rate, near the points of $\Gamma$.
\end{remark}
\begin{remark}\label{rem:modified}
Instead of truncating $\gamma$ at the radius $r_\zeta$ to construct the weights $\gamma_\zeta$, we can take a possibly discontinuous model of the form
\begin{equation}\label{eq:gammazr}
\gamma_{\zeta,r}(x) = 
\begin{cases}
\gamma(x),&\text{if }\text{dist}(x,\Gamma) > r\\
\zeta,&\text{if }\text{dist}(x,\Gamma) \leq r.
\end{cases}
\end{equation}
This model is more general, since we can set $r=r_\zeta$ to recover \eqref{eq:gn}. Choosing $\zeta\gg 1 + (r_0/r)^\alpha$ places a larger penalty on the gradient in the inner region where $\text{dist}(x,\Gamma)\leq r$, compared to Eq.~\eqref{eq:gn}. This model is useful in the analysis of the graph based problem, and gives a sharper result for continuity at the labels (see Remark \ref{rem:HC}). In the limit as $n\to \infty$ we would take $\zeta\to \infty$ and $r \to 0$ with $r \geq r_\zeta$. 
\end{remark}
\begin{remark}
We remark that the discrete functional \eqref{eq:energy} can be rewritten as
\begin{equation}\label{eq:stillsym}
\GE_{n,\eps, \zeta}(u) = \frac{1}{2n^2\eps^{2}}\sum_{x,y\in \X_n}(\gamma_\zeta(x) + \gamma_\zeta(y))\eta_\eps(x-y)|u(x)-u(y)|^2,
\end{equation}
and so the problem has a symmetric weight matrix.
\end{remark}

\subsection{Outline}
\red
The continuum properly-weighted Dirichlet energy, which describes the asymptotic behavior of the properly-weighted graph Laplacian \eqref{eq:LapLearn} is presented in Section \ref{sec:elliptic} (equations \eqref{eq:continuum} and \eqref{eq:Econtinuum}). To show that the continuum problem is well posed and to establish its basic properties,  in Section \ref{sec:elliptic} we also  study properties of singularly weighted Sobolev spaces. In particular the Trace Theorem \ref{thm:trace} plays a key role in showing that the data can be imposed on a set of isolated points, which enables us to show the well-posedness in Theorem  \ref{thm:varexist}. The Euler-Lagrange equation of the variational problem is the elliptic problem we study in Section \ref{sec:ep}. In particular we show that solutions are $C^2$ away from the labels and H\"older continuous globaly.

In Section \ref{sec:convergence} we turn to asymptotics of the graph-based problems. 
We prove in Theorem \ref{thm:conv} that the solutions of the graph-based learning problem  \eqref{eq:LapLearn}, for the properly-weighted Laplacian, converge in the large sample size limit to the solution of a continuum variational problem \eqref{eq:continuum}-\eqref{eq:Econtinuum}.
We achieve this by showing the $\Gamma$-convergence of the discrete variational problems to the corresponding continuum problem.
We also prove a negative result, showing that the nonlocal weighted Laplacian \cite{shi2017weighted} is degenerate (ill-posed) in the large data limit (with fixed number of labeled points). 
In Section \ref{sec:near} we prove that solutions of the graph-based learning problem for the properly-weighted Laplacian attain their labeled values continuously with high probability (Theorem \ref{thm:holder}).  
In Section \ref{sec:numerics} we present the results of numerical simulations illustrating the 
estimators obtained by our method, and its performance in classification tasks on synthetic 
data and  in classifying handwritten digits from the MNIST dataset \cite{lecun1998gradient}. The classification problems on synthetic data contrast the stability of the properly-weighted Laplacian with the instability of the standard graph Laplacian and related methods. The MNIST experiments show superior performance of our method compared to the standard graph Laplacian, and similar performance to the  weighted Laplacian of \cite{shi2017weighted}. In the Appendix \ref{sec:Back} we recall some background results used and show and auxiliary technical result.
\nc

\subsection{Acknowledgements}

Calder was supported by NSF grant DMS:1713691. Slep\v{c}ev acknowledges the NSF support (grants DMS-1516677 and DMS-1814991). He is grateful to University of Minnesota, where this project started, for hospitality. He is also grateful to the Center for Nonlinear Analysis of CMU for its support. Both authors thank the referees for valuable suggestions. 

\section{Analysis of the continuum problem}
\label{sec:elliptic}

The continuum variational problem corresponding to the graph-based problem \eqref{eq:LapLearn} is
\begin{equation}\label{eq:continuum}
\te{minimize }  \E(u)\, \te{ over }  \, \left\{  u\in H^1_\gamma(\Omega)  \text{ and } u=g \text{ on }  \Gamma \right\},
\end{equation}
where $\E$ is given by
\begin{equation}\label{eq:Econtinuum}
\E(u) = \frac{1}{2}\int_\Omega \gamma |\nabla u|^2 \rho^2\, dx,
\end{equation}
\red $\rho$  is continuous and  bounded from above and below by positive constants, \nc
and the weighted Sobolev Space $H^1_\gamma(\Omega) $ is defined by \eqref{eq:H1gamma}. 
It follows from Lemma \ref{lem:lebesgue} that
  for $\gamma$ which grow near points of $\Gamma$ as fast as or faster than $\dist(x, \Gamma)^{-\alpha}$, the functions in $H^1_\gamma(\Omega)$ have a trace at $\Gamma$ (defined by \eqref{eq:trace_def}), which enables one to assign the condition $u=g$ on $\Gamma$ in \eqref{eq:continuum}.

The Euler-Lagrange equation satisfied by minimizers of \eqref{eq:Econtinuum} is the elliptic equation
\begin{equation}\label{eq:pde}
\left\{\begin{aligned}
-\divv(\gamma \rho^2 \nabla u) &= 0&&\text{in }\Omega\setminus \Gamma\\ 
u &=g&&\text{on }\Gamma\\
\frac{\partial u}{\partial \nu} &=0&&\text{on }\partial \Omega.
\end{aligned}\right.
\end{equation}
In this section we study the variational problem \eqref{eq:continuum} and the elliptic problem \eqref{eq:pde} rigorously.  The theory is nonstandard due to the boundary condition $u=g$ on $\Gamma$, since  $\Gamma$  is a collection of isolated points and does not satisfy an exterior sphere condition. As a consequence of this analysis, we prove in Section \ref{sec:near} that solutions of the graph-based problem are continuous at the labels.

Before studying  this problem, we need to perform a careful analysis of a particular weighted Sobolev space.

\subsection{Weighted Sobolev spaces}
\label{sec:weighted}

In this section we study the Sobolev space with norm weighted by $\gamma$. While there exists a rich literature on Weighted Sobolev Spaces, we did not find the precise results we need. Below we develop a self-contained, but brief, description of the spaces with particular weights of interest.

For $u\in H^1(\Omega)$  we define
\begin{equation}\label{eq:h1semi}
[u]^2_{H^1_\gamma(\Omega)} = \int_{\Omega}\gamma|\nabla u|^2\, dx,
\end{equation}
and
\begin{equation}\label{eq:h1}
\|u\|^2_{H^1_\gamma(\Omega)} = \|u\|_{L^2(\Omega)}^2 + [u]^2_{H^1_\gamma(\Omega)}. 
\end{equation}
We define
\begin{equation}\label{eq:H1gamma}
H^1_\gamma(\Omega) = \left\{ u \in H^1(\Omega) \, : \, \|u\|_{H^1_\gamma(\Omega)}< \infty \right\},
\end{equation} 
and endow $H^1_\gamma(\Omega)$ with the norm $\|u\|_{H^1_\gamma(\Omega)}$. We also denote by $H^1_{\gamma,0}(\Omega)$ the closure of $C^\infty_c(\bar{\Omega}\setminus \Gamma)$ in $H^1_\gamma(\Omega)$. The space $H^1_\gamma(\Omega)$ is the natural function space on which to pose the variational problem \eqref{eq:continuum}.

Throughout this section we let $B_r$ denote the open ball of radius $r>0$ centered at the origin in $\R^d$. Whenever we consider the space $H^1_\gamma(B_r)$, we will implicitly assume the choice of $\gamma(x) = |x|^{-\alpha}$. Hence
\begin{equation}\label{eq:h1semib}
[u]^2_{H^1_{\gamma}(B_r)} = \int_{B_r}|\nabla u|^2|x|^{-\alpha}\, dx.
\end{equation}
\red In all other occurrences, $\gamma$ is defined as in Section \ref{sec:main}, and in particular we always assume \eqref{eq:g} holds.\nc
We also use the notation $(u)_{x,r} = \dashint_{B(x,r)}u\, dx$ for the average of $u$ over the ball $B(x,r)$, and $(u)_r := (u)_{0,r}$. We also assume in this section that $\Omega$ has a Lipschitz boundary.

First, we study the trace of $H^1_\gamma(\Omega)$ functions on $\Gamma$. Before proving a general trace theorem, we require a preliminary lemma.
\begin{lemma}\label{lem:lebesgue}
Let $\alpha>d-2$ and $u\in H^1_\gamma(B_r)$. Then $x=0$ is a Lebesgue point~\cite{rudin2006real} for $u$, i.e., $u(0):=\lim_{\eps\to 0}(u)_\eps$ exists, and 
\begin{equation}\label{eq:lim}
|u(0) -(u)_\eps| \leq C \eps^{(\alpha + 2-d)/2}[u]_{H^1_\gamma(B_\eps)}
\end{equation}
for all $0 < \eps \leq r$.
\end{lemma}
\begin{proof}
We compute
\[\int_{B_\eps} |\nabla u|^2 \, dx \leq \int_{B_\eps}|\nabla u|^2|x|^{-\alpha}\eps^\alpha\, dx \leq [u]^2_{H^1_\gamma(B_\eps)}\eps^\alpha.\]
By the Poincar\'e inequality we have
\begin{equation}\label{eq:poin}
\dashint_{B_\eps}(u - (u)_\eps)^2\, dx \leq Cr^2\dashint_{B_\eps}|\nabla u|^2\, dx \leq C[u]^2_{H^1_\gamma(B_\eps)}\eps^{\alpha+2-d}.
\end{equation}
For $0 < s < t \leq r$ we apply \eqref{eq:poin} with $s$ and $t$ in place of $\eps$ to obtain
\begin{align}\label{eq:basicineq}
s^d((u)_s - (u)_t)^2 &\leq C\int_{B_s}((u)_s - (u)_t)^2 \, dx\\
&\leq C\int_{B_s}((u)_s - u)^2 \, dx + C\int_{B_s}((u)_t - u)^2\, dx\notag \\
&\leq C[u]^2_{H^1_\gamma(B_t)}(s^{\alpha + 2} + t^{\alpha + 2}).\notag
\end{align}
For $0 < q < \eps \leq r$ with $\eps \leq 4q$ we can set $s=q$ and $t=\eps$ above to obtain
\begin{equation}\label{eq:case1}
|(u)_q - (u)_\eps|^2 \leq C[u]^2_{H^1_\gamma(B_\eps)}\eps^{\alpha+2-d}.
\end{equation}

For $0 < q < \eps \leq r$ with $\eps > 4q$, let $k\in \N$ be the greatest integer smaller than $\log_2(\eps/q)$. Since $\eps> 2q$, we have $k\geq 1$. Choose $b>0$ so that $b^k = \eps/q$. Then
\[\log(b) = \frac{\log(\eps/q)}{k}\geq \frac{\log(\eps/q)}{\log_2(\eps/q)} \geq \log(2)\]
and
\[\log(b) \leq \frac{\log(\eps/q)}{\log_2(\eps/q)-1} = \frac{\log(2)\log_2(\eps/q)}{\log_2(\eps/q)-1}\leq 2\log(2),\]
since $\log_2(\eps/q)>2$. Therefore $2 \leq b \leq 4$. Let us set $\eps_j = \eps b^{-j}$ and $a_j = (u)_{\eps_j}$. Then $\eps_0 = \eps$ and $\eps_k = \eps b^{-k} = q$. Setting $t=\eps_j$ and $s = \eps_{j+1}$ in \eqref{eq:basicineq} yields
\[|a_j - a_{j+1}|^2 \leq C[u]^2_{H^1_\gamma(B_\eps)}(\eps_{j+1}^{\alpha+2-d} + \eps_j^{\alpha+2}\eps_{j+1}^{-d}) \leq C[u]^2_{H^1_\gamma(B_\eps)} b^{-j(\alpha+2-d)}\eps^{\alpha+2-d}.  \]
Therefore
\begin{equation}\label{eq:avg_bound}
|(u)_q - (u)_\eps|\leq \sum_{j=0}^{k-1}|a_{j+1}-a_j| \leq C[u]_{H^1_\gamma(B_\eps)}\eps^{(\alpha+2-d)/2}
\end{equation}
holds for all $k\geq 1$, where $C$ is independent of $u$, $\eps$ and $k$. 

In either case, we have established that
\begin{equation}\label{eq:cauchy}
|(u)_q - (u)_\eps|^2 \leq C[u]^2_{H^1_\gamma(B_\eps)} \eps^{\alpha + 2-d}
\end{equation}
holds for all $0 < q < \eps \leq r$. Thus, the sequence $\eps\mapsto (u)_\eps$ is Cauchy and converges to a real number as $\eps\to 0$. Sending $q\to 0$ in \eqref{eq:cauchy} completes the proof.
\end{proof}

By Lemma \ref{lem:lebesgue}, we can define the trace operator $\Tr:H^1_\gamma(\Omega) \to \R^\Gamma$ by
\begin{equation}\label{eq:trace_def}
\Tr[u](x) = \lim_{\eps\to 0}\dashint_{B(x,\eps)} u\, dx \ \ \ \ (x\in \Gamma).
\end{equation}
We endow $\R^\Gamma$ with the Euclidean norm.  We now prove our main trace theorem.
\begin{theorem}[Trace Theorem]\label{thm:trace}
\red Let $\alpha>d-2$ and assume $\gamma$ satisfies \eqref{eq:g} and $\Omega$ has a Lipschitz boundary. \nc Then the trace operator $\Tr:H^1_\gamma(\Omega) \to \R^\Gamma$ is bounded, and satisfies $\Tr[u](x) = u(x)$ whenever $u$ is continuous at $x\in \Gamma$. Furthermore, for every $u,v\in H^1_\gamma(\Omega)$ with $\|u-v\|_{L^2(\Omega)}^{2/(\alpha+2)}\leq R/2$ we have
\begin{equation}\label{eq:traceest}
|\Tr[u] - \Tr[v]|\leq C(1+[u]_{H^1_\gamma(\Omega)} + [v]_{H^1_\gamma(\Omega)})\|u-v\|_{L^2(\Omega)}^{1-d/(\alpha+2)}.
\end{equation}
\end{theorem}
\begin{proof}
By Lemma \ref{lem:lebesgue} each $x\in \Gamma$ is a Lebesgue point of $u$, and we have for $r \leq R/2$
\[|\Tr[u](x)|  \leq \dashint_{B(x,r)}u\,dx + C r^{(\alpha + 2-d)/2}[u]_{H^1_\gamma(\Omega)} \leq Cr^{-d/2}\|u\|_{L^2(\Omega)} + Cr^{(\alpha + 2-d)/2}\|u\|_{H^1_\gamma(\Omega)},\]
where $R>0$ is defined in \eqref{eq:Rdef}.
Fixing $r=R/2$ we have $|\Tr[u](x)|\leq C\|u\|_{H^1_\gamma(\Omega)}$, hence $\Tr:H^1_\gamma(\Omega)\to \R$ is bounded.

To prove \eqref{eq:traceest}, let $u,v \in H^{1}_\gamma(\Omega)$ and $x\in \Gamma$. For simplicity, we write $u(x)$ and $v(x)$ for $\Tr[u](x)$ and $\Tr[v](x)$, respectively. By Lemma \ref{lem:lebesgue} we have for $0 < \eps \leq R/2$
\begin{align*}
|u(x)-v(x)| &\leq \left|u(x) -(u)_{x,\eps}\right| + \left|v(x) - (v)_{x,\eps}\right| + \left|(u)_{x,\eps} - (v)_{x,\eps} \right|\\
&\leq C\eps^{(\alpha + 2-d)/2}([u]_{H^1_\gamma(B(x,\eps))} + [v]_{H^1_\gamma(B(x,\eps))}) +C\eps^{-d}\int_{B_\eps}|u-v|\, dx\\
&\leq C\eps^{(\alpha + 2-d)/2}([u]_{H^1_\gamma(\Omega)} + [v]_{H^1_\gamma(\Omega)}) +C\eps^{-d/2}\|u-v\|_{L^2(\Omega)}.
\end{align*}
Choosing $\eps =\|u-v\|_{L^2(\Omega)}^{2/(\alpha+2)}$, we obtain 
\[|u(x) - v(x)|\leq C(1+[u]_{H^1_\gamma(\Omega)} + [v]_{H^1_\gamma(\Omega)})\|u-v\|_{L^2(\Omega)}^{1-d/(\alpha+2)},\]
provided $\eps\leq R/2$. 
\end{proof}

We now examine the decay of the $L^2$ norm of trace zero functions.
\begin{lemma}\label{lem:decay}
Let $\alpha>d-2$ and $u\in H^1_\gamma(B_r)$ with $\Tr[u](0)=0$. Then 
\begin{equation}\label{eq:decay1}
\dashint_{\partial B_\eps}u^2 \, dS + \dashint_{B_\eps}u^2 \, dx \leq C \eps^{\alpha + 2-d}[u]^2_{H^1_\gamma(B_\eps)}
\end{equation}
for all $0 < \eps \leq r$.
\end{lemma}
\begin{proof}
Since $\Tr[u](0) =0$, Lemma \ref{lem:lebesgue} yields
\[\|(u)_\eps\|_{L^2(B_\eps)}\leq C\eps^{d/2}|(u)_{\eps}| \leq C\eps^{(\alpha+2)/2}[u]_{H^1_{\gamma}(B_\eps)}.\]
Recalling \eqref{eq:poin} from the proof of Lemma \ref{lem:lebesgue} we deduce
\[\|u - (u)_\eps\|_{L^2(B_\eps)} \leq C\eps^{(\alpha+2)/2}[u]_{H^1_\gamma(B_\eps)}.\]
Therefore
\[\dashint_{B_\eps}u^2 \, dx = C\eps^{-d}\|u\|^2_{L^2(B_\eps)}\leq C\eps^{\alpha+2-d}[u]^2_{H^1_\gamma(B_\eps)},\]
which establishes one part of \eqref{eq:decay1}.

For the other part, we use a standard trace estimate that we include for completeness. We have
\begin{align*}
\eps\int_{\partial B_\eps}u^2 \, dS&=\int_{B_\eps}\divv(x u^2)\, dx\\
&=\int_{B_\eps}du^2 + 2u\nabla u\cdot x \, dx\\
&\leq C\int_{B_\eps}u^2 \, dx + C\eps^2\int_{B_\eps}|\nabla u|^2\, dx\\
&\leq C\int_{B_\eps}u^2 \, dx + C\eps^{\alpha+2}[u]_{H^1_\gamma(B_\eps)}^2.
\end{align*}
Dividing both sides by $\eps^d$ we obtain
\[\dashint_{\partial B_\eps}u^2 \, dS\leq C\dashint_{B_\eps}u^2 \, dx + C\eps^{\alpha+2-d}[u]_{H^1_\gamma(B_\eps)}^2,\]
which completes the proof.
\end{proof}

We now show that trace zero functions can be approximated in $H^1_\gamma(\Omega)$ by smooth functions compactly supported away from $\Gamma$. 
\begin{theorem}[Trace zero functions]\label{thm:tracezero}
Let  $\alpha> d-2$ and \red assume $\gamma$ satisfies \eqref{eq:g} and $\Omega$ has a Lipschitz boundary. \nc Then $u\in H^1_{\gamma,0}(\Omega)$ if and only if $u\in H^1_\gamma(\Omega)$ and $\Tr[u]=0$.
\end{theorem}
\begin{proof}
If $u\in H^1_{\gamma,0}(\Omega)$, then there exists $u_k\in C^\infty_c(\bar{\Omega}\setminus \Gamma)$ so that $u_k\to u$ in $H^1_\gamma(\Omega)$. In particular, $u_k$ is uniformly bounded in $H^1_\gamma(\Omega)$. Thus, by Theorem \ref{thm:trace}, we have $\Tr[u](x)= \lim_{k\to \infty}u_k(x) = 0$ for each $x\in \Gamma$.

Conversely, let $u\in H^1_\gamma(\Omega)$ such that $\Tr[u]=0$. Without loss of generality, we may assume $\Omega=B_r$, $\Gamma=\{0\}$, and $\Tr[u](0)=0$. Choose a smooth nonincreasing function $\xi:[0,\infty)\to [0,1]$ such that $\xi(t)=1$ for $0 \leq t \leq 1$ and $\xi(t)=0$ for $t\geq 2$. For a positive integer $k\geq 1$ define $\xi_k(x) = \xi(k|x|)$ and $w_k = u(1-\xi_k)$. We compute
\begin{align}\label{eq:tracebound}
\int_{B_r}|\nabla w_k -\nabla u|^2|x|^{-\alpha}\, dx&=\int_{B_r}\left|\xi_k\nabla u + ku\xi'(k|x|)\frac{x}{|x|}\right|^2 |x|^{-\alpha}\, dx\\
&\leq C\int_{B_{2/k}}|\nabla u|^2|x|^{-\alpha}\, dx + Ck^2\int_{B_{2/k}\setminus B_{1/k}}u^2|x|^{-\alpha}\, dx\notag\\
&\leq C\int_{B_{2/k}}|\nabla u|^2|x|^{-\alpha}\, dx + Ck^{\alpha+2}\int_{B_{2/k}}u^2\, dx\notag \\
&\leq C\int_{B_{2/k}}|\nabla u|^2|x|^{-\alpha}\, dx\notag,
\end{align}
the last line following from Lemma \ref{lem:decay}. Therefore $w_k\to u$ in $H^1_\gamma(B_r)$ as $k\to \infty$.  To produce a smooth approximating sequence $u_k$, we simply mollify the sequence $w_k$.
\end{proof}
As a corollary, we can prove density of smooth functions that are locally constant near $\Gamma$.
\begin{corollary} \label{cor:dense}
For any $\alpha>0$ the set
\[ \mathcal S = \{ u \in C^\infty(\bar{\Omega}) \::\: (\exists s>0)(\forall x \in \Gamma)(\forall z \in B(0,1)) \;\; u(x+sz) = u(x) \} \]
is a dense subset of $H^1_\gamma(\Omega)$. 
\end{corollary}
\begin{proof}
We split the proof into two cases.

Case 1: $\alpha>d-2$. Let $u \in H^1_\gamma(\Omega)$. There exists $\psi\in \mathcal S$ such that $\Tr[\psi]=\Tr[u]$. Since $w:=u-\psi\in H^1_{\gamma,0}(\Omega)$, there exists by Theorem \ref{thm:tracezero} a sequence  $\varphi_k\in C^{\infty}_c(\bar{\Omega}\setminus \Gamma)$ such that $\varphi_k \to w$ as $k\to \infty$. We simply note that $\psi_k:=\varphi_k + \psi \in \mathcal S$ and $\psi_k \to u$ in $H^1_\gamma(\Omega)$ as $k\to \infty$.

Case 2: $\alpha \leq d-2$. In this case, $C^\infty(\bar{\Omega})$ is dense in $H^1_\gamma(\Omega)$ by a standard mollification argument, since the weighting kernel $|x|^{-\alpha}$ is integrable. Hence, for $u\in H^1_\gamma(\Omega)$ with $\alpha\leq d-2$ there exists $\varphi_k\in C^\infty(\bar{\Omega})$ such that $\varphi_k\to u$ in $H^1_\gamma(\Omega)$. Since $\varphi_k$ is smooth, we automatically have
\[\int_{\Omega}\dist(x,\Gamma)^{1-d}|\nabla \varphi_k|^2\, dx < \infty.\]
Thus, by case 1, there exists a sequence $\psi_{k,j}\in \mathcal S$ such that for each $k$, $\psi_{k,j}\to \varphi_k$ in $H^1_{\gamma}$ as $j\to \infty$, since $\alpha \leq d-2\leq d-1$. The proof is completed with a diagonal argument.
\end{proof}

Finally, we prove a Hardy-type inequality for trace zero functions in $H^1_\gamma(B_R)$.
\begin{theorem}[Hardy's inequality]\label{thm:poincare}
Let $\alpha> d-2$ and \red assume $\gamma$ satisfies \eqref{eq:g} and $\Omega$ has a Lipschitz boundary. \nc If $u\in H^1_\gamma(B_r)$ with $\Tr[u](0)=0$ then $\frac{u}{|x|^{(\alpha+2)/2}}\in L^2(B_r)$ and
\begin{equation}\label{eq:hardy}
\int_{B_r}\frac{u^2}{|x|^{\alpha+2}}\, dx \leq C[u]^2_{H^1_\gamma(B_r)}.
\end{equation}
\end{theorem}
\begin{proof}
By a change of variables we can reduce to the case of $r=1$. We first note that
\[\divv\left( \frac{x}{|x|^{\alpha+2}} \right) = -\frac{\alpha+2-d}{|x|^{\alpha+2}}\]
for $x\neq 0$. Thus, for $\eps>0$ we have
\begin{align*}
\int_{B_1\setminus B_\eps}\frac{u^2}{|x|^{\alpha+2}}\, dx&=- \red \frac{1}{\alpha+2-d} \nc \int_{B_1\setminus B_\eps}u^2\divv\left( \frac{x}{|x|^{\alpha+2}} \right)\, dx\\
&= \red \frac{1}{\alpha+2-d} \nc  \left[2\int_{B_1\setminus B_\eps}u\nabla u\cdot \frac{x}{|x|^{\alpha+2}}\, dx - \int_{\partial B_1} u^2\, dS + \frac{1}{\eps^{\alpha+1}}\int_{\partial B_\eps}u^2 \, dS\right]\\
&\leq C\int_{B_1\setminus B_\eps}|u||\nabla u||x|^{-\alpha-1}\, dx + C[u]^2_{H^1_\gamma(B_1)} + C\int_{B_\eps}|x|^{-\alpha}|\nabla u|^2\, dx,
\end{align*}
where the last line follows from Lemma \ref{lem:decay} and the assumption $\alpha>d-2$. Applying Cauchy's inequality to the first term and rearranging yields
\[\int_{B_1\setminus B_\eps}\frac{u^2}{|x|^{\alpha+2}}\, dx \leq  C[u]^2_{H^1_\gamma(B_1)} + C\int_{B_\eps}|x|^{-\alpha}|\nabla u|^2\, dx.\]
 Sending $\eps\to 0$ completes the proof.
\end{proof}

We now establish the well posedness of the continuum properly-weighted Laplacian learning problem. 
\begin{theorem} \label{thm:varexist}
\red Let $\alpha > d-2$, and assume $\rho$ is continuous and bounded above and below by positive constants, $\gamma$ satisfies \eqref{eq:g}, and $\Omega$ has a Lipschitz boundary. \nc Then the problem \eqref{eq:continuum} has a unique solution. 
\end{theorem} 
\begin{proof}
The existence follows by the direct method of the calculus of variations. Namely let $u_k$, $k=1,2, \dots$ be a minimizing sequence. By the Sobolev Embedding Theorem, $u_k$ has a subsequence which converges weakly in $H^1_\gamma(\Omega)$ and in $L^2(\Omega)$ towards $u \in H^1_\gamma(\Omega)$. Since $\E$ is convex, it is weakly lower-semicontinuous and thus
 $\E (u) \leq \liminf_{k \to \infty} \E(u_k)$. Furthermore note that \eqref{eq:traceest} implies that $\Tr(u_k)(z) \to \Tr(u)(z)$ for every $z \in \Gamma$. Thus $u = g$ on $\Gamma$. We conclude that $u$ is the desired minimizer. The uniqueness follows from convexity of $\E$, by a standard argument, which is recalled in the proof of Lemma \ref{lem:variational}.
\end{proof}

\subsection{Elliptic problem}
\label{sec:ep}

We now study the elliptic Euler-Lagrange equation \eqref{eq:pde}. We additionally assume in this section that $\Omega$ has a $C^{2,\alpha}$ boundary and $\rho\in C^{1,\sigma}(\bar{\Omega})$ for some $\sigma>0$. \red As before, we assume $\rho$ is bounded above and below by positive constants. \nc

\begin{definition}\label{def:weak}
We say that $u\in H^1_\gamma(\Omega)$ is a weak solution of \eqref{eq:pde} if
\begin{equation}\label{eq:weakform}
\int_{\Omega} \gamma \rho^2 \nabla u \cdot \nabla \varphi\, dx = 0
\end{equation}
for all $\varphi\in H^1_{\gamma,0}(\Omega)$ and $\Tr[u](x) = g(x)$ for all $x\in \Gamma$.
\end{definition}

We first need a preliminary proposition on barrier functions.
\begin{proposition}[Barrier]\label{prop:barrier}
Let $\alpha > d-2$ and fix any $0 < \beta < \alpha + 2-d$. Then there exists $c>0$ depending on $\alpha,\beta,\rho$ and $d$ such that $w(x) = |x|^\beta$ satisfies
\begin{equation}\label{eq:barrier}
\divv(\rho^2(1+|x|^{-\alpha})\nabla w) \leq -\frac{\beta}{2}(\alpha + 2 - \beta -d)\rho^2 |x|^{-(\alpha+2-\beta)}
\end{equation}
for all $0 < |x| \leq c$.
\end{proposition}
\begin{proof}
Since $\nabla w(x) = \beta|x|^{\beta-2}x$ we have
\begin{align*}
\divv(\rho^2(1+ |x|^{-\alpha})\nabla w) &=\beta\divv(\rho^2|x|^{\beta-2}x) + \beta\divv(\rho^2|x|^{\beta-\alpha-2}x)\\
&=2\beta|x|^{\beta-\alpha-2}(1 + |x|^{\alpha})\rho \nabla \rho \cdot x +\beta\rho^2\divv(|x|^{\beta-2}x) + \beta\rho^2\divv(|x|^{\beta-\alpha-2}x)\\
&=\beta\rho^2|x|^{\beta-\alpha-2}\left[ 2(1 + |x|^{\alpha})\nabla \log\rho \cdot x + (d+\beta-2)|x|^{\alpha} + d + \beta-\alpha-2  \right]\\
&\red \leq\beta\rho^2|x|^{\beta-\alpha-2}\left[ C(1 + |x|^{\alpha})|x| + (d+\beta-2)|x|^{\alpha} + d + \beta-\alpha-2  \right]
\end{align*}
\red The proof is completed by choosing $c>0$ small enough so that when $0 < |x| \leq c$ 
\[C(1 + |x|^{\alpha})|x| + (d+\beta-2)|x|^{\alpha}  \leq \tfrac{1}{2}(\alpha+2-\beta-d).\qedhere \]
\nc
\end{proof}

\begin{theorem}\label{thm:wellposedness}
\red Let $\alpha>d-2$. Assume $\gamma$ satisfies \eqref{eq:g}, $\Omega$ has a $C^{2,\alpha}$ boundary, and $\rho\in C^{1,\sigma}(\bar{\Omega})$ is bounded above and below by positive constants \nc. Then the elliptic equation \eqref{eq:pde} has a unique weak solution $u\in H^1_\gamma(\Omega)$. Furthermore, $u\in C(\bar{\Omega})\cap C^{2,\sigma}_{loc}(\bar{\Omega}\setminus \Gamma)$ and satisfies for every $0 < \beta < \alpha + 2-d$
\begin{equation}\label{eq:holderpde}
|u(x)-u(y)|\leq C(\beta)|x-y|^\beta \ \ \ \ (x\in \bar{\Omega},y\in \Gamma).
\end{equation}
\end{theorem}
\begin{proof}
For $\eps>0$ set 
\[\Omega_\eps := \Omega \setminus \bigcup_{y\in \Gamma}\bar{B(y,\eps)}\]
and let $u_{\eps}\in C^{2,\sigma}(\bar{\Omega_\eps})$ be the unique solution of the approximating problem
\begin{equation}\label{eq:eulereps}
\left\{\begin{aligned}
-\divv(\gamma \rho^2 \nabla u_{\eps}) &= 0&&\text{in }\Omega_\eps\\ 
u_{\eps} &=g(y)&&\text{on }\partial B(y,\eps) \text{ for all }y\in \Gamma\\
\frac{\partial u_{\eps}}{\partial \nu} &=0&&\text{on }\partial \Omega.
\end{aligned}\right.
\end{equation}
It is a classical result that $u_{\eps}$ is the unique solution of the variational problem
\begin{equation}\label{eq:vp}
\min \left\{ \int_{\Omega_\eps}\gamma \rho^2 |\nabla u|^2\, dx\, : \, u\in H^1(\Omega_\eps) \text{ and } \forall y\in \Gamma,\, u=g(y) \text{ on }  \partial B(y,\eps) \right\}.
\end{equation}
In particular, it follows that 
\begin{equation}\label{eq:energybound}
\sup_{\eps>0}\int_{\Omega_\eps}\gamma \rho^2|\nabla u_{\eps}|^2\, dx < \infty.
\end{equation}
By the maximum principle 
\begin{equation}\label{eq:maxprinc}
\min_{\Gamma}g \leq u_{\eps}\leq \max_{\Gamma}g.
\end{equation}
Let $y\in \Gamma$. By Proposition \ref{prop:barrier}, $w(x) := |x-y|^\beta$ satisfies
\[-\divv(\gamma \rho^2 \nabla w)  \geq \frac{\beta}{2}(\alpha+2-\beta-d)|x-y|^{-(\alpha+2-\beta)} > 0\]
for $0 < |x-y| \leq c$, where $c$ depends on $\alpha,\beta,\rho$, and $d$. Thus, another application of the maximum principle yields
\[u_{\eps}(x) \leq g(y) + C|x-y|^{\beta}\]
for all $x\in \bar{\Omega_\eps}$, where $C$ is independent of $\eps>0$. The other direction is similar, yielding
\begin{equation}\label{eq:hbarrier}
|u_{\eps}(x)- g(y)| \leq C|x-y|^{\beta} \ \ \text{ for all }x \in \bar{\Omega_\eps}.
\end{equation}
By the Schauder estimates \cite{gilbarg2015elliptic}, for each $\delta>0$ there exists a constant $C>0$, independent of $\eps$, such that
\[\|u_{\eps}\|_{C^{2,\sigma}(\bar{\Omega_\delta})} \leq C\]
for all $0 < \eps < \delta$. Therefore, there exists a subsequence $u_{\eps_k}$ and $u\in C^{2,\sigma}_{loc}(\bar{\Omega}\setminus \Gamma)$ such that $u_{\eps_k} \to u$ in $C^2_{loc}(\bar{\Omega}\setminus \Gamma)$. In particular, $u$ solves \eqref{eq:pde} classically and satisfies \eqref{eq:holderpde}, due to  \eqref{eq:hbarrier}. Thus $u\in C(\bar{\Omega})$ and $u=g$ on $\Gamma$. Finally, it follows from \eqref{eq:energybound} that $u\in H^1_\gamma(\Omega)$, and so $u$ is a weak solution of \eqref{eq:pde}, as per Definition \ref{def:weak}. Uniqueness of weak solutions follows by a standard energy method argument.
\end{proof}

\begin{lemma}\label{lem:variational}
\red Let $\alpha>d-2$. Assume $\gamma$ satisfies \eqref{eq:g}, $\Omega$ has a $C^{2,\alpha}$ boundary, and $\rho\in C^{1,\sigma}(\bar{\Omega})$ is bounded above and below by positive constants \nc. Then solution $u$ of the  variational problem \eqref{eq:continuum}  is the unique weak solution of the Euler-Lagrange equation \eqref{eq:pde}.
\end{lemma}
\begin{proof}
Let $u\in H^1_\gamma(\Omega)$ be the unique weak solution of \eqref{eq:pde}, and let $w\in H^1_\gamma(\Omega)$ with $\Tr[w]=\Tr[u]$. Then by the definition of weak solution
\[\int_\Omega \gamma \rho^2 \nabla u \cdot \nabla (u-w) \, dx = 0.\]
Therefore
\begin{align*}
\int_\Omega \gamma \rho^2 |\nabla u|^2 \,dx &=\int_\Omega \gamma \rho^2 \nabla u\cdot \nabla w \,dx\\
&=\frac{1}{2}\int_\Omega \gamma \rho^2 |\nabla u|^2 \,dx+\frac{1}{2}\int_\Omega \gamma \rho^2 |\nabla w|^2 \,dx - \frac{1}{2}\int_\Omega \gamma \rho^2 |\nabla u - \nabla  w|^2 \,dx,
\end{align*}
and so
\[\int_\Omega \gamma \rho^2 |\nabla u|^2 \,dx  = \int_\Omega \gamma \rho^2 |\nabla w|^2 \,dx - \int_\Omega \gamma \rho^2 |\nabla u - \nabla  w|^2 \,dx.\]
It follows that $u$ is the unique solution of the variational problem \eqref{eq:continuum}.
\end{proof}

\section{Discrete to continuum Convergence} \label{sec:convergence}

Throughout this section we consider $\Omega$, $\X_n$, and $\eta$ which satisfy the assumptions of Section \ref{sec:main}.
Let $\mu_n = \frac{1}{n} \sum_{i=1}^n \delta_{x_i}$ be the empirical measure of the sample. Let $d_\infty(\mu, \mu_n)$ be the $\infty$-transportation distance between $\mu$ and $\mu_n$, discussed in Appendix \ref{sec:ot}. 

We now state our main result.  In order to compare discrete and continuum minimizers we use the $TL^p$ topology introduced in \cite{GTS16}. We review the topology and its basic properties in Appendix \ref{sec:TLp}.

\begin{theorem} \label{thm:conv}
Let $\eps_n$ be a sequence of positive numbers converging to zero as $n \to \infty$ and such that 
$\eps_n \gg d_\infty(\mu, \mu_n)$.
Let $\zeta_n \in (1, \infty]$ be such that $\zeta_n \gg n \eps_n^2$.
Consider $\alpha>d-2$. 
Let $u_n$ be a sequence of minimizers of the problem \eqref{eq:LapLearn} for $\GE_{n,\eps_n, \zeta_n}$. Then \red almost surely \nc $(\mu_n, u_n)$
converges in $TL^2$ to $(\mu,u)$ where $u$ is the minimizer of  \eqref{eq:continuum}.
\end{theorem}
 Our approach to proving the theorem is via establishing the $\Gamma$-convergence of the discrete constrained functionals to the continuum ones. The overall approach to consistency of learning algorithms follows the one developed in \cite{GTS16, GTS17jmlr}. Ensuring that the discrete problem  induces enough regularity for one to be able to show that the label values are preserved in the limit at points of $\Gamma$ follows the general strategy of \cite{SleTho17plap}. However the problems and proofs are rather different. We remark that one can also use the PDE-based approach of \cite{calder2017game}, but this would require a slightly more restrictive  range on $\eps_n$. Nevertheless the PDE-based approach gives superior regularity of solutions which we exploit in Section \ref{sec:reg}.

\begin{proof}
Since, \red almost surely \nc  $\eps_n \gg d_\infty(\mu, \mu_n) \gtrsim n^{-1/d}$ it follows that $\zeta_n \to \infty$ as $n \to \infty$. 
\red We note that by discrete comparison principle $\|u_n\|_{L^\infty(\mu_n)} \leq \max_{\Gamma} \|g\|$. \nc
By Lemma \ref{lem:Gamma1}, the discrete energy $\GE_{n,\eps_n, \zeta_n}$ $\Gamma$-converges to $\theta_\eta \E$ and \red the sequence $\{(\mu_n, u_n)\}_{n=1,2, \dots}$ is precompact in $TL^2$. Therefore $(\mu_n, u_n)$ converges along a subsequence in $TL^2$ metric to $(\tilde \mu, u)$. Since $\mu_n$ converges to $\mu$ in Wasserstein metric, $\tilde \mu = \mu$. The fact that $u$ is the minimizer of \eqref{eq:continuum} now follows directly from $\Gamma$-convergence of Proposition \ref{prop:Gamma} below. Consequently, the fact that the whole sequence $u_n$ converges to $u$  follows from the uniqueness of the minimizer of \eqref{eq:continuum}. 
\end{proof}

\red
\begin{remark}
While above we address only algebraically growing weights $\gamma$ (see \eqref{eq:g})  it is straightforward to modify the proofs to show that if $\gamma$ grows faster than algebraically at labeled points (say $\gamma(x) = \exp(1/\dist(x,\Gamma)$) the conclusion of the theorem hold (in any dimension $d \geq 2$). 
\end{remark}

\begin{remark}
In this paper we assume that the data measure is supported on the set $\Omega$ of full dimension. There are no substantial obstacles in extending the results to the manifold setting where the data are sampled from a measure which is supported on a smooth submanifold of $\R^d$. One would only need to adjust the statements using manifold analogues of the weighted Dirichlet energy and the Laplacian. The convergence of graph Laplacian in the manifold setting has already been established in the standard setting \cite{GGHS}. In the manifold setting the dimension $d$ in the results above should be replaced by the dimension of the data manifold. 
\end{remark}
\nc

\begin{proposition} \label{prop:Gamma}
Let $\eps_n$ be a sequence of positive numbers converging to zero as $n \to \infty$ and such that 
$\eps_n \gg d_\infty(\mu, \mu_n)$.
Let $\zeta_n \in (1, \infty]$ be such that $\zeta_n \gg n \eps_n^2$ and \red $\zeta_n \gg \eps_n^{2-d}$ if $d >2$ and $\zeta_n \gg -\ln \eps_n$ if $d=2$\nc.
Let $\alpha>d-2$. 
Then the constrained properly-weighted graph Dirichlet energy, defined on $TL^2(\Omega)$ by 
\[ \GE^{con}_{n,\eps_n, \zeta_n}(\tilde \mu_n, u_n) = 
  \begin{cases}
    \GE_{n,\eps_n, \zeta_n}(u_n) \quad & \te{if } \tilde \mu_n = \mu_n \te{ and } u_n = g \te{ on } \Gamma \\
    \infty & \te{else}  
  \end{cases}
\]
$\Gamma$-converges \red almost surely \nc in $TL^2$ to the constrained continuum properly-weighted Dirichlet energy
\[ \theta_\eta \E^{con}(\tilde \mu, u) = 
  \begin{cases}
    \theta_\eta \E(u)\quad & \te{if } \tilde \mu = \mu, \; u \in H^1_\gamma(\Omega) \te{ and } u = g \te{ on } \Gamma \\
    \infty & \te{else},
  \end{cases}
\]
where the value of $u$ on $\Gamma$ is considered in the sense of the trace and 
\[ \theta_\eta = \frac{1}{d} \int_{\R^d} \eta(z) |z|^2 dz. \]
\end{proposition}
\red The proof of the $\Gamma$ convergence of the unconstrained functionals follows from known results in a straightforward way. We state and prove it in a separate lemma  below. 
The real difficulty is in proving that the constraints are preserved in the limit. Since the $TL^2$ topology alone is not sufficient to ensure this, we need to establish some control of oscillations near the labeled points. This relies on on several technical lemmas which are of some independent interest. We state them in the subsection \ref{sec:dcestimates}. The proof of Proposition \ref{prop:Gamma} is presented in subsection \ref{sec:propproof}. 
 \nc

\begin{lemma} \label{lem:Gamma1}
Assume $\alpha>0$ and $\zeta_n \geq 1$. Under assumptions on $\X_n$ and $\eps_n$ of Proposition \ref{prop:Gamma}, the discrete energy $\GE_{n, \eps_n, \zeta_n}$ $\;\Gamma$-converges almost surely with respect to $TL^2$ topology to the energy $\theta_\eta \E$, defined in 
\eqref{eq:Econtinuum} as $n \to \infty$ if $\zeta_n \to \infty$.  \red Furthermore let $\{u_n \}_{n=1,2, \dots}$ be a sequence such that  $\sup_n \GE_{n, \eps_n, \zeta_n} < \infty$ and 
$\sup_n \|u_n\|_{L^\infty(\mu_n)} < \infty$.  Then $\{(\mu_n,u_n) \::\: n \in \N\}$ is precompact in $TL^2\!.$ \nc
\end{lemma}
\begin{proof}
From results in the literature \cite{GTS16, GTS18spectral} it follows that for any fixed $\zeta>0$ the discrete energies $\GE_{n, \eps_n, \zeta}$ $\:\Gamma$-converge to $\theta_\eta \E(\tacka; \gamma_\zeta)$ as $n \to \infty$, under standard assumptions on $\eps_n$. 
To show the liminf inequality for general  $\zeta_n$ consider a sequence $(\mu_n,u_n)$ $TL^2$ converging to $(\mu,u)$. For any fixed $k$,
\[ \liminf_{n \to \infty} \GE_{n, \eps_n, \zeta_n}(u_n) \geq \theta_\eta  \E(u; \zeta_k), \]
\red where $\E(u; \zeta_k)$ is given by \eqref{eq:Econtinuum} with $\gamma$ replaced by 
$\gamma_{\zeta_k}$, \nc
which implies the desired inequality by taking supremum over $k$. 
The limsup inequality follows by a simple diagonalization argument. 

We recall from the literature (e.g \cite{GTS16} or Proposition 4.4 of \cite{SleTho17plap}) that the precompactness of bounded sequences   with bounded energies already holds for the weight $\gamma \equiv 1$. Thus  the precompactness for $\GE_{n, \eps_n, \zeta_n}$ follows by comparison. 
\end{proof}

\subsubsection{The negative result} 
\begin{proposition} \label{prop:Gamma_subcrit}
Let $\eps_n$ be a sequence of positive numbers converging to zero as $n \to \infty$ and such that 
$\eps_n \gg d_\infty(\mu, \mu_n)$.
Let $\zeta_n \geq 1$ be a sequence converging to infinity. 
Consider $\alpha \leq d-2$. 
Then the constrained energy $    \GE_{n,\eps_n, \zeta_n}$,  defined in Proposition \ref{prop:Gamma}, 
$\,\Gamma$-converges \red almost surely \nc  in $TL^2$ metric to the unconstrained continuum energy $\theta_\eta \E$.
\end{proposition}
\begin{proof}
The liminf part  of the $\Gamma$-convergence claim follows from the liminf claim of Lemma \ref{lem:Gamma1}. 

To show the limsup inequality, we first observe that by localizing near the points of $\Gamma$, and given that limsup inequality holds for the unconstrained functional, the problem can be reduced to considering $\Gamma = \{0\}$, $u\equiv 0$, and the construction a sequence of functions $u_n \in L^2(\mu_n)$ such that $u_n(0)=1$, $ \GE_{n,\eps_n, \zeta_n}(u_n) \to 0$ as $n \to \infty$ and $u_n \to 0$ in $TL^2$ as $n \to \infty$.

We now  make some observation about the continuum functional. Namely when $\alpha \leq d-2$ then the function $ \varphi(x) = \ln \left(\ln \left(\frac{1}{|x|}\right) \right)$ belongs to $H^1_\gamma(B(0,1))$. Let $w_k = \max\{ \min\{\frac{1}{k} \varphi(x),1\}, 0\}$. Let $r_k>0$ be such that $w_k =1$ on $B(0,r_k)$. By mollifying we can obtain a smooth approximation $v_k$, $v_k =1$ on $B(0, r_k/2)$ and $\| v_k \|_{H^1_\gamma(B(0,1))} \leq 2 \| u_k \|_{H^1_\gamma(B(0,1))}$. 
Arguing as in  Section~5 of~\cite{GTS16}, if one defines for each $k \in \N$, a sequence $u^k_n \in L^2(\mu_n)$ by \red  $u^k_n(x_i) = v_k(x_i)$ for all $x_i \in X_n$ \nc  one has 
$u^k_n \to \red  v_k \nc$ in $TL^2$ and $\limsup_{n \to \infty}   \GE_{n,\eps_n, \zeta_n}(u^k_n) \geq \theta_\eta \E(\red v_k) \nc$. Since $v_k \to 0$ in $H^1_\gamma(B(0,1))$ as $k \to \infty$, the conclusion follows by a diagonalization argument. 
\end{proof}

\begin{corollary} \label{cor:neg0}
Let $\eps_n$ be a sequence of positive numbers converging to zero as $n \to \infty$ and such that 
$\eps_n \gg d_\infty(\mu, \mu_n)$.
Let $\zeta_n \geq 1$ be a sequence converging to $\infty$ as $n \to \infty$. Let $\alpha \leq d-2$. 
Let $u_n$ be a sequence of minimizers of the problem \eqref{eq:LapLearn} for $\GE_{n,\eps_n, \zeta_n}$. Let $c_n$ be the average of $u_n$ (with respect to measure $\mu_n$). Then \red almost surely \nc  $(\mu_n, u_n-c_n)$ converges in $TL^2$ to $(\mu,0)$; in other words the information about the labels is forgotten in the limit.
\end{corollary} 
\begin{proof}
Assume the claim is false. Then there exists $\delta>0$ and a subsequence $u_{n_j}$ such that foe all $j$, $d_{TL^2}((\mu_{n_j},u_{n_j}-c_{n_j}), (\mu,0)) > \delta$ for all $j$. By the maximum principle functions  $u_n$ are bounded by extremal values of $g$. Consequently, by Lemma \ref{lem:Gamma1},  $u_{n_j}-c_{n_j}$ has a further convergent subsequence. Without a loss of generality we can assume that  $u_{n_j}-c_{n_j}$ converges to some $v \in L^2(\mu)$. Then $\int v d \mu = \lim_{j \to \infty} \int u_{n_j} - c_{n_j} d \mu_{n_j} = 0$. 

By the limsup part of $\Gamma$-convergence of Proposition \ref{prop:Gamma_subcrit} there exists a sequence $v_n \in L^2(\mu_n)$ such that $\GE_{n,\eps_n, \zeta_n}^{con}(v_n) \to 0$ as $n \to \infty$. Since $u_{n_j}$ are minimizers 
$\GE_{n,\eps_n, \zeta_n}(u_{n_j}-c_{n_j}) \to 0$ as $n \to \infty$. We conclude by the liminf part of $\Gamma$-convergence that $\E(v) = 0$. Since $\int v d \mu =0$ this implies that $v \equiv 0$, which contradicts the assumption about the sequence.
\end{proof}

We note that the analogue of the negative result in Corollary \ref{cor:neg0} for the standard graph Laplacian (corresponding to $\gamma \equiv 1$) was proved in \cite{SleTho17plap}[Theorem 2.1].
The following corollary then follows by the squeeze theorem for $\Gamma$-convergence.
\begin{corollary} \label{cor:gen_neg}
Under the assumptions of Proposition \ref{prop:Gamma_subcrit} consider any sequence of graph based functionals $\mathcal F_n$ such that for $\GE_{n, \eps_n, 1} \leq \mathcal F_n \leq \GE_{n, \eps_n, \zeta_n}$ (where we note that $\GE_{n, \eps_n, 1} $ is just a convenient way to write the standard graph Laplacian). Let  $u_n$ be the minimizers of \eqref{eq:LapLearn} for $\F_n$ and 
let $c_n$ be the average of $u_n$ (with respect to measure $\mu_n$). Then $(\mu_n, u_n-c_n)$ converges \red almost surely \nc   in $TL^2$ to $(\mu,0)$. 
\end{corollary}
A particular consequence of this corollary is that the minimizers of the algorithm in \cite{shi2017weighted} converge to a constant as $n \to \infty$. 

\bigskip

\subsection{Estimates for the discrete to continuum convergence} \label{sec:dcestimates}

Here we establish several results needed in the proofs of the main results above. We follow a similar strategy as \cite{SleTho17plap}. Let us define the nonlocal continuum energy as
\begin{equation}\label{eq:nlene}
\E_{\eps, \zeta}(u) = \frac{1}{\eps^{2}} \iint \gamma_\zeta(x)\eta_\eps(x-y)|u(x)-u(y)|^2 d \mu(x) d \mu(y).
\end{equation}
It serves as an intermediary between the discrete graph based functionals and the continuum derivative-based functionals. 
\medskip

\begin{lemma}[discrete to nonlocal control]  \label{lem:disc-nonlocal}
Consider $\Omega$, $\mu$, $\eta$, $\zeta$, and $x_i$ as in Theorem \ref{thm:conv}.
Let  $\tilde{\eta}(|x|) = 1$ for $|x|\leq 1$ and $\tilde{\eta}(|x|) = 0$ otherwise, and so $\tilde \eta \leq \eta$.
Let $T_n$ be a sequence of transport maps satisfying the conclusions of Theorem~\ref{thm:TransBound} and let $\tilde{\eps}_n = \eps_n-2\|T_n-\Id\|_{L^\infty(\Omega)}$.
Define $\GE_{n, \eps_n, \zeta_n}(\tacka ;\eta)$ by~\eqref{eq:energy} 
\red and $\E_{\tilde \eps_n, \tilde \zeta_n}(\tacka;  \tilde \eta)$ be \eqref{eq:nlene}, \nc 
 where we explicitly denote the dependence of $\eta$.
Let $\tilde \zeta_n >0$ be such that $\zeta_n \geq \tilde\zeta_n\;$ and 
\red
$ \left(\frac{r_0}{2C} \right)^\alpha \ell_n^{-\alpha} \geq \tilde\zeta_n\;$ where $C$ is the constant from Theorem~\ref{thm:TransBound} and $\ell_n$ is the transportation length scale from the same theorem. \nc
Then there exists $n_0 \in \N$ and a constant $\overline C>0$ (independent of $n$ and $u_n$) such that for all $n\geq n_0$
\[\E_{\tilde \eps_n, \tilde \zeta_n}(u_n\circ T_n;  \tilde \eta) \leq  \overline C\, \GE_{n,\eps_n,\zeta_n}(u_n; \eta)\]
\end{lemma}
\begin{proof}
If  $\left| \frac{x-z}{\tilde{\eps}_n}\right|  < 1$ then
\[ |T_n(x) - T_n(z)| \leq 2 \| T_n - \Id \|_{L^\infty(\Omega)} + | x-z| \leq 2 \| T_n - \Id \|_{L^\infty(\Omega)} +  \tilde{\eps}_n = \eps_n. \]
So,
\[ \left|  \frac{x-z}{\tilde{\eps}_n} \right|  < 1 \quad \te{ implies }\quad \left|  \frac{T_n(x) - T_n(z)}{\eps_n} \right|  \leq 1 \]
and therefore
\[ \left|  \frac{x-z}{\tilde{\eps}_n} \right|  < 1 \quad \te{ implies } \quad \tilde{\eta}\left( \frac{|x-z|}{\tilde{\eps}_n}\right) = 1 = \tilde{\eta}\left( \frac{|T_n(x)-T_n(z)|}{\eps_n}\right). \]
Hence,
\[ \tilde{\eta}\left( \frac{|x-z|}{\tilde{\eps}_n} \right) \leq \tilde{\eta}\left( \frac{|T_n(x) - T_n(z)|}{\eps_n} \right) \leq \eta\left( \frac{|T_n(x) - T_n(z)|}{\eps_n} \right). \]
From the assumptions on $\tilde \zeta_n$ and $T_n$ follows that 
\[  2 \| T_n - \Id \|_{L^\infty(\Omega)} \leq  r_0 (\tilde \zeta_n -1)^{-\frac{1}{\alpha}} = r_{\tilde \zeta_n} \]
where $r_{\tilde \zeta_n}$ is the length scale such that $1 + \left( \frac{r_0}{\dist(x,\Gamma)} \right)^{\alpha} > \tilde \zeta_n$ if $\dist(x,\Gamma) < r_{\tilde \zeta_n}$.
We claim that for a.e. $x \in \Omega$
\begin{equation} \label{temp22}
 \min \left\{1 + \left( \frac{r_0}{\dist(x,\Gamma)} \right)^{\alpha}, \tilde \zeta_n \right\}  \leq {2^\alpha} \min \left\{1 + \left( \frac{r_0}{\dist(T_n(x),\Gamma)} \right)^{\alpha}, \tilde \zeta_n \right\}  
 \end{equation}
Namely if $ d(x, \Gamma) \leq r_{\tilde \zeta_n}$ then $ d(T_n, \Gamma) \leq |T_n(x) - x| + r_{\tilde \zeta_n} \leq 2 r_{\tilde \zeta_n}$ for a.e. such $x$. Thus 
\[ 1 + \left( \frac{r_0}{\dist(T_n(x),\Gamma)} \right)^{\alpha} \geq 1 + \frac{1}{2^\alpha} \left( \frac{r_0}{r_{\tilde \zeta_n}} \right)^{\alpha}  \geq \frac{1}{2^\alpha} \tilde \zeta_n \]
If $ d(x, \Gamma) > r_{\tilde \zeta_n}$ then $d(x, \Gamma) \geq \frac12 d(T_n(x), \Gamma)$ for a.e. such $x$. Thus $\left( \frac{r_0}{\dist(T_n(x),\Gamma)} \right)^{\alpha} \geq \frac{1}{2^\alpha} 
\left( \frac{r_0}{\dist(x,\Gamma)} \right)^{\alpha}$. 
\medskip

Using \eqref{temp22} we conclude
\begin{align*}
& \E_{\tilde \eps_n,\tilde \zeta_n}(u_n\circ T_n;\tilde{\eta})  \\
& = \frac{1}{{\tilde \eps_n}^{2}} \iint  
\min \left\{1 + \left( \frac{r_0}{\dist(x,\Gamma)} \right)^{\alpha}, \tilde \zeta_n \right\} \tilde\eta_{\tilde \eps_n}(x-y)|u_n(T_n(x))-u_n(T_n(y))|^2 d \mu(x) d \mu(y) \\
& \leq 2^\alpha
 \frac{\eps_n^d}{\tilde{\eps}_n^{d+2}} \iint  \min \left\{1 + \left( \frac{r_0}{\dist(T_n(x),\Gamma)} \right)^{\alpha},  \zeta_n \right\}  \\
 & \phantom{   \leq 2^\alpha \frac{\eps_n^d}{\tilde{\eps}_n^{d+2}} \iint  \;} 
     \eta_{\eps_n}(|T_n(x)-T_n(y)|) \left|  u_n(T_n(x)) - u_n(T_n(z)) \right| ^2d \mu(x) d \mu(y) \\ 
& = 2^\alpha \frac{\eps_n^{d+2}}{\tilde{\eps}_n^{d+2}}\GE_{n,\eps_n,\zeta_n}(u_n; \eta).
\end{align*}
\end{proof}

In the next lemma we show that boundedness of non-local energies implies regularity at scales greater than $\eps$ with weight $\gamma_{\tilde \zeta}$.
This allows us to relate non-local bounds to local bounds after mollification using a mollifier $J\in C^\infty_c(\R^d,[0, \infty))$, with $\int_{\R^d}J(x)\, d x = 1$, and $J_\eps(x)=\frac{1}{\eps^d} J(x/\eps)$.

\begin{lemma} [nonlocal to weak local control] \label{lem:nonlocal-wloc}
There exists a constant $C\geq 1$ and a radially symmetric mollifier $J$ with $\supp(J)\subseteq \overline{B(0,1)}$ such that for all $\eps>0$, $u\in L^2(\Omega)$, and any
$\Omega^\prime\subset\subset \Omega$ (i.e. for every $\Omega^\prime$ that is compactly contained in $\Omega$) with $\dist(\Omega^\prime,\partial \Omega)> \eps$ it holds that\begin{equation} \label{eq:nonlocal-wloc}
 \E(u*J_\eps; \gamma_{\tilde \zeta}, \Omega^\prime) \leq C \E_{\eps, \tilde \zeta}(u;\Omega) 
\end{equation}
 where for both functionals we explicitly denote the dependence of the domain.
\end{lemma}

\begin{proof}
Let $J$ be a radially symmetric mollifier whose support is contained in $\overline{B(0,1)}$.
There exists $\beta>0$ such that $J \leq \beta \eta(|\cdot|)$ and $|\nabla J| \leq \beta \eta(|\cdot|)$.
Let $u_\eps = J_\eps * u$.
For arbitrary  $x\in \Omega$ with $\dist(x,\partial \Omega)> {\eps}$ we have
\begin{align*}
|\nabla {u_\eps}(x) |  & = \left|\int_\Omega \nabla J_{{\eps}} \left(x-z\right)u(z) \, d z \right|\\
& =  \left| \int_\Omega \nabla J_{{\eps}}\left(x-z \right)\left(u(z) - u(x) \right)\, d z - \int_{\R^d\setminus \Omega} \nabla J_{{\eps}} \left(x-z \right)u(x) \, d z \right|\\
& \leq \frac{\beta}{{\eps}^{d+1}} \int_\Omega \eta\left(\frac{|x-z|}{{\eps}} \right)\left|u(z) - u(x) \right|\, d z + \frac{1}{{\eps^{d+1}}} \int_{\R^d\setminus \Omega} \left|\nabla J \left(\frac{x-z}{\eps}\right)\right|\left|u(x) \right|\, d z.
\end{align*}
where the second line follows from $ \int_{\R^d} \nabla J(w) \, d w = 0$.
For the second term we have
\begin{align*}
\frac{1}{{\eps}^{d+1}} \int_{\R^d\setminus \Omega} \left| \nabla J \left( \frac{x-z}{{\eps}}\right)\right| |u(x)| \, d z & =0 
\end{align*}
since for all $z\in\R^d\setminus\Omega$ and $x\in\Omega$ with $\dist(x,\partial \Omega)> {\eps}$ it follows that $|x-z| > \eps$ and thus $\nabla J \left( \frac{x-z}{{\eps}}\right)=0$.
Therefore, for $\theta_\eta=\int_{\R^d} \eta(|w|) \, d w$,
\begin{align*}
|\nabla {u_\eps}(x) |^2 & \leq \beta^2 \left(\int_\Omega \frac{1}{{\eps}} {\eta}_{{\eps}}(|x-z|) \left| u(z) - u(x) \right|\, d z \right)^2 \\
 & = \frac{\theta_\eta^2 \beta^2}{\eps^2} \left(\int_\Omega \frac{\eta_\eps(|x-z|)}{\theta_\eta} |u(z)-u(x)| \, d z\right)^2 \\
 & \leq \theta_\eta \beta^2 \int_\Omega {\eta}_{{\eps}}(|x-z|) \frac{\left|u(z) - u(x) \right|^2}{{\eps}^2} \, d z 
\end{align*}
by Jensen's inequality (since $\frac{1}{\theta_\eta}\int_{\R^d} \eta_\eps(|x-z|)\, d z = 1$).
Hence,
\begin{align*}
\int_{\Omega^\prime} \left|\nabla {u_\eps}(x) \right|^2 \gamma_{\tilde \zeta} (x) \rho^2(x) \, d x & \leq \theta_\eta \beta^2 \int_\Omega \int_\Omega {\eta}_{{\eps}}(|x-z|) \left|\frac{u(z) - u(x)}{\eps} \right|^2 \gamma_{\tilde \zeta} (x)  \rho^2(x) \, d z \, d x \\
 & \leq \frac{\theta_\eta \beta^2 \sup_{x\in \Omega} \rho(x)}{\inf_{x\in \Omega} \rho(x)} \,\E_{\eps, \tilde \zeta}(u;\Omega) 
\end{align*}
which completes the proof.
\end{proof}

We now show that controlling the local energy with cut-off near the singularity is sufficient to be able to find a nearby (in $H^1_\gamma$) function which has a similarly bounded energy without a cut-off.

\begin{lemma}[weak local to strong local control] \label{lem:wloca-loc}
Consider $\tilde \zeta>1$ such that  $r_{\tilde \zeta}$ defined in \eqref{eq_rzeta} satisfies $r_{\tilde \zeta} \leq  \frac12 \min \{ |x-y| \: :\: x,y \in \Gamma, x \neq y \}$. Let $ \bar r =  r_{\tilde \zeta}$. Then there exists a constant $C>0$ such that for every $u \in H^1(\Omega)$ there exists 
$v \in H^1_\gamma(\Omega)$ such that 
\begin{align} 
v|_{B(z,\bar r/2)} & \equiv \frac{1}{|B(0, \bar r/2)|} \int_{B(z, \bar r /2)} u(x) dx \quad \te{ for all } z \in \Gamma, \label{tline1} \\
v & = u \quad \te{ on } \Omega \setminus \Gamma_{\bar r}   \label{tline2} \\%
\E(v; \gamma) & \leq C \E(u, \gamma_{\tilde \zeta}).  \label{eq:wloc-loc}
\end{align}
\end{lemma}
\begin{proof}
Using the finiteness of $\Gamma$,
from Lemma \ref{lem:lewl} via translations and a rescaling follows that there exists $ c\geq 1$, and $v \in H^1(\Omega)$ satisfying \eqref{tline1} and \eqref{tline2} 
 such that 
\[ \int_{\Gamma_{\bar r}} |\nabla v(x)|^2 dx \leq c  \int_{\Gamma_{\bar r}} |\nabla u(x)|^2 dx. \]
Using that $\nabla v = 0$ on $\Gamma_{\frac{\bar r}{2}}$ and $1 + \left( \frac{ r_0}{\bar r}\right)^\alpha  = \tilde \zeta$  we obtain
\begin{align*}
\E(v; \gamma) & =  \int_{\Gamma_{\bar r} }\gamma(x) |\nabla v(x)|^2 \rho(x) dx  +  \int_{\Omega \setminus \Gamma_{\bar r} } \gamma(x) |\nabla u(x)|^2 \rho(x) dx \\
& \leq  \int_{\Gamma_{\bar r} \setminus \Gamma_{\frac{\bar r}{2}}  } \left( 1 + \left( \frac{2 r_0}{\bar r}\right)^\alpha  \right) \, |\nabla v(x)|^2 \rho(x) dx  +
 \int_{\Omega \setminus \Gamma_{\bar r} } \gamma_{\tilde \zeta}  |\nabla u(x)|^2 \rho(x) dx  \\
 & \leq 2^\alpha \int_{\Gamma_{\bar r} \setminus \Gamma_{\frac{\bar r}{2}}  }  \tilde \zeta \,  |\nabla v(x)|^2 \rho(x) dx  
 +  \int_{\Omega \setminus \Gamma_{\bar r} } \gamma_{\tilde \zeta}  |\nabla u(x)|^2 \rho(x) dx  \\
 & \leq 2^\alpha c \int_{\Gamma_{\bar r}  }  \tilde \zeta \,  |\nabla u(x)|^2 \rho(x) dx   +
    \int_{\Omega \setminus \Gamma_{\bar r} } \gamma_{\tilde \zeta}  |\nabla u(x)|^2 \rho(x) dx   \\
 & \leq 2^\alpha c \int_{\Omega } \gamma_{\tilde \zeta}     |\nabla u(x)|^2 \rho(x) dx 
\end{align*}
\end{proof}

\red
\begin{lemma} \label{lem:avgest}
There exists $C>0$ such that for all $0< \eps <r  \leq 1$, for all $u \in H^1(B(0,r))$ such that $u \leq 0$ on $B(0, \eps)$
\begin{equation}\label{eq:avgest}
 \dashint_{B(0,r)} u(x) dx \leq C \sqrt{\int_{B(0,r) \setminus B(0, \eps)} |\nabla u|^2 dx } \:
 \begin{cases}
 \eps^{\frac{2-d}{2}} \quad & \te{if } d \geq 3 \\
 \sqrt{-\ln \eps} & \te{if } d=2.
 \end{cases} 
\end{equation}
\end{lemma}
\begin{proof}
We only prove the claim for $r=1$. The claim for general $r$ follows by a simple change of variables $y = x/r$. Furthermore we note that we can assume that $u =0$ on $B(0, \eps)$, since for general $u$ one can consider $\tilde u = \max \{u,0\}$ and note that 
\[ \int_{B(0,1)} u dx \leq \int_{B(0,1)} \tilde u dx \quad \te{ and } \quad 
 \sqrt{\int_{B(0,1) \setminus B(0, \eps)} |\nabla \tilde u|^2 dx}\,  \leq \sqrt{\int_{B(0,1) \setminus B(0, \eps)} |\nabla  u|^2 dx} \,.
\]
Let $v(x)=\frac{1}{2d}|x|^2$ and 
\[\Phi(x) =
\begin{cases}
-\frac{1}{2}\log|x|,&\text{if }d=2\\
\frac{1}{d(d-2)|x|^{d-2}},&\text{if }d\geq 3.
\end{cases}\]
Then $\Delta(v+\Phi) = 1$ for $x\neq 0$ and $\frac{\partial v}{\partial \nu}+\frac{\partial \Phi}{\partial \nu} =0$ on $\partial B(0,1)$. Since $u=0$ on $B(0,\eps)$ we have
\begin{align*}
\int_{B(0,1)}u\, dx&=\int_{B(0,1)\setminus B(0,\eps)}u\Delta(v+\Phi)\, dx\\
&=-\int_{B(0,1)\setminus B(0,\eps)}\nabla u\cdot (\nabla v + \nabla \Phi)\, dx\\
&=\frac{1}{d}\int_{B(0,1)\setminus B(0,\eps)}\nabla u \cdot x\left( \frac{1}{|x|^d}-1 \right)\, dx\\
&\leq \frac{2}{d}\int_{B(0,1)\setminus B(0,\eps)} |\nabla u||x|^{1-d}\, dx\\
&\leq \frac{2}{d}\sqrt{\int_{B(0,1)\setminus B(0,\eps)} |\nabla u|^2\, dx}\sqrt{\int_{B(0,1)\setminus B(0,\eps)} |x|^{2-2d}\, dx}.
\end{align*}
The proof is completed by integrating the last term on the line above.
\end{proof}
\nc

\subsection{Proof of Proposition \ref{prop:Gamma}. }
\label{sec:propproof}

\begin{proof}
To show the $\limsup$ inequality recall that $\mathcal S$, the set of smooth functions which are constant in some neighborhood of $\Gamma$, 
 is dense in $H^1_\gamma(\Omega)$, by Corollary \ref{cor:dense}. 
The fact that for every $f \in \mathcal S$, $\GE_{n, \eps_n, \zeta_n}(f) \to \theta_\eta \E(f)$ follows by a standard argument, which was for example presented for total variation in Section~5 of~\cite{GTS16}. The existence of a recovery sequence for arbitrary $f \in H^1_\gamma(\Omega)$ follows by a density argument.

To show the $\liminf$ inequality consider a sequence $(\mu_n, u_n)$ converging in $TL^2$ to $(\mu,u)$. We can assume without a loss of generality that $u_n|_\Gamma = g$ and that $\liminf_{n \to \infty}     \GE_{n,\eps_n, \zeta_n}(u_n)$ is finite. 
Since $\eps_n \gg d_\infty(\mu, \mu_n) \gtrsim n^{-1/d}$ it follows that $\zeta_n \to \infty$ as $n \to \infty$. 
By Lemma \ref{lem:Gamma1}, discrete energy $\GE_{n,\eps_n, \zeta_n}$ $\Gamma$-converges to $\E$. Thus $u \in H^1_\gamma(\Omega)$ and 
\[ \liminf_{n \to \infty}     \GE_{n,\eps_n, \zeta_n}(u_n) \geq \theta_\eta \E(u). \]

What remains to be shown is that $u|_\Gamma = g$. The fact that $u|_\Gamma$ is a well defined object follows from Lemma \ref{lem:lebesgue}.
Let us assume that $\liminf_{n \to \infty}     \GE_{n,\eps_n, \zeta_n}(u_n)  = \lim_{n \to \infty}     \GE_{n,\eps_n, \zeta_n}(u_n)$. For the general case one needs to consider a subsequence, which we omit for notational simplicity. 
We have that $E_{max} = \sup_n  \GE_{n,\eps_n, \zeta_n}(u_n)  < \infty.$

We first show that near points  $z \in \Gamma$, the values of $u_n$ remain, on average, close to $g(z)$. More precisely
\[ E_{max} \geq \GE_{n,\eps_n, \zeta_n}(u_n) \geq \frac{1}{2n^2\eps_n^{2}} \sum_{x \in \X_n} \zeta_n \eta_{\eps_n}(x-z) |u_n(x) - g(z)|^2. \]
and thus 
\[ \frac{1}{n} \sum_{x \in \X_n} \eta_{\eps_n}(x-z) |u_n(x) - g(z)|^2 \leq 2E_{max} \, \frac{n \eps_n^2}{\zeta_n}. \]
Since $\eta\geq 1$ on $B(0,1)$
\[ \frac{1}{n} \sum_{x \in \X_n, |x-z| < \eps_n }  |u_n(x) - g(z)|^2 \leq 2E_{max} \, \frac{n \eps_n^{d+2}}{\zeta_n}. \]
Let $T_n$ be a sequence of transport maps satisfying the conclusions of Theorem~\ref{thm:TransBound} and let $\tilde{\eps}_n = \eps_n-2\|T_n-\Id\|_{L^\infty(\Omega)}$.

Then for a.e. $x \in B(z, \tilde \eps_n)$, $T_n(x) \in B(z, \eps_n)$ and thus 
\[ \int_{B(z, \tilde \eps_n)} |u_n (T_n(x)) - g(z)|^2 \rho(x)  dx  \leq 2E_{max} \, \frac{n \eps_n^{d+2}}{\zeta_n}. \]
Therefore
\begin{equation} \label{eq:oscest}
 \dashint_{B(z, \tilde \eps_n)} |u_n (T_n(x)) - g(z)|^2 \rho(x)  dx  \lesssim \frac{n \eps_n^2}{\zeta_n} \ll 1
\end{equation}
by the assumption on $\zeta_n$. \red Consequently for all $y \in B\left( z, \frac{\tilde \eps_n}{2} \right)$
\begin{equation} \label{eq:oscest2}
 \dashint_{B(y, \, \tilde \eps_n/2)} |u_n (T_n(x)) - g(z)|^2  dx  \lesssim \frac{n \eps_n^2}{\zeta_n} \ll 1. 
\end{equation}
\nc

By Lemma \ref{lem:disc-nonlocal} we know that, for $\tilde \zeta_n = \min \left\{ \zeta_n , \left(\frac{r_0}{2C} \right)^\alpha \left(\frac{n}{\ln n}\right)^{\alpha/d}  \right\}$ where $C$ is the constant from Theorem~\ref{thm:TransBound}
\begin{equation} \label{eq:tempest1}
\E_{\tilde \eps_n, \tilde \zeta_n}(u_n \circ T_n;  \tilde \eta) \lesssim  \GE_{n,\eps_n,\zeta_n}(u_n; \eta)
\end{equation}
\red We note that since $\alpha > d-2$, and $\eps_n \gg \left(\frac{\ln n}{n}\right)^{1/d}, $ 
 $\; \tilde \zeta_n \gg \eps_n^{2-d}$ if $d >2$ and $\tilde \zeta_n \gg -\ln \eps_n$ 
Let $\hat \eps_n  =  \frac{\tilde \eps_n}{2} $. By definition of $\E_{\hat \eps_n, \tilde \zeta_n}$
\begin{equation} \label{eq:tempest1b}
\E_{\hat \eps_n, \tilde \zeta_n}(u_n \circ T_n;  \tilde \eta) \lesssim  \E_{\tilde \eps_n, \tilde \zeta_n}(u_n \circ T_n;  \tilde \eta) 
\end{equation}
\nc

Let $J$ be a mollifier used in the proof of Lemma \ref{lem:nonlocal-wloc} and let \red $\tilde u_n = (u_n \circ T_n) * J_{\hat \eps_n}$. \nc From \eqref{eq:oscest2} follows that \red for all
$y \in B\left( z, \hat \eps_n \right)$
\begin{equation} \label{eq:ungest}
 |\tilde u_n(y) - g(z)| \lesssim \frac{n \eps_n^2}{\zeta_n} 
\end{equation}
for all $z \in \Gamma$. \nc 
Combining the estimate of the lemma with \eqref{eq:tempest1} yields
\[  \E(\tilde u_n; \gamma_{\tilde \zeta_n}, \Omega^\prime_n)  \lesssim   \GE_{n,\eps_n,\zeta_n}(u_n; \eta) \]
\red where $\Omega_n' = \{ y \in \Omega \::\: d(y, \partial \Omega) > \eps_n \}$. \nc
%
%
Finally by Lemma \ref{lem:wloca-loc} there exist $v_n \in H^1_\gamma(\Omega^\prime_n)$ such that \red for all $z \in \Gamma,$
$\; v_n(z) = \dashint_{B\left( z, \bar r/2 \right)} \tilde u_n(y) dy$ \nc 
where $\bar r = r_{\tilde \zeta_n}$
 and
\begin{equation} \label{eq:vnene}
 \E(v_n, \Omega^\prime_n)  \lesssim   \GE_{n,\eps_n,\zeta_n}(u_n; \eta). 
 \end{equation}
\red From \eqref{eq:ungest} follows that for some $C$ independent of $n$, for all $z \in \Gamma$, and all $y \in B(z, \hat \eps_n)$, using that $\int_{B(0, \bar r/2)} |\nabla \tilde u_n|^2 dx \leq 
\frac{1}{\tilde \zeta_n} \E(\tilde u_n; \gamma_{\tilde \zeta_n}, \Omega^\prime_n) $, 
$u(y) -g(z) - C \frac{n \eps_n^2}{\zeta_n} \leq 0$. Thus by Lemma \ref{lem:avgest}
\[  \dashint_{B(0,\bar r/2)} \tilde u_n(x) dx \leq g(z) + C \frac{n \eps_n^2}{\zeta_n}  
+ C_1
 \tilde \zeta_n^{-1/2} \,
 \begin{cases}
 \eps^{\frac{2-d}{2}} \quad & \te{if } d \geq 3 \\
 \sqrt{-\ln \eps} & \te{if } d=2.
 \end{cases}  
 \]
for some $C_1$ independent of $n$. By the assumptions on $\zeta_n$ and definition of $\tilde \zeta_n$ the right hand side converges to zero an $n \to \infty$. Analogous lower bound is obtained following the same argument. 
Therefore $v_n(z) - g(z) = \dashint_{B(0,\bar r/2)} \tilde u_n(x) dx - g(z)$ converges to zero as $n \to \infty$.
 \nc

We note that by construction $d_{TL^2}((\mu_n, u_n), (\mu, \tilde u_n)) \to 0$ and thus $\tilde u_n \to u$ in $L^2(\Omega)$. By construction $\| v_n - \tilde u_n \|_{L^2(\Omega^\prime
_n)} \lesssim 
\| \tilde u_n \|_{L^2(\Gamma_{\bar r_n}) }$, where $\bar r_n = r_{\tilde \zeta_n}$ and for any $s >0$
\begin{equation} \label{eq:thickenedGamma}
\Gamma_{s} = \{ x \in \Omega \::\: \dist(x,\Gamma)< s\}.
\end{equation}
Since $(\mu_n, u_n)  \overset{TL^2}{\longrightarrow} (\mu,u)$ it follows that 
\[ \int_{T_n^{-1}(\Gamma_{\bar r_n})} |u_n(T_n(x))-u(x)|^2 d\mu(x) \to 0 \qquad \te{ as } n \to \infty. \]
Since $u \in L^2(\mu)$, $\lim_{\delta \to 0} \sup\left\{ \int_A u^2(x) dx \: : \: \mu(A)< \delta \right\} = 0$.
 Therefore 
\[ \| \tilde u_n \|_{L^2(\Gamma_{\bar r_n}) } \leq 2 \int_{T_n^{-1}(\Gamma_{\bar r_n})} |u_n(T_n(x))-u(x)|^2 + u^2(x) d\mu(x) \to 0 \qquad \te{as } n \to \infty. \]
Thus $\| v_n - \tilde u_n \|_{L^2(\Omega^\prime_n)} \to 0$ and $n \to \infty$ and consequently $v_n \to u$ in $L^2(\mu)$. From \eqref{eq:vnene} follows that $v_n$ is a bounded sequence in $H^1_\gamma(K)$ for any compact subset $K \subset \subset \Omega$. Combining this with the fact that $v_n \to u$ in $L^2(\mu)$ implies, via estimate \eqref{eq:traceest} of the Trace Theorem (Theorem \ref{thm:trace}), that $v_n(z) \to \Tr u(z) $ as $n \to \infty$ for all $z \in \Gamma$. 
\red  Since $v_n|_{\Gamma} \to  g$ as $n \to \infty$ \nc  we conclude that
 $\Tr u(z) = g(z)$ for all $z \in \Gamma$.
\end{proof}


\section{Regularity of minimizers of the graph properly-weighted Laplacian} \label{sec:reg}

\subsection{H\"older estimate near labeled points}
\label{sec:near}
Our main result in this section is a type of H\"older estimate near the labeled points, which shows that solutions of the graph-based learning problem \eqref{eq:optimality} attain their boundary values on $\Gamma$ continuously, with high probability. The proof is a graph-based version of the barrier argument from Theorem \ref{thm:wellposedness} that established continuity at labels in the continuum PDE, given in Eq.~\eqref{eq:holderpde}.  Barrier arguments for proving H\"older regularity of solutions of PDEs are standard techniques for first order equations, such as Hamilton-Jacobi equations \cite{bardi2008optimal}. Normally, barrier arguments do not work for second order elliptic equations (since fundamental solutions are unbounded), though there are a handful of exceptions, such as the $p$-Laplace equation for $p>d$ \cite{calder2017game}, level set equations for affine curvature motion \cite{calder2018limit}, and our continuum equation \eqref{eq:pde}. 

\red
Our proof uses the barrier $v(x) = C|x-y|^{\beta}$, which is a supersolution of the continuum PDE \eqref{eq:pde} for $0 < \beta < \alpha + 2-d$, due to Proposition \ref{prop:barrier}. For $\beta < 2$, the barrier has a singularity at $x=y$, which has to be treated carefully in the translation to the graph setting. We show in Lemma \ref{lem:barrier} that $v$ is a supersolution on the graph with high probability away from a small ball $B(y,C\eps)$. Due to the singularity in the barrier, we cannot prove the supersolution property within the ball $B(y,C\eps)$. To fill in the gap within this ball, we require a local regularity result, given in Lemma \ref{prop:micro_holder}, that relies on the variational structure of the problem. At a high level, the proof is similar to the proof of H\"older regularity of solutions to the graph-based game theoretic $p$-Laplace equation, given in \cite{calder2017game}, though many of the ingredients are different. In particular, in \cite{calder2017game} there is no variational interpretation of the problem, and the local argument utilizes another barrier construction.  

We now proceed to present the main results in this section. Throughout we always assume $n\eps^d \geq 1$.  \nc Our main result is the following H\"older-type estimate.
\begin{theorem}[H\"older estimate]\label{thm:holder}
\red Let $\alpha>d-2$, and assume $\gamma$ satisfies \eqref{eq:g} and $\rho\in C^{1,\sigma}(\bar{\Omega})$ is bounded above and below by positive constants. Let $0< \eps \leq 1$, $\zeta\geq 1 +\eps^{-\alpha}$, $0 < \beta < \alpha + 2-d$, and let $u\in L^2(\X_n)$ be the solution of \eqref{eq:optimality}. Then \nc for each $z\in \Gamma$ the event that
\begin{equation}\label{eq:holder}
|u(x_i) - u(z)| \leq C|x_i-z|^\beta+Cn^{1/2}\eps^{1 + \alpha/2}
\end{equation}
holds for all $x_i\in \X_n$ occurs with probability at least $1 -C\exp\left(  -cn\eps^{d +4} + \log(n) \right)$.
\end{theorem}
The proof of Theorem \ref{thm:holder}, given at the end of the section, relies on some preliminary results that we establish after a few remarks.
\begin{remark}
For the result in Theorem \ref{thm:holder} to be useful, we must choose $\eps_n\to 0$ so that $n\eps_n^{d+4}\gg \log(n)$ and $n\eps_n^{\alpha+2} \ll 1$. Therefore, we must have $\alpha > d+2$ and 
\begin{equation}\label{eq:epsbounds}
\left( \frac{\log(n)}{n} \right)^{1/(d+4)} \ll \eps_n \ll \left( \frac{1}{n} \right)^{1/(\alpha+2)}.
\end{equation}
\end{remark}
\begin{remark}\label{rem:HC}
If we replace $\gamma_\zeta$ with $\gamma_{\zeta,C\eps}$, as defined in Remark \ref{rem:modified}, then we can improve Theorem \ref{thm:holder} to read
\begin{equation}\label{eq:holder2}
|u(x_i) - u(z)| \leq C|x_i-z|^\beta+C\zeta^{-1/2}n^{1/2}\eps,
\end{equation}
under the same assumptions and with the same probability, except we also require $\zeta\geq 1 + C\eps^{-\alpha}$. In this model, the restrictive upper bound in \eqref{eq:epsbounds} is not required.
\end{remark}

We now turn to the proof of Theorem \ref{thm:holder}. We first recall a useful lemma from \cite{calder2017game}.
\begin{lemma}[Remark 7 from \cite{calder2017game}]\label{lem:ch}
Let  $Y_1,Y_2,Y_3,\dots,Y_n$  be  a sequence of \emph{i.i.d}~random variables on  $\R^d$ with Lebesgue density $\rho:\R^d\to \R$, let $\psi:\R^d \to \R$ be bounded and Borel measurable with compact support in a ball $B(x,h)$ for some $h>0$, and define
\[ Y = \sum_{i=1}^n \psi(Y_i).\]
Then for any $0 \leq \lambda \leq 1$
\begin{equation}\label{eq:con2}
\P\left( |Y-\mathbb{E}(Y)|\geq C\|\psi\|_{L^\infty(B(x,h))} nh^d\lambda\right) \leq 2\exp(-cnh^d\lambda^2),
\end{equation}
for all $0 < \lambda \leq 1$, where $C,c>0$ are constants depending only on $\|\rho\|_{L^\infty(B(x,h))}$ and $d$. 
\end{lemma}
We can use Lemma \ref{lem:ch} to prove pointwise consistency for our properly-weighted graph Laplacian. It extends, in a refined form, the results of \cite{calder2017game}[Theorem 5]. 
It is related to well known results on the pointwise consistency of the graph Laplacian \cite{Singer06}.

For simplicity we set 
\begin{equation}\label{eq:lapr}
\Delta_\rho \varphi =\rho^{-1}\divv\left( \gamma \rho^2 \nabla \varphi \right).
\end{equation}
\begin{theorem}\label{thm:consistency}
For $\delta>0$, let $D_{n,\eps,\delta}$ be the event that
\begin{equation}\label{eq:consistency}
\left|\L_{n,\eps, \infty}\varphi(x_i) - \tfrac{1}{2}\sigma_\eta\Delta_\rho \varphi(x_i)\right|\leq C( (\eps\beta_1 + \eps^2 \beta_2)M^{-(\alpha+2)} + \eps\beta_2 M^{-(\alpha+1)} + \delta\beta_3 M^{-\alpha})
\end{equation}
holds for all $x_i$ with $2\eps < \dist(x_i,\Gamma)\leq R/4$ and all $\varphi\in C^3(B(x_i,2\eps))$, where $\beta_k = \|\varphi\|_{C^k(B(x_i,2\eps))}$ and $M =\dist(x_i,\Gamma)-2\eps$. Then for $\eps \leq \delta\leq 1$ we have
\begin{equation}\label{eq:prob}
\P(D_{n,\eps,\delta}) \geq 1 -C\exp\left(  -c\delta ^2n\eps^{d +2} + \log(n) \right)
\end{equation}
\end{theorem}
\begin{proof}
Let us write $\L$ in place of $\L_{n,\eps,\infty}$ for simplicity. We also define
\[w(x,y) = \frac{1}{2n\eps^2}(\gamma(x) + \gamma(y))\eta_{\eps}(x-y).\]
Then
\[\L u(x) = \sum_{y\in \X_n}w(x,y)(u(y) - u(x)).\]
By conditioning on the location of $x\in X_n$, we can assume without loss of generality that $x\in \Omega$ is a fixed (non-random) point, $B(x,2\eps)\subset \Omega$ and $\text{dist}(x,\Gamma)> 2\eps$. Let $\varphi \in C^3(B(x,2\eps))$, $p=D\varphi (x)$ and  $A=D^2\varphi (x)$. Note that 
\begin{align}\label{eq:ts}
 \L\varphi (x)=\sum_{i=1}^dp_i\sum_{y\in X_n}w(x,y)(y_i -x_i) &+\frac{1}{2}\sum_{i,j=1}^da_{ij}\sum_{y\in X_n}w(x,y)(y_i -x_i)(y_j -x_j) \nonumber\\
&\hspace{1.5in}+O\left( \eps^3\beta_3 \text{deg}(x) \right),
\end{align}
where $\text{deg}(x)$ is the degree given by
\[\text{deg}(x) = \sum_{y\in \X_n}w(x,y).\]
Since $\dist(y,\Gamma) \geq \dist(x,\Gamma) - 2\eps$ we have
\begin{align*}
w(x,y)&\leq \frac{C}{2n\eps^{d+2}}\left( \frac{1}{\dist(x,\Gamma)^\alpha} + \frac{1}{\dist(y,\Gamma)^\alpha} \right)\\
&\leq\frac{C}{n\eps^{d+2}(\dist(x,\Gamma)-2\eps)^\alpha}.
\end{align*}
By Lemma  \ref{lem:ch} 
\[\left|\text{deg}(x) -n\int_{B(x,2\eps)}w(x,y)\rho(y)\, dy  \right|\leq C \eps^{-2}(\dist(x,\Gamma)-2\eps)^{-\alpha},\] 
holds with probability at least $1-2\exp\left( -cn\eps^d \right)$. This implies
\[\text{deg}(x) \leq C \eps^{-2}(\dist(x,\Gamma)-2\eps)^{-\alpha}. \]
By another application of Lemma  \ref{lem:ch}, both
\[\left|\sum_{y\in X_n}w(x,y)(y_i -x_i) -n\int_{B(x,2\eps)}w(x,y)(y_i -x_i)\rho(y)\, dy  \right|\geq C\delta (\dist(x,\Gamma)-2\eps)^{-\alpha},\]
and
\begin{align*}
\left|\sum_{y\in X_n}w(x,y)(y_i -x_i)(y_j-x_j) -n\int_{B(x,2\eps)}w(x,y)(y_i -x_i)(y_j-x_j)\rho(y)\, dy  \right|&\\
&\hspace{-3cm}\geq C\delta \eps (\dist(x,\Gamma)-2\eps)^{-\alpha},
\end{align*}
occur with probability at most $2\exp\left(  -c\delta ^2n\eps^{d +2} \right)$ provided $0 < \delta \eps \leq 1$.
Thus, if $\eps \leq \delta\leq \eps^{-1}$ we have 
\begin{align}\label{eq:app}
\L\varphi (x)&=\frac{1}{2}\int_{B(0,2)}(\gamma(x) + \gamma(x+z\eps))\rho(x+z\eps)\boldeta(|z|)\left( \frac{1}{\eps}p\cdot z +\frac{1}{2}z\cdot A z \right)\, dz \\
& \hspace{2.5in}+O\left( \delta \beta_3 (\dist(x,\Gamma)-2\eps)^{-\alpha}\right)\notag
\end{align}
holds for all  $\varphi \in C^3(\R^d)$ with probability at least $1 -C\exp\left(  -c\delta ^2n\eps^{d +2} \right).$
Note that
\begin{align*}
\frac{1}{2}(\gamma(x) + \gamma(x+z\eps)) \rho(x+z\eps)&=\gamma(x)\rho(x) + \gamma(x)\nabla \rho(x)\cdot z \eps + \frac{1}{2}\rho(x)\nabla \gamma(x)\cdot z \eps\\
&\hspace{0.25in} + O(\eps^2\|\gamma\|_{C^2(B(x,2\eps))}+ \eps^3\|\gamma\|_{C^1(B(x,2\eps))}).
\end{align*}
We now have
\begin{align*}
\int_{B(0,2)}\gamma(x)\rho(x)\boldeta(|z|)\left( \frac{1}{\eps}p\cdot z +\frac{1}{2}z\cdot A z \right)\, dz &= \frac{1}{2}\gamma(x)\rho(x)\sum_{i,j=1}^d a_{ij}\int_{B(0,2)}\boldeta(|z|)z_iz_j\, dz\\
&=\frac{1}{2}\gamma(x)\rho(x)\sum_{i=1}^d a_{ii}\int_{B(0,2)}\boldeta(|z|)z_i^2\, dz\\
&=\frac{\sigma_\eta}{2}\gamma(x)\rho(x)\text{Trace}(A),
\end{align*}
\begin{align*}
&\int_{B(0,2)}\gamma(x)(\nabla \rho(x)\cdot z \eps)\boldeta(|z|)\left( \frac{1}{\eps}p\cdot z +\frac{1}{2}z\cdot A z \right)\, dz\\
&\hspace{2cm}= \gamma(x)\nabla \rho(x)\cdot\int_{B(0,2)}\boldeta(|z|)(p\cdot z)z\, dz + O(\eps \beta_2 \dist(x,\Gamma)^{-\alpha})\\
&\hspace{2cm}= \gamma(x)\nabla \rho(x)\cdot\sum_{i=1}^dp_i\int_{B(0,2)}\boldeta(|z|)z_iz\, dz + O(\eps \beta_2 \dist(x,\Gamma)^{-\alpha})\\
&\hspace{2cm}= \sigma_\eta\gamma(x)\nabla \rho(x)\cdot p + O(\eps \beta_2 \dist(x,\Gamma)^{-\alpha}),
\end{align*}
and
\begin{align*}
&\frac{1}{2}\int_{B(0,2)}\rho(x)(\nabla \gamma(x)\cdot z \eps)\boldeta(|z|)\left( \frac{1}{\eps}p\cdot z +\frac{1}{2}z\cdot A z \right)\, dz\\
&\hspace{4cm}= \frac{\sigma_\eta}{2}\rho(x)\nabla \gamma(x)\cdot p + O(\eps \beta_2 \|\gamma\|_{C^1(B(x,2\eps))}).
\end{align*}
Assembling these together with \eqref{eq:app} we have that
\begin{align*}
\L\varphi(x) &=\frac{\sigma_\eta}{2\rho(x)}\divv\left( \gamma \rho^2 \nabla \varphi \right) + O\Big(\eps\beta_1\|\gamma\|_{C^2(B(x,2\eps))} + \eps\beta_2\|\gamma\|_{C^1(B(x,2\eps))} \\
&\hspace{2in}+ \eps^2\beta_2 \|\gamma\|_{C^2(B(x,2\eps))}+ \delta\beta_3(\dist(x,\Gamma)-2\eps)^{-\alpha}\Big)
\end{align*} 
holds for all $\varphi\in C^3(B(x,2\eps))$ with probability at least $1 -C\exp\left(  -c\delta ^2n\eps^{d +2} \right).$ The proof is completed by computing
\[\|\gamma\|_{C^k(B(x,2\eps))} \leq C(\dist(x,\Gamma)-2\eps)^{-\alpha-k},\]
and applying a union bound over $x_1,\dots,x_n$.
\end{proof}

We now establish that the function $|x|^\beta$ for $0 < \beta < \alpha+2-d$ serves as a barrier (e.g., is a supersolution) on the graph with high probability.
\begin{lemma}[Barrier lemma]\label{lem:barrier}
Let $\alpha > d-2$ and fix any $0 < \beta < \alpha + 2-d$.  For $y\in \Gamma$ define $\varphi(x) = |x-y|^\beta$. Then the event that
\begin{equation}\label{eq:supersolution}
\L_{n,\eps,\infty}\varphi(x_i) \leq -c|x_i-y|^{-(\alpha+2-\beta)}
\end{equation}
for all $x_i$ with $C\eps < |x_i-y| \leq c$ occurs with probability at least $1 -C\exp\left(  -cn\eps^{d +4} + \log(n) \right)$.
\end{lemma}
\begin{proof}
Let us write $\L$ in place of $\L_{n,\eps,\infty}$ for simplicity. We use Theorem  \ref{thm:consistency} and Proposition \ref{prop:barrier}. Note in Theorem \ref{thm:consistency} that if we restrict $3\leq \eps |x_i-y|\leq R/4$ then 
\[M = \dist(x_i,\Gamma) - 2\eps = |x_i-y| - 2\eps \geq \frac{1}{3}|x_i-y|.\]
Also, for $\beta_k = \|\varphi\|_{C^k(B(x_i,2\eps))}$ we compute
\[\beta_k \leq C\beta|x_i-y|^{\beta-k}.\]
Hence, setting $\delta=\eps$ in Theorem \ref{thm:consistency} we obtain that
\begin{equation}\label{eq:consistencybarrier}
\left|\L\varphi(x_i) - \tfrac{1}{2}\sigma_\eta\Delta_\rho \varphi(x_i)\right|\leq C\eps |x_i-y|^{\beta-\alpha-3}( 1 + \eps|x_i-y|^{-1})
\end{equation}
holds for all $x_i$ with $3\eps\leq |x_i-y|\leq r$ with probability at least
\[1 -C\exp\left(  -cn\eps^{d +4} + \log(n) \right).\]
For the rest of the proof we restrict to the event that \eqref{eq:consistencybarrier} holds. 

Note that since $\beta-\alpha-3 < 0$ and $|x_i-y|\geq 3\eps$, it follows from \eqref{eq:consistencybarrier} that
\begin{equation}\label{eq:consistencybarrier2}
\L\varphi(x_i)\leq  \tfrac{1}{2}\sigma_\eta\Delta_\rho \varphi(x_i) + C\eps |x_i-y|^{\beta-\alpha-3}.
\end{equation}
Combining this with Proposition \ref{prop:barrier} we have
\[\L\varphi(x_i)\leq  -c|x_i-y|^{\beta-\alpha-2} + C\eps |x_i-y|^{\beta-\alpha-3} = -c|x_i-y|^{\beta-\alpha-2}(1 - C\eps|x_i-y|^{-1}),\]
provided $3\eps\leq |x_i-y|\leq c$. The proof is completed by restricting $|x_i-y|\geq 2C\eps$.
\end{proof}

The barrier lemma (Lemma \ref{lem:barrier}) establishes the barrier property away from the local neighborhood $B(x,C\eps)$. The singularity in the barrier (for $\alpha < d$) and the singularity in $\gamma$ prevent us from pushing the barrier lemma inside this local neighborhood.  Hence, the barrier can only be used to establish the following macroscopic continuity result.
\begin{proposition}[Macroscopic H\"older estimate]\label{prop:macro_holder}
Let $u\in L^2(\X_n)$ be the solution of \eqref{eq:optimality}, let $\alpha > d-2$, and fix any $0 < \beta < \alpha + 2-d$. For each $y\in \Gamma$ the event that
\begin{equation}\label{eq:macro_holder}
u(x_i) - u(y) \leq C|x_i-y|^\beta+\sup_{x\in \X_n\cap B(y,\delta_{\eps,\zeta})}(u(x)-u(y))
\end{equation}
holds for all $x_i\in \X_n$ occurs with probability at least $1 -C\exp\left(  -cn\eps^{d +4} + \log(n) \right)$, where $\delta_{\eps,\zeta}=\max\{C\eps,r_\zeta+2\eps\}$.
\end{proposition}
\begin{proof}
We note the graph is connected with probability at least $1-C\exp(-cn\eps^d + \log(n))$. The proof uses the barrier function 
\begin{equation}\label{eq:barrierfunction}
\varphi(x) = \sup_{B(y,\delta_{\eps,\zeta})}u + K|x-y|^{\beta}
\end{equation}
constructed in Lemma \ref{lem:barrier} for a sufficiently large $K$, and the maximum principle on a connected graph. By Lemma \ref{lem:barrier} we have
\begin{equation}\label{eq:super}
\L_{n,\eps,\zeta}\varphi(x_i) \leq -cK|x_i-y|^{-(\alpha+2-\beta)}
\end{equation}
for all $x_i$ with $\delta_{\eps,\zeta}\leq |x_i-y|\leq c$.  By the maximum principle we have
\[\min_{\Gamma}g \leq u \leq \max_{\Gamma}g.\]
Therefore, we can choose $K$ large enough so that $\varphi(x_i)> u(x_i)$ for $|x_i-y|\geq c$. We trivially have $u(x_i) < \varphi(x_i)$ for $|x_i-y|\leq \delta_{\eps,\zeta}$. Since $\L_{n,\eps,\zeta}(u - \varphi) \leq 0$ for $\delta_{\eps,\zeta}\leq|x_i-y|\leq c$, the maximum principle on a graph yields $u\leq \varphi$ on $\X_n$, which completes the proof.
\end{proof}

We now establish a local regularity result that allows us to fill in the gap within the ball $B(x,C\eps)$. The local result depends only on the variational structure of the problem, and does not use a barrier argument.
\begin{proposition}[Local H\"older estimate]\label{prop:micro_holder}
Let $u\in L^2(\X_n)$. For each $z\in \Gamma$ the event that
\begin{equation}\label{eq:micro_holder}
|u(x)-u(y)|^2 \leq \frac{Cn\eps^2}{\min\{\gamma_\zeta(x),\gamma_\zeta(y)\}}\GE_{n,\eps, \zeta}(u),
\end{equation}
holds for all $x,y\in \X_n\cap B(z,r)$ with $|x-y|\leq \eps$ occurs with probability at least $1 -C\exp\left(  -cn\eps^{d} + \log(n) \right)$.
\end{proposition}
\begin{proof}
Let $z\in \Gamma$, and fix $r>0$. Partition the cube $K:=\prod_{i=1}^d [z_i-r,z_i+r]$ into hypercubes $K_1,\dots,K_m$ of side length $h>0$, where $m = (2r/h)^d$. Let $Z_i$ denote the number of random variables falling in cube $K_i$. By Lemma \ref{lem:ch} we have
\begin{equation}\label{eq:conmeas}
\P(Z_i \leq \mathbb{E}[Z_i] - Cnh^d\lambda) \leq \exp\left( -cnh^d\lambda^2 \right)
\end{equation}
for any $0  <\lambda\leq 1$. Since $\mathbb{E}[Z_i] = nh^d$ we have
\begin{equation}\label{eq:bound}
\P\left(\min_{1\leq i \leq m}Z_i \leq \tfrac{1}{2}nh^d\right) \leq m\exp\left( -cnh^d \right).
\end{equation}
Let $x,y\in \X_n \cap B(z,r)$ such that $|x-y|\leq \eps$, and let $\bar{x} = (x+y)/2\in B(z,r)$. Let $K_i$ denote the cube to which $\bar{x}$ belongs. Then for all $w\in K_i$ we have $|\bar{x}-w|\leq \sqrt{d}h$. Therefore, if $\sqrt{d}h\leq \eps/2$ then
\[|x-w| \leq |x-\bar{x}| + |\bar{x} - w| \leq \frac{\eps}{2} + \sqrt{d}h \leq \eps,\]
and $|y-w|\leq \eps$ for all $w \in K_i$. It follows that
\[K_i \subset B(x,\eps)\cap B(y,\eps).\]
For the remainder of the proof, we set $h = \eps/(2\sqrt{d})$ and restrict ourselves to the event that
\begin{equation}\label{eq:event}
\min_{1\leq i \leq m}Z_i \geq cn\eps^d
\end{equation}

Let
\[K = \X_n\cap B(x,\eps)\cap B(y,\eps).\]
Note that for any $z\in K$ we have
\[\min\{|u(x)-u(z)|,|u(y)-u(z)|\}\geq \frac{1}{2}|u(x)-u(y)|.\]
Now we have
\begin{align*}
\GE_{n,\eps, \zeta}(u) &= \frac{1}{2n^2\eps^{2}}\sum_{x,y\in \X_n}(\gamma_\zeta(x) + \gamma_\zeta(y))\eta_\eps(x-y)|u(x)-u(y)|^2\\
&\geq \frac{1}{2n^2\eps^{2}}\left[\sum_{z\in \X_n}\gamma_\zeta(x)\eta_\eps(x-z)|u(x)-u(z)|^2 + \gamma_\zeta(y)\eta_\eps(y-z)|u(y)-u(z)|^2 \right]\\
&\geq \frac{c\min\{\gamma_\zeta(x),\gamma_\zeta(y)\}|K|}{n^2\eps^{d+2}}|u(x)-u(y)|^2.
\end{align*}
Since $K_i\cap \X_n\subset K$, we have $|K|\geq cn\eps^d$, and hence
\begin{equation}\label{eq:est}
|u(x)-u(y)|^2 \leq \frac{Cn\eps^2}{\min\{\gamma_\zeta(x),\gamma_\zeta(y)\}}\GE_{n,\eps, \zeta}(u),
\end{equation}
which completes the proof.
\end{proof}
We are now equipped to give the proof of Theorem \ref{thm:holder}.
\begin{proof}[Proof  of Theorem \ref{thm:holder}]
The proof combines the macroscopic H\"older estimate (Proposition \ref{prop:macro_holder}), and the local H\"older estimate (Proposition \ref{prop:micro_holder}), and is split into two steps.

1. We note that
\begin{equation}\label{eq:rzbound}
r_\zeta = \frac{r_0}{(\zeta-1)^{1/\alpha}} \leq C\eps,
\end{equation}
as $\zeta\geq 1+\eps^{-\alpha}$. 
By \eqref{eq:rzbound} and Theorem \ref{prop:macro_holder}, the event that
\begin{equation}\label{eq:macrobound}
|u(x) - u(z)| \leq C|x-z|^\beta + \sup_{x\in \X_n \cap B(z,C\eps)}|u(x)-u(y)|,
\end{equation}
holds for all $x\in \X_n$ occurs with probability at least $1-C\exp(-cn\eps^{d+4}+ \log(n))$.  

2. We note that with probability at least $1-C\exp\left( -cn\eps^d +\log(n) \right)$ we have $\GE_{n\eps,\zeta}(u) \leq C$ for a constant $C$. Therefore, by Proposition \ref{prop:micro_holder} we have that
\begin{equation}\label{eq:pathstep}
|u(x)-u(y)|^2\leq \frac{Cn\eps^2}{\min_{B(z,C\eps)}\gamma_\zeta}
\end{equation}
holds for all $x,y\in \X_n\cap B(z,C\eps)$ with $|x-y|\leq \eps$ with probability at least $1-C\exp\left( -cn\eps^d + \log(n) \right)$.
As in the proof of Proposition \ref{prop:micro_holder}, we partition the cube $K:=\prod_{i=1}^d [z_i-C\eps,z_i+C\eps]$ into hypercubes of side length $h < \eps$, and find that all cubes have at least one point from $\X_n$ with probability at least $1-C\exp(-cn\eps^d + \log(n))$. Thus, by traversing neighboring cubes, we can construct a path from $z\in \Gamma$ to any $x\in B(z,C\eps)$ consisting of at most a constant number of points from $\X_n\cap B(z,C\eps)$, with each step in the path smaller than $\eps$. Applying \eqref{eq:pathstep} along the path yields 
\[|u(x)-u(z)|^2\leq \frac{Cn\eps^2}{\min\{(C\eps)^{-\alpha},\zeta\}}\]
for all $x\in \X_n\cap B(z,C\eps)$ with probability at least $1-C\exp\left( -cn\eps^d + \log(n) \right)$. Since $\zeta\geq \eps^{-\alpha}$ we deduce
\[\sup_{x\in \X_n \cap B(z,C\eps)}|u(x)-u(z)|^2\leq Cn\eps^{2+\alpha},\]
which completes the proof.
\end{proof}

\section{Numerical experiments}
\label{sec:numerics}

\begin{figure}
\centering
\subfloat[$\alpha=0$]{\includegraphics[width=0.33\textwidth,clip = true, trim = 30 30 30 30]{GL_n1e5_eps2.png}}
\subfloat[$\alpha=0.5, \zeta = 50n\eps^2$]{\includegraphics[width=0.33\textwidth,clip = true, trim = 30 30 30 30]{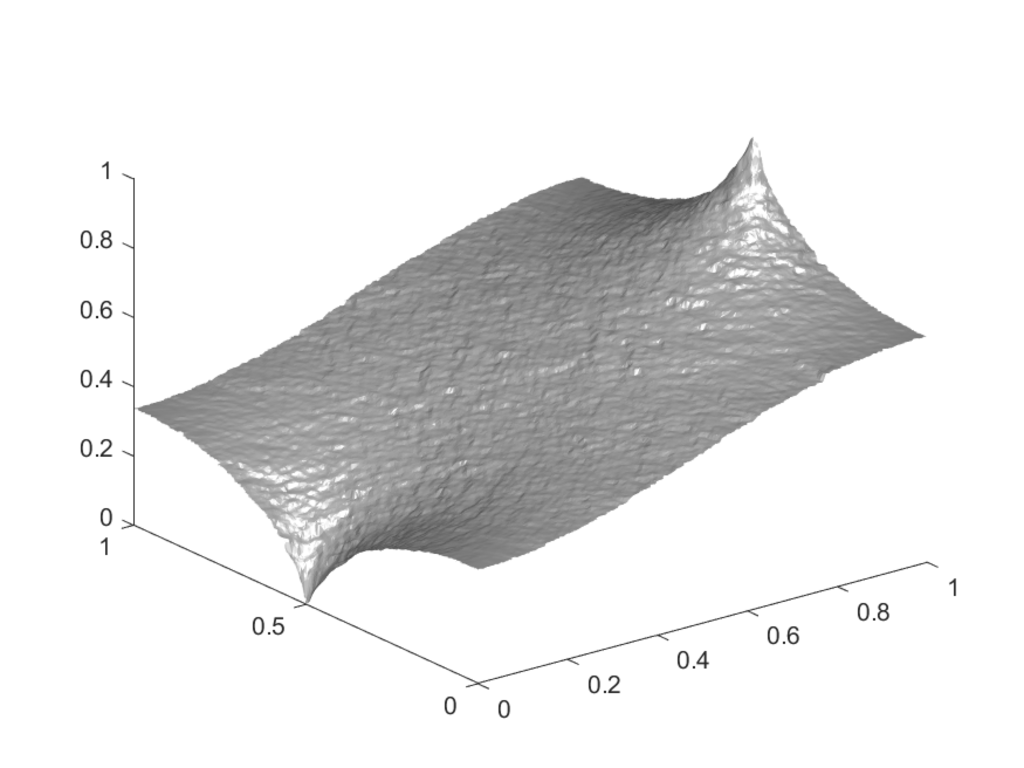}}
\subfloat[$\alpha=1, \zeta = 50n\eps^2$]{\includegraphics[width=0.33\textwidth,clip = true, trim = 30 30 30 30]{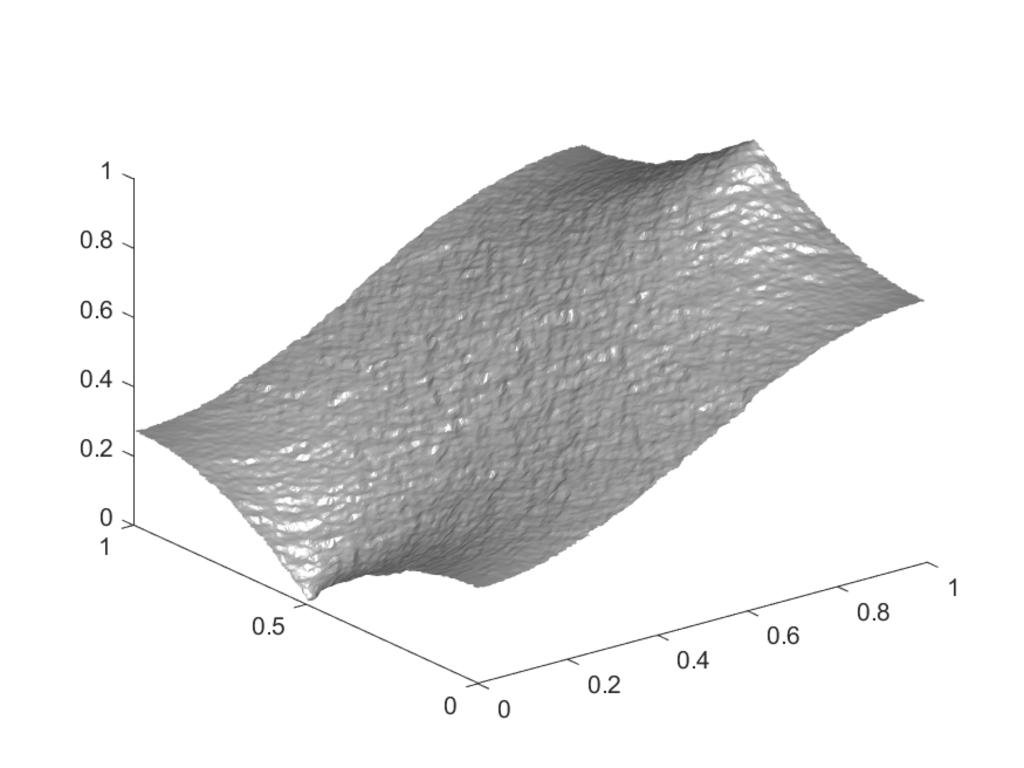}}\\
\subfloat[$\alpha=2,\zeta = 50n\eps^2$]{\includegraphics[width=0.33\textwidth,clip = true, trim = 30 30 30 30]{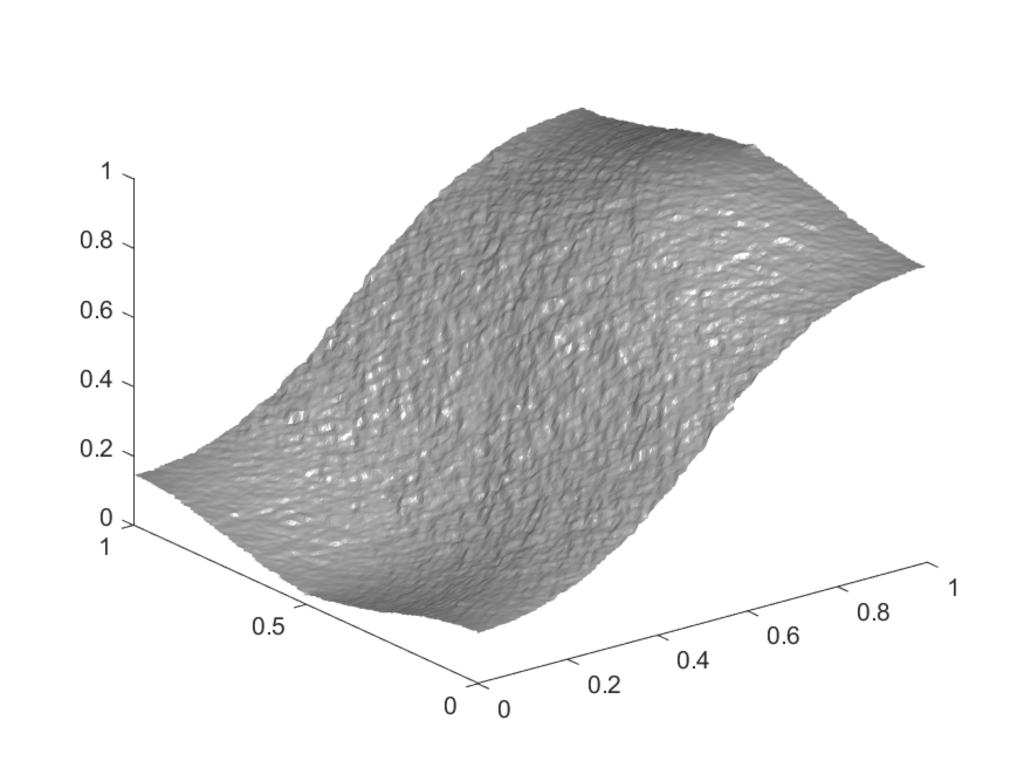}}
\subfloat[$\alpha=5, \zeta = 10^3 n\eps^2$]{\includegraphics[width=0.33\textwidth,clip = true, trim = 30 30 30 30]{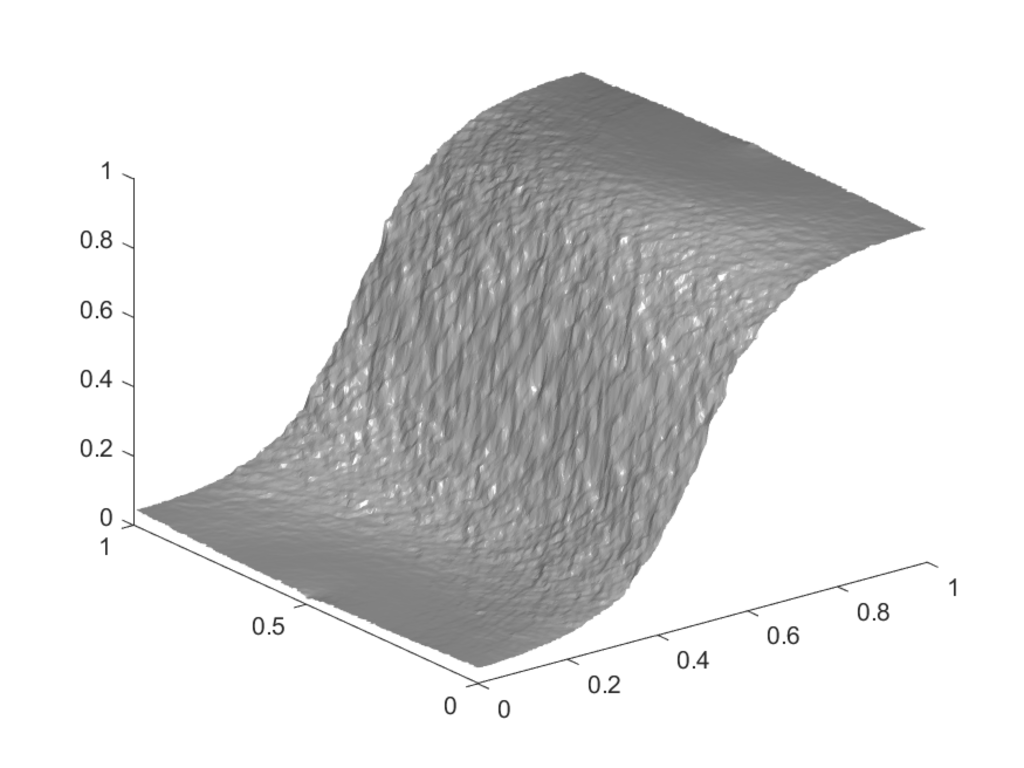}}
\subfloat[$\alpha=10, \zeta = 10^5n\eps^2$]{\includegraphics[width=0.33\textwidth,clip = true, trim = 30 30 30 30]{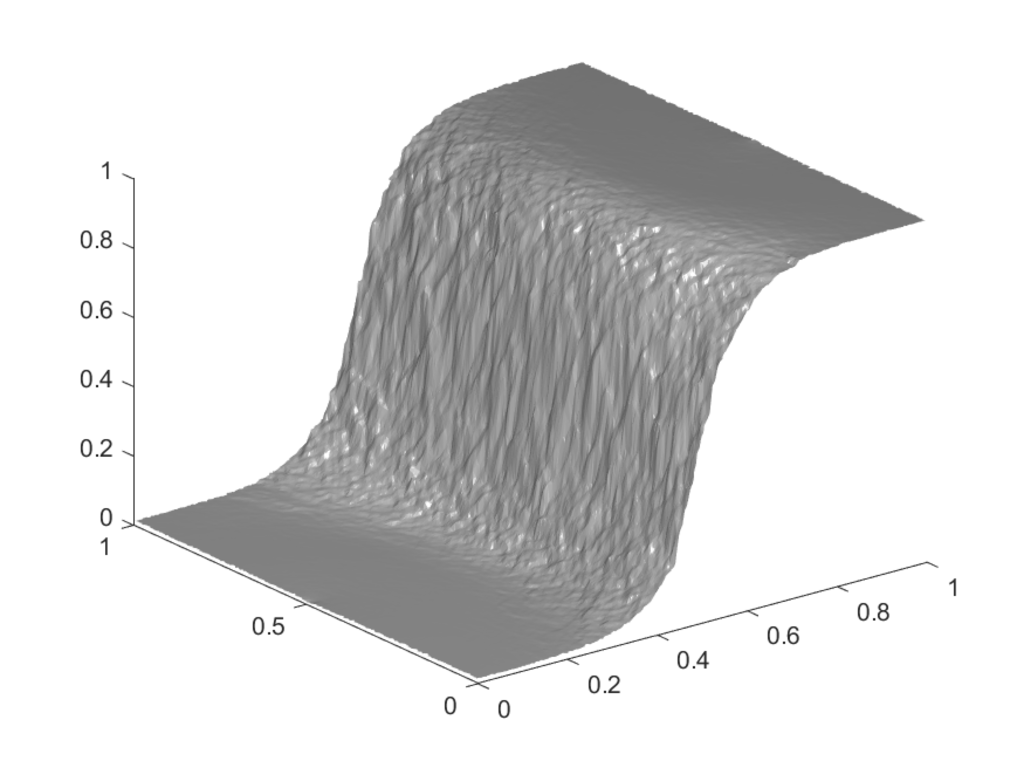}}
\caption{Comparison of (A) the standard graph Laplacian, and (B)--(F) our weighted graph Laplacian with various values of $\alpha$ and $\zeta$. The graph is a random geometric graph (described in the text) on $[0,1]^2$, and the labels are $g(0,0.5)=0$ and $g(1,0.5)=1$. }
\label{fig:alpha} 
\end{figure}

\begin{figure}
\centering
\subfloat[Graph Laplacian]{\includegraphics[width=0.33\textwidth,clip = true, trim = 30 30 30 30]{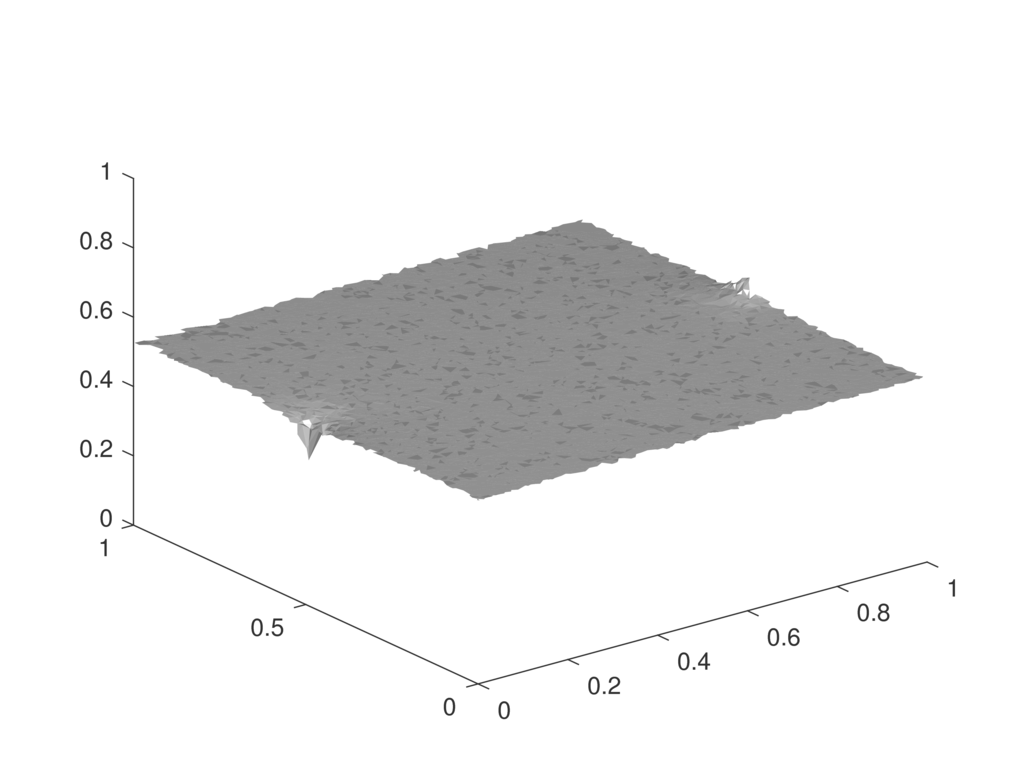}}
\subfloat[Weighted Lap. ~\cite{shi2017weighted}]{\includegraphics[width=0.33\textwidth,clip = true, trim = 30 30 30 30]{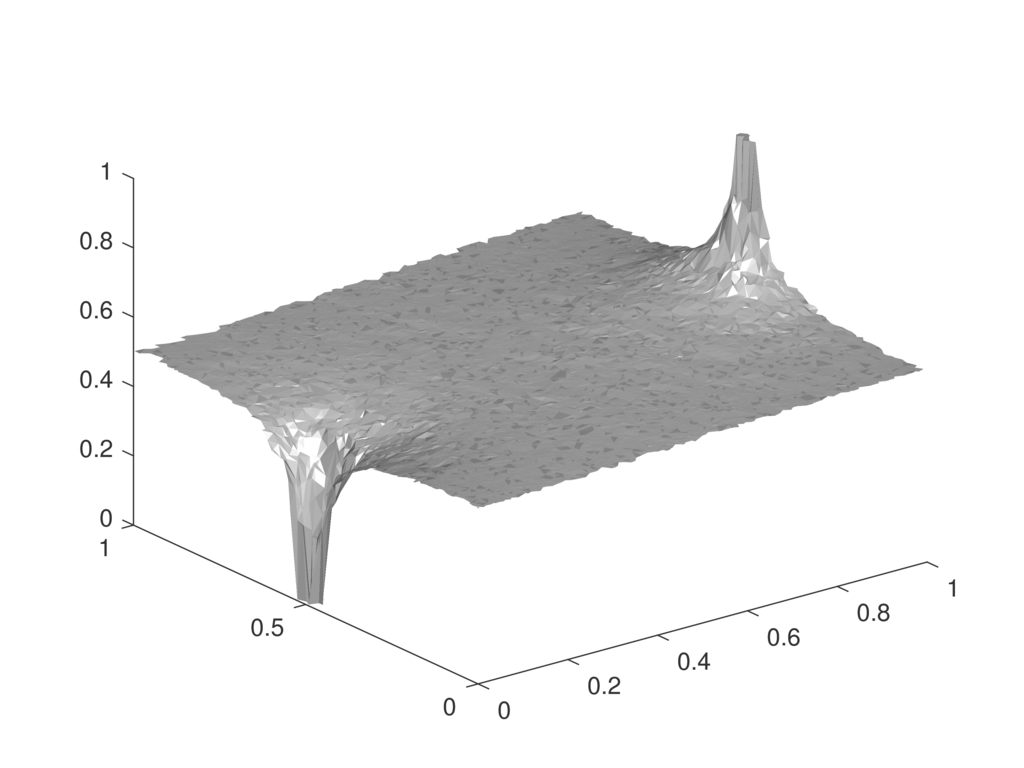}\label{fig:shi}}
\subfloat[PW Laplacian ($\alpha=2$)]{\includegraphics[width=0.33\textwidth,clip = true, trim = 30 30 30 30]{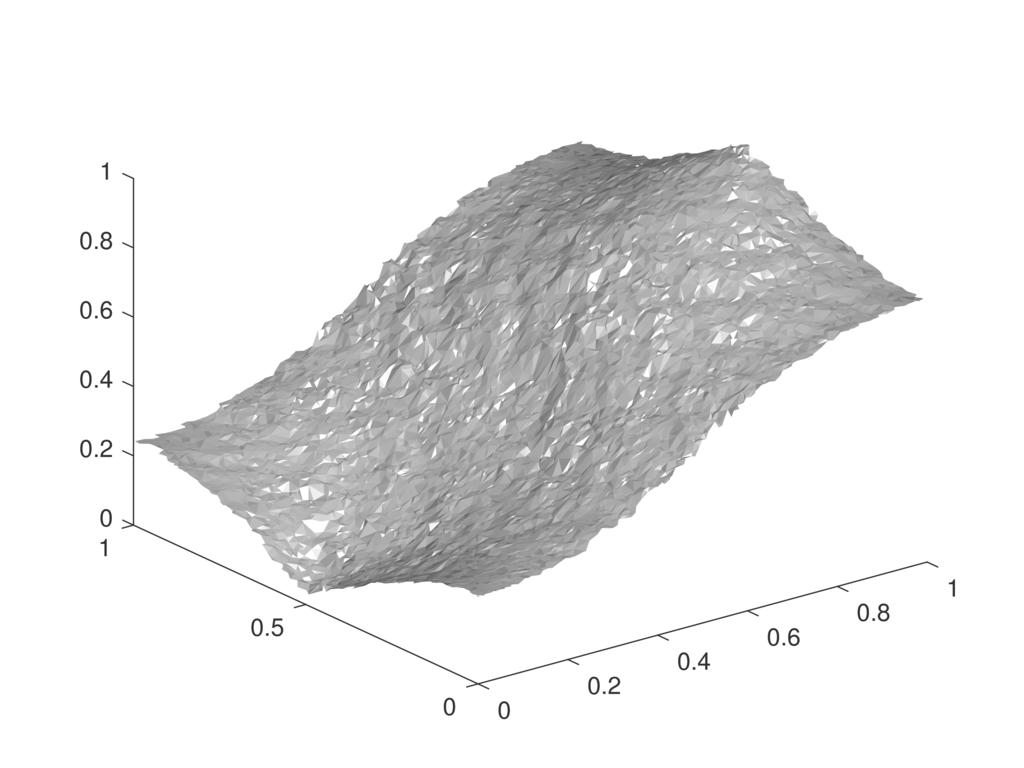}\label{fig:ours}}
\caption{Comparison of (a) the standard graph Laplacian, (b) the nonlocal graph Laplacian \cite{shi2017weighted}, and (c) our properly-weighted Laplacian. The graph is a random geometric graph (described in the text) on $[0,1]^3$, and the surfaces plotted are a slice of the learned function $u$ near $x_3=0.5$. The labels are $g(0,0.5,0.5) = 0$ and $g(1,0.5,0.5) = 1$. }
\label{fig:shiosherzhu} 
\end{figure}

We present the results of several numerical experiments illustrating the properly-weighted Laplacian and comparing it with the nonlocal \cite{shi2017weighted} and standard graph Laplacian on real and synthetic data. All experiments were performed in Matlab and use Matlab backslash to solve the graph Laplacian system. We mention there are indirect solvers that may be faster in certain applications, such as preconditioned conjugate gradient \cite{greenbaum1997iterative}, algebraic multigrid \cite{greenbaum1997iterative,ruge1987algebraic,brandt2011multigrid}, or more recent fast Laplacian solvers \cite{spielman2004nearly,kyng2015algorithms}. Thus, the CPU times reported below have the potential to be improved substantially.

\subsubsection{Comparison of the profiles obtained}
First, we perform an experiment with two labels on the box $[0,1]^d$ to illustrate our method and the differences with the nonlocal graph Laplacian \cite{shi2017weighted}. The graph is a sequence of $n$ \emph{i.i.d.}~random variables uniformly distributed on the unit box $[0,1]^d$ in $\R^d$, and two labels are given $g(0,0.5,\dots,0.5)=0$ and $g(1,0.5,\dots,0.5)=1$. We set $\eps = 2/n^{1/d}$ and $r_0=1$. The weights follow a Gaussian distribution with $\sigma = \eps/2$. In Figure \ref{fig:alpha} we show plots of the triangulated surface representing the learned function $u$ on the graph for various values of $\alpha$ and $\zeta$. Here, $n=10^5$ and $d=2$, and each simulation took approximately 1.5 seconds of CPU time. We notice that as $\alpha$ is increased, the learned functions are smoother in a vicinity of each label. The case of $\alpha=0$ corresponds to the standard graph Laplacian, and returns an approximately constant label $u=0.5$, which illustrates the degeneracy of the standard Laplacian with few labels. We note that as $\alpha$ is increased, we must increase $\zeta$ as well (recall \eqref{eq:gn}), otherwise the ball $B(x,r_\zeta)$ on which $\gamma$ is truncated to value $\zeta$ becomes very large, and the method reduces to the standard graph Laplacian. This simply illustrates that the implicit rate in the condition $\zeta \gg n\eps^2$ in Theorem \ref{thm:conv} depends on $\alpha$.

In Figure \ref{fig:shiosherzhu}, we use the same model, but with $n=2\times 10^5$ points in dimension $d=3$, and set $\zeta = 50n\eps^2$. For visualization, we show the learned function restricted to a neighborhood of the slice $x_3=0.5$. Figure \ref{fig:shi} illustrates the degeneracy of the nonlocal graph Laplacian \cite{shi2017weighted}, which returns a nearly constant label function. In contrast, our method, show in in Figure \ref{fig:ours}, smoothly interpolates between the two labels. Each simulation in Figure \ref{fig:shiosherzhu} took approximately 25 seconds to compute.

\begin{figure}
\centering
\subfloat[Graph Laplacian]{\includegraphics[width=0.33\textwidth,clip = true, trim = 30 0 0 0]{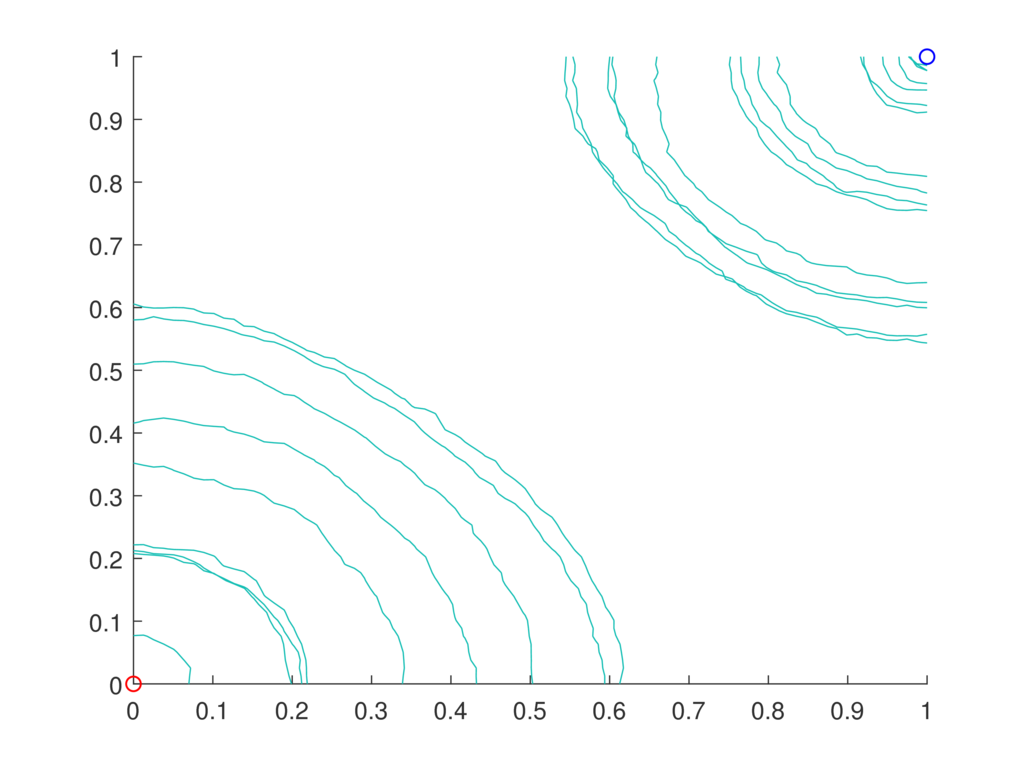}}
\subfloat[Weighted Lap.~\cite{shi2017weighted}]{\includegraphics[width=0.33\textwidth,clip = true, trim = 30 0 0 0]{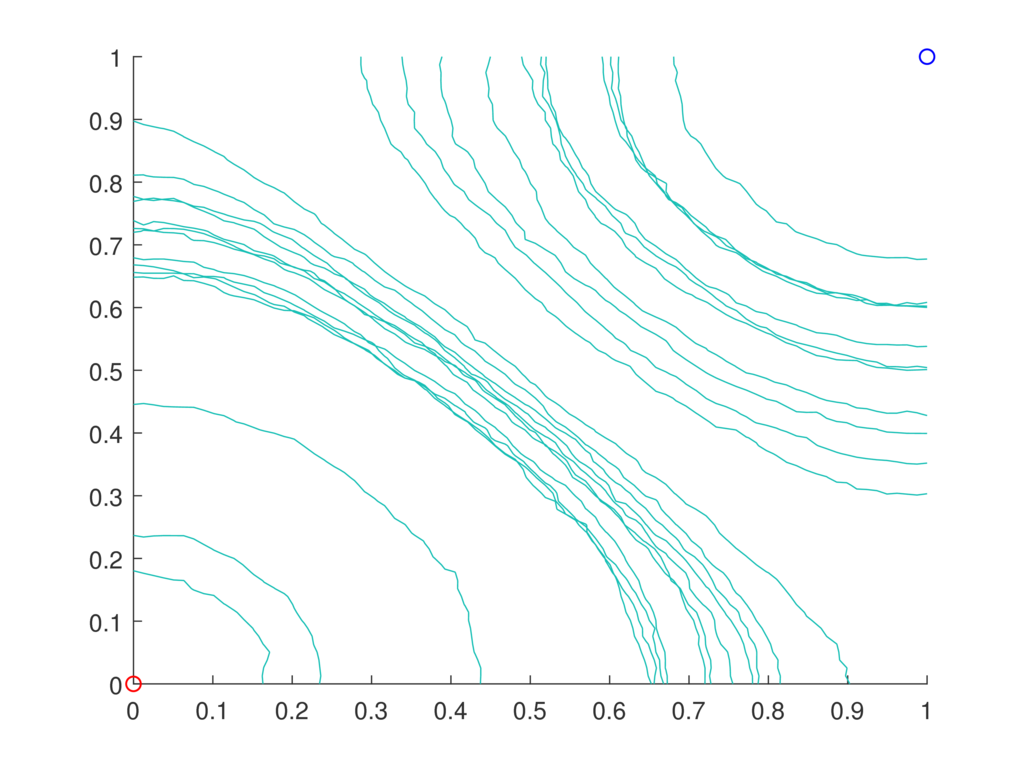}}
\subfloat[PW Laplacian ($\alpha=5$)]{\includegraphics[width=0.33\textwidth,clip = true, trim = 30 0 0 0]{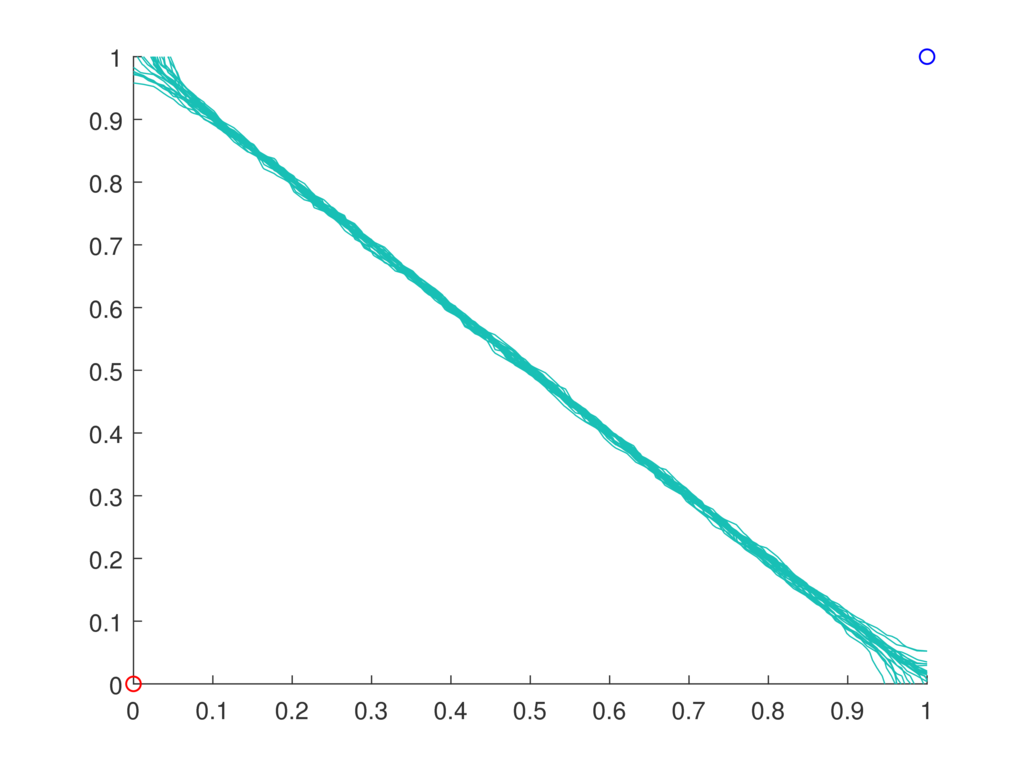}}
\caption{Comparison of decision boundaries plotted over 25 random trials for a synthetic classification problem for (A) the standard graph Laplacian, (B) the nonlocal graph Laplacian \cite{shi2017weighted}, and (C) our properly-weighted graph Laplacian. Two labels are given, $g(0,0)=0$ (red point) and $g(1,1)=1$ (blue point), and the graph is a uniform random geometric graph (described in the text). }
\label{fig:unif} 
\end{figure}

\subsubsection{Comparison of decision boundaries in 2D}
We now give a synthetic classification example. The graph consists of $n=10^5$ \emph{i.i.d} uniform random variables on $[0,1]^2$, and the weights are chosen to be Gaussian with $\sigma=\eps/2$. We chose $\alpha=5$, $\zeta=10^6 n\eps^2$, $r_0=1$, and $\eps = 3/\sqrt{n}$. Two labels, $g(0,0) = 0$ and $g(1,1) = 1$ are provided. Figure \ref{fig:unif} shows the decision boundaries (i.e., the level-set $\{u = 0.5\}$) over 25 trials for the standard graph Laplacian, the nonlocal Laplacian, and our method. Each trial took roughly 1.5 seconds to compute. We see that the nonlocal and standard Laplacian are highly sensitive to small variations in the graph, giving a wide variety of results over the 25 trials. This is a reflection of the degeneracy, or ill-posedness, in the small label regime, and suggests the methods are very sensitive to perturbations in the data. In contrast, our method very consistently divides the square along the diagonal.


\begin{figure}
\centering
\subfloat[Graph Laplacian]{\includegraphics[width=0.33\textwidth,clip = true, trim = 60 30 90 30]{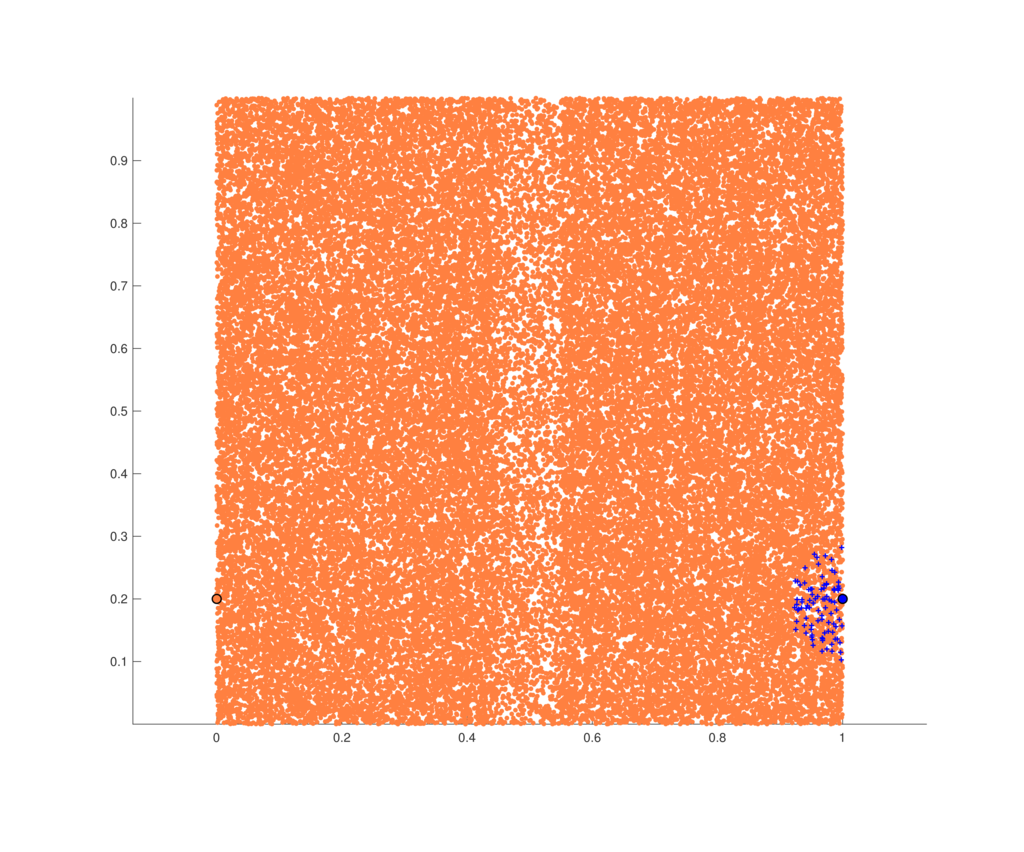}}
\subfloat[Weighted  Lap. \cite{shi2017weighted}]{\includegraphics[width=0.33\textwidth,clip = true, trim = 60 75 90 0]{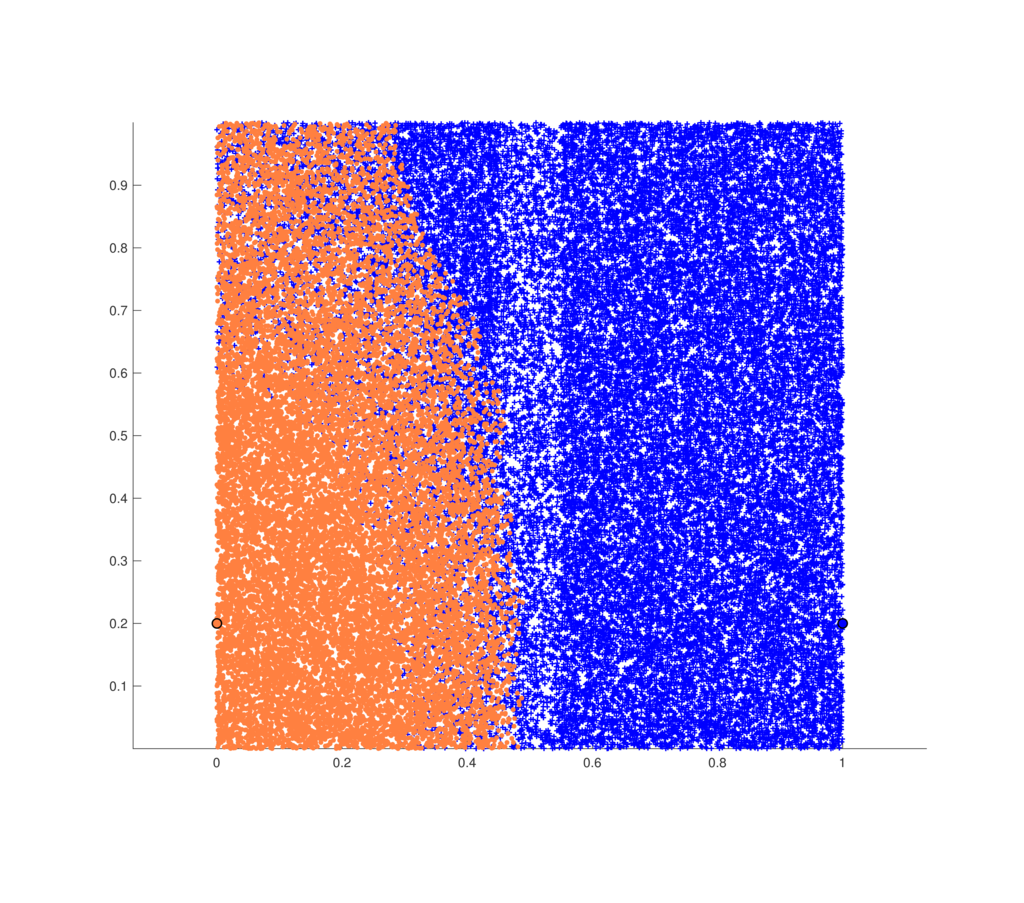}}
\subfloat[Our method ($\alpha=5$)]{\includegraphics[width=0.33\textwidth,clip = true, trim = 60 45 100 0]{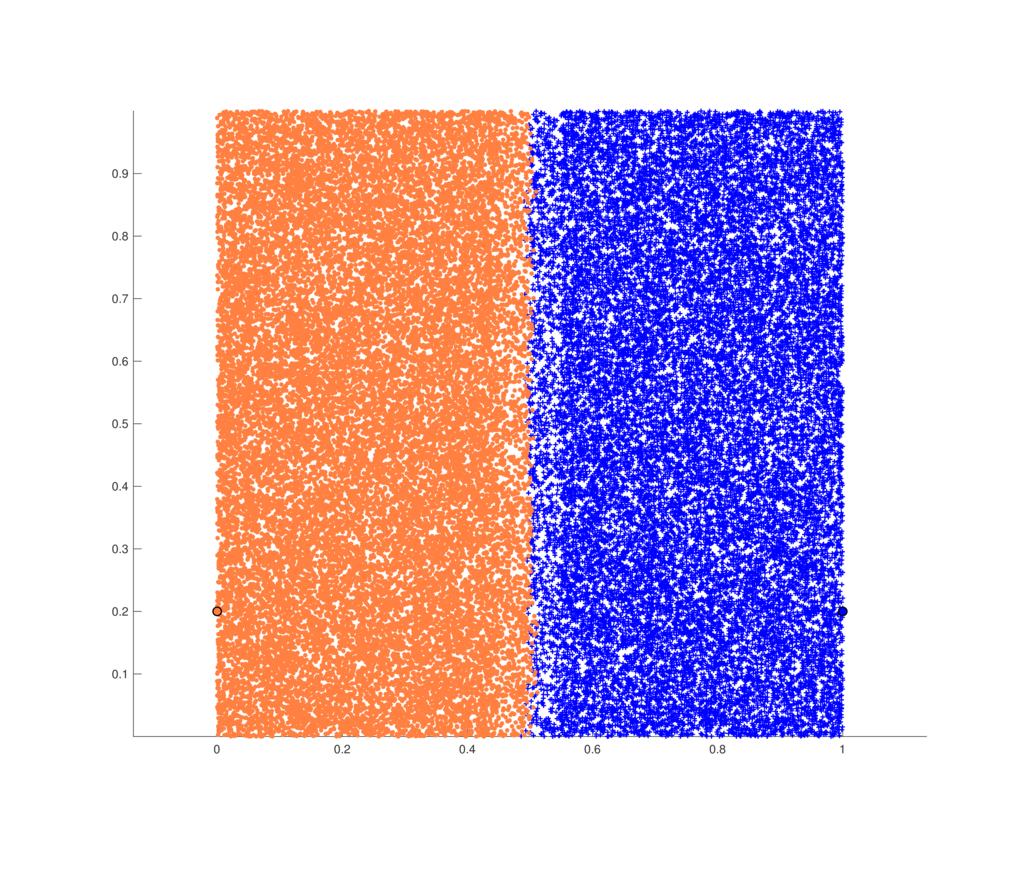}}
\caption{Comparison of results for a synthetic classification problem for (A) the standard graph Laplacian, (B) the nonlocal graph Laplacian \cite{shi2017weighted}, and (C) our weighted graph Laplacian. The domain is $[0,1]^3$ and the density is $1$ except for the strip $[0.45,0.55]\times [0,1]\times [0,1]$ where it is 0.6. The given labeled points are $g(0,0.2,0.2)=0$ and $g(1,0.2,0.2)=1$. There are $n=50,000$ points in the domain. Connectivity distance for the graph construction is $3/n^{\frac13}$ and for our method $\alpha=5$.}
\label{fig:strip} 
\end{figure}

\subsubsection{Comparison of classes obtained in 3D}

We consider samples of the measure supported on domain is $[0,1]^3$ and with density $1$ except for the strip $[0.45,0.55]\times [0,1]\times [0,1]$ where the density is 0.6. We considered 20 runs with $n=50,000$ points in the domain. The given labeled points are $g(0,0.2,0.2)=0$ and $g(1,0.2,0.2)=1$. Due to the symmetry, the correct decision boundary is the plane $x_1=0.5$. We used a connectivity distance for the graph construction of $\eps=3/n^{\frac13}$, which yielded a typical vertex degree of about $116$. We consider Gaussian weights  with $\sigma=\eps/2$. We chose $\alpha=5$, $\zeta=10^6 n\eps^2$, $r_0=1$ for our method.  A typical result for one run is illustrated on Figure \ref{fig:strip}. 
The standard graph Laplacian produced very unstable results with the average of  $49.8\%$ misclassified points. The nonlocal Laplacian of \cite{shi2017weighted} was also rather unstable with sometimes almost perfect decision boundary and sometimes large sections of misclassified points. On average it misclassified $11\%$ of points. Our method was stable and in all experiments identified the correct boundary, with average classification error of $0.25\%$. We observed similar outcomes for a variety of sets of parameters.

\begin{figure}
\centering
\includegraphics[width=\textwidth,clip=true,trim = 50 180 50 150]{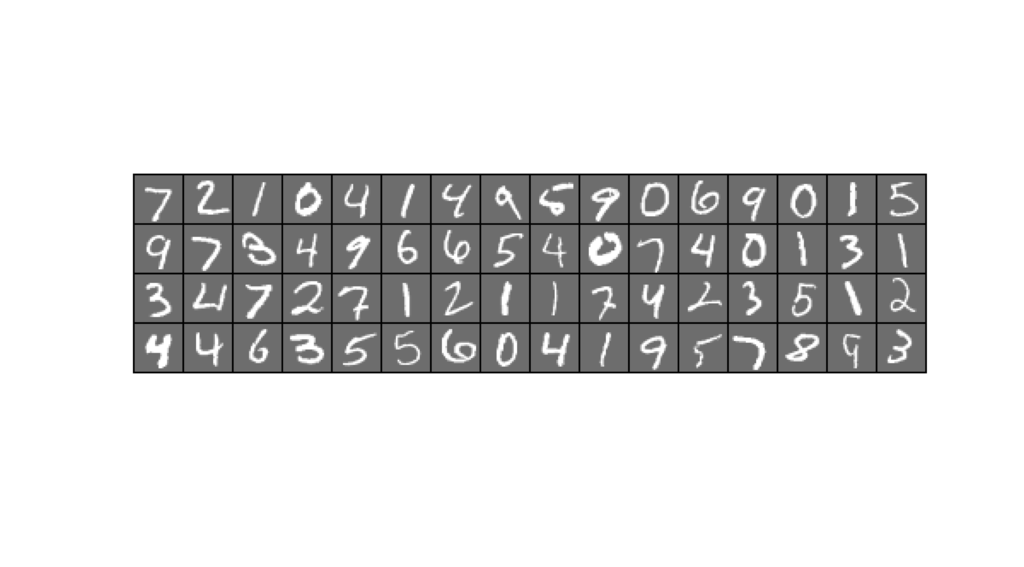}
\caption{Example of some of the handwritten digits from the MNIST dataset \cite{lecun1998gradient}.}
\label{fig:MNIST}
\end{figure}

\subsubsection{Comparison on the MNIST dataset}
Our last experiment considers classification of handwritten digits from the MNIST dataset,  which consists of 70,000 grayscale 28x28 pixel images of handwritten digits $0$--$9$ \cite{lecun1998gradient}. Figure \ref{fig:MNIST} shows examples of some of the images in the MNIST dataset. MNIST is estimated to have intrinsic dimension between $d=12$ and $d=14$~\cite{hein2005intrinsic,costa2006determining}, which suggests a larger value for $\alpha$ is appropriate. We used all $70,000$ MNIST images to construct the graph. Our construction is the same as in \cite{shi2017weighted}; we connect each data point to its nearest $50$ neighbors (in Euclidean distance), and assign Gaussian weights taking  $\sigma$ to be the distance to the $20^{\rm th}$ nearest neighbor. We symmetrize the graph by replacing the weight matrix $W$ with $\tfrac{1}{2}(W^T + W)$, which is done automatically by the variational formulation (recall \eqref{eq:stillsym}). We then take $L$ randomly chosen images from each class (digit) as labels, where $L=1,3,5,7,10$, and provide the true labels for these digits. The semi-supervised learning algorithm performs $10$ binary classifications, for each digit versus the rest, which generates functions $u^0,u^1,u^2,\dots,u^9$ on the graph. The label for each image $x$ in the dataset is chosen as the index $i$ for which $u^i(x)$ is maximal. The algorithm is standard in semi-supervised learning, and identical to the one used in \cite{shi2017weighted}.

\begin{table}[!t]
\centering
\begin{tabular}{|c|c|c|c|c|c|c|c|c|c|c|}
 \hline
\textbf{\# Labels} &\multicolumn{2}{c|}{\textbf{10}}&\multicolumn{2}{c|}{\textbf{30}} &\multicolumn{2}{c|}{\textbf{50}}&\multicolumn{2}{c|}{\textbf{70}}&\multicolumn{2}{c|}{\textbf{100}}\\
\hline
\textbf{Method} & Mean & Std & Mean & Std & Mean & Std & Mean & Std & Mean & Std\\
\hline
Graph Laplacian                       &14.2\% &6.3   &24.3\%   &11.9    &53\%    &10.9  &68\% &6.4  & 76.1\% & 7.6   \\
\hline
Weighted Lap.~\cite{shi2017weighted}  &67.9\% &8.7   &84.8\%    &2.7   &88.8\%      &1.1    & 89.6\% &1.3 &90.9\% &1.1  \\
\hline
PW-Laplacian                           &68\%  &8.6  &84.9\%     &2.7   &88.8\%       &1.1    & 89.6\% &1.3 &90.9\% &1.1  \\
\hline
\end{tabular}
\vspace*{5pt}
\caption{Accuracy for classification of MNIST handwritten digits with $10,30,50,70$ and $100$ labels via the standard graph Laplacian, the nonlocal weighted Laplacian~\cite{shi2017weighted}, and our properly-weighted Laplacian. The results are averaged over 10 trials, and the mean and standard deviation of accuracy are reported. }
\label{tab:MNIST}
\end{table}

For each value of $L\in \{1,3,5,7,10\}$, we ran the experiment described above 10 times, choosing the labels randomly (in the same way for each algorithm) every time. Each of the $500$ trials took approximately 15 minutes to compute in Matlab. The mean and standard deviation of accuracy are shown in Table \ref{tab:MNIST}. Our method performs very similarly to the nonlocal Laplacian \cite{shi2017weighted}, and both significantly outperform the standard graph Laplacian. We ran our method for $\alpha=2,5,10$, producing nearly identical results in all cases. We used $\zeta = 10^7$ and $r_0=0.1$. We found the results for our method were largely insensitive to many of the parameters in our algorithm; the accuracy begins to decrease when $\alpha<1$ and when $r_0 > 1$. We note that the accuracy scores reported in Table \ref{tab:MNIST} are much higher than those reported in \cite{shi2017weighted}; we believe this is because the authors in \cite{shi2017weighted} subsampled MNIST to 16,000 images. This observation speaks favorably to the semi-supervised paradigm that learning can be improved by access to additional unlabeled data. We note that the accuracy scores for our method and the nonlocal weighted Laplacian \cite{shi2017weighted} are identical (to one significant digit) for $30,50,70$, and $100$ labels. Most data points in MNIST are relatively far from their nearest neighbors, and so our nonlocal weights have less effect, compared to the low dimensional examples presented above. For this reason, the weight matrix for our method is very similar to the nonlocal Laplacian \cite{shi2017weighted}. We expect to see more of a difference in applications to larger datasets. For example, it would be interesting (and challenging) to apply these techniques to a dataset like ImageNet~\cite{deng2009imagenet}, which consists of over $14$ million natural images belonging to over 20,000 categories.

\appendix

\section{Background Material \label{sec:Back}}
Here we recall some of the notions our work depends on and establish an auxiliary technical result.


\subsection{\texorpdfstring{$\Gamma$}{Gamma}--Convergence} \label{subsec:Back:Gamma}

$\Gamma$-convergence was introduced by De Giorgi in 1970's to study limits of variational problems. 
We refer to~\cite{braides02,dalmaso93} for comprehensive  introduction to $\Gamma$-convergence.
We now recall the notion of $\Gamma$-convergence is in a random setting.

\begin{definition}[$\Gamma$-convergence]
\label{def:Back:Gamma:GamCon}
Let $(Z,d)$ be a metric space, $L^0(Z;\R\cup\{\pm \infty\})$ be the set of measurable functions from $Z$ to $\R\cup\{\pm \infty\}$, and $(\X,\P)$ be a probability space.
The function $\X\ni\omega\mapsto\E_n^{(\omega)} \in L^0(Z;\R\cup\{\pm \infty\})$ is a random variable.
We say $\E_n^{(\omega)}$ \textit{$\Gamma$-converge almost surely} on the domain $Z$ to $\E_\infty :Z\to \R\cup\{\pm\infty\}$ with respect to $d$, and write $\E_\infty = \Glim_{n \to \infty}\E_n^{(\omega)}$, if there exists a set $\X^\prime\subset \X$ with $\P(\X^\prime) = 1$, such that for all $\omega\in \X^\prime$ and all $f\in Z$:
\begin{itemize}
\item[(i)] (liminf inequality) for every sequence $\{u_n\}_{n=1,\dots}$ in $Z$ converging to $f$
\[\E_\infty(f) \leq \liminf_{n\to \infty}\E_n^{(\omega)}(u_n), \te{ and } \]
\item[(ii)] (recovery sequence) there exists a sequence $\{u_n\}_{n=1,2,\dots}$ in $Z$ converging to $f$ such that
\[\E_\infty(f) \geq \limsup_{n\to \infty}\E_n^{(\omega)}(u_n). \]
\end{itemize}
\end{definition}

For simplicity we suppress the dependence of $\omega$ in writing our functionals.
The almost sure nature of the convergence in our claims in ensured by 
considering the set of realizations of $\{x_i\}_{i=1,\dots}$ such that the conclusions of Theorem~\ref{thm:TransBound} hold (which they do almost surely).
 
 \red
An important result concerning $\Gamma$-convergence is that any subsequential limit of the sequence of minimizers of $\E_n$ is a minimizer of the limiting functional $\E_\infty$. So to show that the minimizers of $\E_n$ converge at least along a subsequence to a minimizer of $\E_\infty$it suffices to establish the precompactness of the set of minimizers.  \nc
We make this precise in the theorem below. 
Its proof can be found in~\cite[Theorem 1.21]{braides02} or~\cite[Theorem 7.23]{dalmaso93}.

\begin{theorem}[Convergence of Minimizers]
\label{thm:Back:Gamma:Conmin}
Let $(Z,d)$ be a metric space and $\E_n: Z\to [0,\infty]$ be a sequence of functionals.
Let $u_n$ be a minimizing sequence for $\E_n$.
If the set $\{u_n\}_{n=1,2,\dots}$ is precompact and $\E_\infty = \Glim_n\E_n$ where $\E_\infty:Z\to[0,\infty]$ is not identically $\infty$ then
\[ \min_Z\E_\infty = \lim_{n\to \infty} \inf_Z\E_n. \]
Furthermore any cluster point of $\{u_n\}_{n=1,2,\dots}$ is a minimizer of  $E_\infty$.
\end{theorem}
The theorem is also true if we replace minimizers with approximate minimizers.
\medskip

We note that $\Gamma$-convergence is defined for functionals on a common metric space.
Section~\ref{sec:TLp} overviews the metric space we use to analyze the asymptotics of our semi-supervised learning models, in particular it allows us to go from discrete to continuum.


\subsection{Optimal Transportation and Approximation of Measures} \label{sec:ot}

Here we recall the notion of optimal transportation between measures and the metric it introduces.
Comprehensive treatment of the topic can be found in books of Villani \cite{villani09} and Santambrogio \cite{santambrogio}.

Given a bounded, open set $\Omega\subset \R^d$, and probability measures $\mu$ and $\nu$ in $\mathcal{P}(\overline \Omega)$ we define the set $\Pi(\mu,\nu)$ of transportation plans, or couplings, between $\mu$ and $\nu$ to be the set of probability measures on the product space $\pi \in \mathcal{P}(\overline \Omega \times \overline \Omega)$ whose first marginal is $\mu$ and second marginal is $\nu$.
We then define the $p$-optimal transportation distance (a.k.a. $p$-Wasserstein distance)
by
\[ d_p(\mu, \nu) = 
\begin{cases}
\displaystyle{\inf_{\pi\in\Pi(\mu,\nu) } \left( \int_{\Omega\times \Omega} |x-y|^p \, d \pi(x,y)  \right)^\frac{1}{p}}  \quad & \te{if } 1 \leq p < \infty \\
\displaystyle{ \inf_{\pi\in\Pi(\mu,\nu)}  \pi\text{-}\esssup_{(x,y)} |x-y|  } & \te{if } p=\infty.
\end{cases} \]
If $\mu$ is absolutely continuous with respect to the Lebesgue measure on $\Omega$, then the distance can be rewritten using transportation maps, $T:  \Omega \to \Omega$, instead of transportation plans,
\[ d_p(\mu, \nu) = 
\begin{cases}
\displaystyle{\inf_{T_{\#}\mu=\nu } \left(  \int_\Omega |x-T(x)|^p \, d \mu(x)    \right)^\frac{1}{p}}  \quad & \te{if } 1 \leq p < \infty \\
\displaystyle{ \inf_{T_{\#}\mu=\nu}  \mu\text{-}\esssup_{x} |x-T(x)|  } & \te{if } p=\infty.
\end{cases} \]
where $T_{\#}\mu=\nu$ means that the push forward of the measure $\mu$ by $T$ is the measure $\nu$, namely that $T$ is Borel measurable and such that for all $U\subset\overline{\Omega}$, open, $\mu(T^{-1}(U)) = \nu(U)$. 

When $p< \infty$ the metric $d_p$ metrizes the $\text{weak}^*$ convergence of measures.

Optimal transportation plays an important role in comparing the discrete and continuum objects we study.
In particular, we use sharp estimates on the $\infty$-optimal transportation 
distance between a measure and the empirical measure of its sample.
In the form below, for $d \geq 2$, they were established in~\cite{GTS15trans}, which extended the related results in~\cite{AKT, LeightonShor, ShorYukich, Talagrand}.

\begin{theorem}
\label{thm:TransBound}
For $d \geq 2$, let $\Omega\subset\R^d$ be open, connected and bounded with Lipschitz boundary.
Let $\mu$ be a probability measure on $\Omega$ with density (with respect to Lebesgue) $\rho$ which is bounded above and below by positive constants.
Let $x_1,x_2,\dots$ be a sequence of independent random variables with distribution $\mu$ and let $\mu_n$ be the empirical measure.
Then, there exists constants $C \geq c > 0$ such that almost surely there exists a sequence of transportation maps $\{T_n\}_{n=1}^\infty$ from $\mu$ to $\mu_n$ with the property
\[ c \leq  \liminf_{n\to \infty} \frac{\|T_n-\Id\|_{L^\infty(\Omega)}}{\ell_n}  \leq \limsup_{n\to \infty} \frac{\|T_n-\Id\|_{L^\infty(\Omega)}}{\ell_n} \leq C \]
where
\begin{equation} \label{eq:elln_transport}
\ell_n =
\begin{cases}
  \frac{(\ln n)^{\frac{3}{4}}}{\sqrt{n}} \; & \text{if } d=2 \\ 
  \frac{(\ln n)^{\frac{1}{d}}}{n^{\frac{1}{d}}} & \text{if } d\geq 3.
\end{cases}
\end{equation}
\end{theorem}


\subsection{The \texorpdfstring{$TL^p$}{TLp} Space} \label{sec:TLp}

The discrete functionals we consider (e.g. $\E_{n, \eps_n, \zeta_n}$)  are defined for functions $u_n : \X_n \to \R$, while the limit functional $\E$ acts on functions $f:\Omega \to \R$, where $\Omega$ is an open set.
We can view $u_n$ as elements of $L^p(\mu_n)$ where $\mu_n$ is the empirical measure of the sample $\mu_n = \frac{1}{n} \sum_{i=1}^n \delta_{x_i}$.
Likewise $f \in L^p(\mu)$ where $\mu$ is the measure with density $\rho$ from which the data points are sampled.
In order to compare $f$ and $u_n$ in a way that is consistent with the $L^p$ topology we use the $TL^p$ space that was introduced in~\cite{GTS16}, where it was used to study the continuum limit of the graph total variation. 
Subsequent development of the $TL^p$ space has been carried out in~\cite{GTS18spectral, Thorpe2017jmiv}.

To compare the functions $u_n$ and $f$ above we need to take into account their domains, or more precisely to account for  $\mu$ and $\mu_n$. 
For that purpose the space of configurations is defined to be 
\[ TL^p(\Omega) = \left\{(\mu,f) \, : \, \mu \in \mathcal{P}(\overline \Omega), f \in L^p(\mu) \right\}. \]
The metric on the space is
\[ d_{TL^p}^p((\mu,f),(\nu,g)) = \inf \left\{ \int_{\Omega\times \Omega} |x-y|^p + |f(x) - g(y)|^p \, d \pi(x,y) \, : \, \pi\in\Pi(\mu,\nu) \right\} \]
where $\Pi(\mu,\nu) $ the set of transportation plans defined in Section~\ref{sec:ot}.
We note that the minimizing $\pi$ exists and that $TL^p$ space is a metric space,~\cite{GTS16}.

As shown in~\cite{GTS16},  when $\mu$ is absolutely continuous with respect to the  Lebesgue measure on $\Omega$, then the distance can be rewritten using transportation maps $T$, instead of transportation plans,
\[ d_{TL^p}^p((\mu,f),(\nu,g)) = \inf \left\{ \int_\Omega |x-T(x)|^p + |f(x) - g(T(x))|^p \, d \mu(x) \, : \, T_{\#} \mu = \nu \right\} \]
where the push forward of the measure $T_{\#} \mu$ is defined in Section~\ref{sec:ot}.
This formula provides an interpretation of the distance in our setting.
Namely, to compare functions $u_n: \X_n\to \R$ we define a mapping $T_n:\Omega \to \X_n$ and compare the functions $\tilde{f}_n = u_n\circ T_n$ and $f$ in $L^p(\mu)$, while also accounting for the transport, namely the $|x - T_n(x)|^p$ term.

We remark that the $TL^p(\overline \Omega)$ space is not complete, and that its completion was discussed in~\cite{GTS16}.
In the setting of this paper, since the corresponding measure is clear from context, we often say that $u_n$ converges in $TL^p$ to $f$ as a short way to say that $(\mu_n, u_n)$ converges in $TL^p$ to $(\mu,f)$.

\subsection{Local estimates for weighted Laplacian} \label{sec:lewl}

\begin{lemma} \label{lem:lewl}
There exists $C>0$ such that for each $ u \in H^1(B(0,1))$ there exists $ v \in H^1(B(0,1))$ such that 
\begin{align*}
 v|_{B(0,\frac12)} & \equiv \frac{1}{|B(0,\frac12)|} \int_{B(0,\frac12)}  u(x) dx \\
 v|_{\partial B(0,1)} & =  u|_{\partial B(0,1)} \\
 \|\nabla   v\|_{L^2(B(0,1))}  & \leq C \|\nabla  u\|_{L^2(B(0,1))}
\end{align*}
where the value on the boundary is considered in sense of the $L^2(\partial B(0,1))$ trace. 
\end{lemma}
\begin{proof}
 Let 
\[ \overline u =  \frac{1}{|B(0,\frac12)|} \int_{B(0,\frac12)}  u(x) dx. \]
Let $\phi \in C^\infty([0, 1] , [0,1])$ be such that $\phi(r) = 1$ for all $r \in [0, \frac12]$ and $\phi(1)=0$. Let $M= \max_{r \in [0,1]}  |\phi'(r)|$. Let
\[ v(x) = \phi(|x|) \overline u + (1-\phi(|x|)) u(x). \]
By Poincar\'e inequality stated in Theorem 13.27 of \cite{leoni2} there exists $C_1 >0$, independent of $u$, 
\[ \int_{B(0,1)} |u(x) - \overline u|^2 dx  \leq C_1 \int_{B(0,1)} |\nabla u(x)|^2 dx. \]
Using the Poincar\'e inequality we obtain, for $C =2+2MC_1$, 
\[ 
\int_{B(0,1)} |\nabla v|^2 dx  \leq 2 \int_{B(0,1)} |\nabla u|^2  + M | u - \overline u|^2 dx 
 \leq C \int_{B(0,1)} |\nabla u|^2 dx. 
\]
\end{proof}


\begin{thebibliography}{10}

\bibitem{AKT}
M.~Ajtai, J.~Koml{\'o}s, and G.~Tusn{\'a}dy.
\newblock On optimal matchings.
\newblock {\em Combinatorica}, 4(4):259--264, 1984.

\bibitem{alamgir2011phase}
M.~Alamgir and U.~V. Luxburg.
\newblock Phase transition in the family of p-resistances.
\newblock In {\em Advances in Neural Information Processing Systems}, pages
  379--387, 2011.

\bibitem{ando2007learning}
R.~K. Ando and T.~Zhang.
\newblock Learning on graph with laplacian regularization.
\newblock {\em Advances in Neural Information Processing Systems}, 19:25, 2007.

\bibitem{bardi2008optimal}
M.~Bardi and I.~Capuzzo-Dolcetta.
\newblock {\em Optimal control and viscosity solutions of
  Hamilton-Jacobi-Bellman equations}.
\newblock Springer Science \& Business Media, 2008.

\bibitem{BelNiy03SSL}
M.~Belkin and P.~Niyogi.
\newblock Using manifold structure for partially labeled classification.
\newblock In {\em Advances in Neural Information Processing Systems (NIPS)},
  pages 953--960, 2003.

\bibitem{BLSZ17}
A.~Bertozzi, X.~Luo, A.~Stuart, and K.~Zygalakis.
\newblock Uncertainty quantification in graph-based classification of high
  dimensional data.
\newblock {\em SIAM/ASA Journal on Uncertainty Quantification}, 6(2):568--595,
  2018.

\bibitem{braides02}
A.~Braides.
\newblock {\em {$\Gamma$}-convergence for beginners}, volume~22 of {\em Oxford
  Lecture Series in Mathematics and its Applications}.
\newblock Oxford University Press, Oxford, 2002.

\bibitem{brandt2011multigrid}
A.~Brandt and O.~E. Livne.
\newblock {\em Multigrid techniques: 1984 guide with applications to fluid
  dynamics}, volume~67.
\newblock SIAM, 2011.

\bibitem{bridle2013p}
N.~Bridle and X.~Zhu.
\newblock p-voltages: Laplacian regularization for semi-supervised learning on
  high-dimensional data.
\newblock In {\em Eleventh Workshop on Mining and Learning with Graphs
  (MLG2013)}, 2013.

\bibitem{calderLip2017}
J.~Calder.
\newblock {Consistency of Lipschitz learning with infinite unlabeled and finite
  labeled data}.
\newblock {\em arXiv:1710.10364}, 2017.

\bibitem{calder2017game}
J.~Calder.
\newblock The game theoretic p-{L}aplacian and semi-supervised learning with
  few labels.
\newblock {\em arXiv:1711.10144}, 2017.

\bibitem{calder2018limit}
J.~Calder and C.~K. Smart.
\newblock The limit shape of convex hull peeling.
\newblock {\em arXiv preprint arXiv:1805.08278}, 2018.

\bibitem{ssl}
O.~Chapelle, B.~Scholkopf, and A.~Zien.
\newblock {\em Semi-supervised learning}.
\newblock MIT, 2006.

\bibitem{costa2006determining}
J.~A. Costa and A.~O. Hero.
\newblock Determining intrinsic dimension and entropy of high-dimensional shape
  spaces.
\newblock In {\em Statistics and Analysis of Shapes}, pages 231--252. Springer,
  2006.

\bibitem{dalmaso93}
G.~Dal~Maso.
\newblock {\em An introduction to {$\Gamma$}-convergence}, volume~8 of {\em
  Progress in Nonlinear Differential Equations and their Applications}.
\newblock Birkh\"{a}user Boston, Inc., Boston, MA, 1993.

\bibitem{deng2009imagenet}
J.~Deng, W.~Dong, R.~Socher, L.-J. Li, K.~Li, and L.~Fei-Fei.
\newblock Imagenet: A large-scale hierarchical image database.
\newblock In {\em Computer Vision and Pattern Recognition, 2009. CVPR 2009.
  IEEE Conference on}, pages 248--255. IEEE, 2009.

\bibitem{DSST18}
M.~M. Dunlop, D.~Slep{\v{c}}ev, A.~M. Stuart, and M.~Thorpe.
\newblock Large data and zero noise limits of graph-based semi-supervised
  learning algorithms.
\newblock {\em to appear in Appl. Comput. Harmon. Anal, arXiv preprint
  arXiv:1805.09450}, 2018.

\bibitem{el2016asymptotic}
A.~El~Alaoui, X.~Cheng, A.~Ramdas, M.~J. Wainwright, and M.~I. Jordan.
\newblock Asymptotic behavior of lp-based {L}aplacian regularization in
  semi-supervised learning.
\newblock In {\em 29th Annual Conference on Learning Theory}, pages 879--906,
  2016.

\bibitem{GGHS}
N.~Garc{\'i}a~Trillos, M.~Gerlach, M.~Hein, and D.~Slep{\v{c}}ev.
\newblock Error estimates for spectral convergence of the graph {L}aplacian on
  random geometric graphs towards the {L}aplace-{B}eltrami operator.
\newblock {\em arXiv preprint arXiv:1801.10108}, 2018.

\bibitem{GTS15trans}
N.~Garc\'{i}a~Trillos and D.~Slep\v{c}ev.
\newblock On the rate of convergence of empirical measures in
  {$\infty$}-transportation distance.
\newblock {\em Canad. J. Math.}, 67(6):1358--1383, 2015.

\bibitem{GTS16}
N.~Garc\'{i}a~Trillos and D.~Slep\v{c}ev.
\newblock Continuum limit of total variation on point clouds.
\newblock {\em Archive for Rational Mechanics and Analysis}, 220(1):193--241,
  2016.

\bibitem{GTS18spectral}
N.~Garc\'{i}a~Trillos and D.~Slep\v{c}ev.
\newblock A variational approach to the consistency of spectral clustering.
\newblock {\em Appl. Comput. Harmon. Anal.}, 45(2):239--281, 2018.

\bibitem{GTS17jmlr}
N.~Garc\'{i}a~Trillos, D.~Slep\v{c}ev, J.~von Brecht, T.~Laurent, and
  X.~Bresson.
\newblock Consistency of {C}heeger and ratio graph cuts.
\newblock {\em J. Mach. Learn. Res.}, 17:Paper No. 181, 46, 2016.

\bibitem{gilbarg2015elliptic}
D.~Gilbarg and N.~S. Trudinger.
\newblock {\em Elliptic partial differential equations of second order}.
\newblock springer, 2015.

\bibitem{greenbaum1997iterative}
A.~Greenbaum.
\newblock {\em Iterative methods for solving linear systems}, volume~17.
\newblock Siam, 1997.

\bibitem{he2004manifold}
J.~He, M.~Li, H.-J. Zhang, H.~Tong, and C.~Zhang.
\newblock Manifold-ranking based image retrieval.
\newblock In {\em Proceedings of the 12th Annual ACM International Conference
  on Multimedia}, pages 9--16. ACM, 2004.

\bibitem{he2006generalized}
J.~He, M.~Li, H.-J. Zhang, H.~Tong, and C.~Zhang.
\newblock Generalized manifold-ranking-based image retrieval.
\newblock {\em IEEE Transactions on image processing}, 15(10):3170--3177, 2006.

\bibitem{hein2005intrinsic}
M.~Hein and J.-Y. Audibert.
\newblock Intrinsic dimensionality estimation of submanifolds in {Rd}.
\newblock In {\em Proceedings of the 22nd International Conference on Machine
  learning}, pages 289--296. ACM, 2005.

\bibitem{kyng2015algorithms}
R.~Kyng, A.~Rao, S.~Sachdeva, and D.~A. Spielman.
\newblock Algorithms for lipschitz learning on graphs.
\newblock In {\em Proceedings of The 28th Conference on Learning Theory}, pages
  1190--1223, 2015.

\bibitem{lecun1998gradient}
Y.~LeCun, L.~Bottou, Y.~Bengio, and P.~Haffner.
\newblock Gradient-based learning applied to document recognition.
\newblock {\em Proceedings of the IEEE}, 86(11):2278--2324, 1998.

\bibitem{LeightonShor}
T.~Leighton and P.~Shor.
\newblock Tight bounds for minimax grid matching with applications to the
  average case analysis of algorithms.
\newblock {\em Combinatorica}, 9(2):161--187, 1989.

\bibitem{leoni2}
G.~Leoni.
\newblock {\em A first course in Sobolev spaces}, volume 181.
\newblock American Mathematical Soc., 2017.

\bibitem{luxburg2004distance}
U.~v. Luxburg and O.~Bousquet.
\newblock Distance-based classification with lipschitz functions.
\newblock {\em Journal of Machine Learning Research}, 5(Jun):669--695, 2004.

\bibitem{nadler2009semi}
B.~Nadler, N.~Srebro, and X.~Zhou.
\newblock Semi-supervised learning with the graph {L}aplacian: The limit of
  infinite unlabelled data.
\newblock In {\em Neural Information Processing Systems (NIPS)}, 2009.

\bibitem{rudin2006real}
W.~Rudin.
\newblock {\em Real and complex analysis}.
\newblock Tata McGraw-Hill Education, 2006.

\bibitem{ruge1987algebraic}
J.~W. Ruge and K.~St{\"u}ben.
\newblock Algebraic multigrid.
\newblock In {\em Multigrid methods}, pages 73--130. SIAM, 1987.

\bibitem{santambrogio}
F.~Santambrogio.
\newblock {\em Optimal transport for applied mathematicians}, volume~87 of {\em
  Progress in Nonlinear Differential Equations and their Applications}.
\newblock Birkh\"{a}user/Springer, Cham, 2015.
\newblock Calculus of variations, PDEs, and modeling.

\bibitem{shi2017weighted}
Z.~Shi, S.~Osher, and W.~Zhu.
\newblock Weighted nonlocal laplacian on interpolation from sparse data.
\newblock {\em Journal of Scientific Computing}, 73(2-3):1164--1177, 2017.

\bibitem{ShorYukich}
P.~W. Shor and J.~E. Yukich.
\newblock Minimax grid matching and empirical measures.
\newblock {\em Ann. Probab.}, 19(3):1338--1348, 1991.

\bibitem{Singer06}
A.~Singer.
\newblock From graph to manifold {L}aplacian: The convergence rate.
\newblock {\em Applied and Computational Harmonic Analysis}, 21(1):128--134,
  2006.

\bibitem{SleTho17plap}
D.~Slep{\v{c}}ev and M.~Thorpe.
\newblock Analysis of p-{L}aplacian regularization in semi-supervised learning.
\newblock {\em to appear in SIAM J. Math. Anal., arXiv:1707.06213}, 2017.

\bibitem{spielman2004nearly}
D.~A. Spielman and S.-H. Teng.
\newblock Nearly-linear time algorithms for graph partitioning, graph
  sparsification, and solving linear systems.
\newblock In {\em Proceedings of the thirty-sixth annual ACM symposium on
  Theory of computing}, pages 81--90. ACM, 2004.

\bibitem{Talagrand}
M.~Talagrand.
\newblock {\em Upper and lower bounds of stochastic processes}, volume~60 of
  {\em Modern Surveys in Mathematics}.
\newblock Springer-Verlag, Berlin Heidelberg, 2014.

\bibitem{Thorpe2017jmiv}
M.~Thorpe, S.~Park, S.~Kolouri, G.~K. Rohde, and D.~Slep\v{c}ev.
\newblock A transportation {$L^p$} distance for signal analysis.
\newblock {\em J. Math. Imaging Vision}, 59(2):187--210, 2017.

\bibitem{villani09}
C.~Villani.
\newblock {\em Optimal transport}, volume 338 of {\em Grundlehren der
  Mathematischen Wissenschaften [Fundamental Principles of Mathematical
  Sciences]}.
\newblock Springer-Verlag, Berlin, 2009.
\newblock Old and new.

\bibitem{wang2013multi}
Y.~Wang, M.~A. Cheema, X.~Lin, and Q.~Zhang.
\newblock Multi-manifold ranking: Using multiple features for better image
  retrieval.
\newblock In {\em Pacific-Asia Conference on Knowledge Discovery and Data
  Mining}, pages 449--460. Springer, 2013.

\bibitem{xu2011efficient}
B.~Xu, J.~Bu, C.~Chen, D.~Cai, X.~He, W.~Liu, and J.~Luo.
\newblock Efficient manifold ranking for image retrieval.
\newblock In {\em Proceedings of the 34th international ACM SIGIR Conference on
  Research and Development in Information Retrieval}, pages 525--534. ACM,
  2011.

\bibitem{yang2013saliency}
C.~Yang, L.~Zhang, H.~Lu, X.~Ruan, and M.-H. Yang.
\newblock Saliency detection via graph-based manifold ranking.
\newblock In {\em Proceedings of the IEEE conference on Computer Vision and
  Pattern Recognition}, pages 3166--3173, 2013.

\bibitem{zhou2004learning}
D.~Zhou, O.~Bousquet, T.~N. Lal, J.~Weston, and B.~Sch{\"o}lkopf.
\newblock Learning with local and global consistency.
\newblock {\em Advances in Neural Information Processing Systems},
  16(16):321--328, 2004.

\bibitem{zhou2005learning}
D.~Zhou, J.~Huang, and B.~Sch{\"o}lkopf.
\newblock Learning from labeled and unlabeled data on a directed graph.
\newblock In {\em Proceedings of the 22nd International Conference on Machine
  Learning}, pages 1036--1043. ACM, 2005.

\bibitem{ZhoSch05}
D.~Zhou and B.~Sch\"{o}lkopf.
\newblock Regularization on discrete spaces.
\newblock In {\em Proceedings of the 27th DAGM Conference on Pattern
  Recognition}, PR'05, pages 361--368, Berlin, Heidelberg, 2005.
  Springer-Verlag.

\bibitem{zhou2004ranking}
D.~Zhou, J.~Weston, A.~Gretton, O.~Bousquet, and B.~Sch{\"o}lkopf.
\newblock Ranking on data manifolds.
\newblock {\em Advances in Neural Information Processing Systems}, 16:169--176,
  2004.

\bibitem{zhou11}
X.~Zhou and M.~Belkin.
\newblock Semi-supervised learning by higher order regularization.
\newblock In {\em Proceedings of the 14th International Conference on
  Artificial Intelligence and Statistics}, volume~15 of {\em Proceedings of
  Machine Learning Research}, pages 892--900, 2011.

\bibitem{zhou2011iterated}
X.~Zhou, M.~Belkin, and N.~Srebro.
\newblock An iterated graph laplacian approach for ranking on manifolds.
\newblock In {\em Proceedings of the 17th ACM SIGKDD International Conference
  on Knowledge Discovery and Data Mining}, pages 877--885. ACM, 2011.

\bibitem{zhu2003semi}
X.~Zhu, Z.~Ghahramani, J.~Lafferty, et~al.
\newblock Semi-supervised learning using {G}aussian fields and harmonic
  functions.
\newblock In {\em International Conference on Machine Learning}, volume~3,
  pages 912--919, 2003.

\end{thebibliography}
\end{document}